\documentclass{article}

\usepackage[T1]{fontenc}

\usepackage{graphicx}
\graphicspath{ {images/} }
\usepackage{caption}
\usepackage{subcaption}
\usepackage{float}
\usepackage{geometry}
\usepackage{fancyhdr}
\usepackage{enumerate}

\usepackage{cite}
\usepackage{tikz}
\usetikzlibrary{arrows}

\usepackage{mathtools}
\usepackage{amsmath}
\usepackage{amsthm}
\usepackage{amssymb}
\usepackage{graphicx}
\usepackage{enumitem}

\usepackage{url}

\newcommand{\summ}{\sum\limits}

\newtheorem{theorem}{Theorem}[section]

\newtheorem{corollary}[theorem]{Corollary}

\newtheorem{lemma}[theorem]{Lemma}
\newtheorem{problem}[theorem]{Problem}
\newtheorem{proposition}[theorem]{Proposition}
\newtheorem{question}[theorem]{Question}

\newtheorem{observation}[theorem]{Observation}
\title{\textbf{Boundary and Hearing Independent Broadcasts in \\
Graphs and Trees}}
\author{J. I. Hoepner, G. MacGillivray, and C. M. Mynhardt
\\
Department of Mathematics and Statistics
\\
University of Victoria
\\
PO BOX 1700 STN CSC \\
Victoria, B.C. \\
Canada V8W 2Y2 \\
\small{julesihoepner@gmail.com, kieka@uvic.ca, gmacgill@uvic.ca}}
\date{May 31, 2023}

\begin{document}
\maketitle

\begin{abstract}

A \textit{broadcast} on a connected graph $G$ with vertex set $V(G)$ is a function $f:V(G)\rightarrow \{0, 1, ..., \text{diam}(G)\}$ such that $f(v)\leq e(v)$ (the eccentricity of $v$) for all $v\in V$. A vertex $v$ is said to be \textit{broadcasting} if $f(v)>0$, with the set of all such vertices denoted $V_f^+$. A vertex $u$ \textit{hears} $f$ from $v\in V_f^+$ if $d_G(u, v)\leq f(v)$. The broadcast $f$ is \textit{hearing independent} if no broadcasting vertex hears another. If, in addition, any vertex $u$ that hears $f$ from multiple broadcasting vertices satisfies $f(v)\leq d_G(u, v)$ for all $v\in V_f^+$, the broadcast is said to be \textit{boundary independent.}

The \textit{cost} of $f$ is $\sigma(f)=\sum_{v\in V(G)}f(v)$.
The minimum cost of a maximal boundary independent broadcast on $G$, called the \textit{lower bn-independence number}, is denoted $i_{bn}(G)$. The \textit{lower h-independence number} $i_h(G)$ is defined analogously for hearing independent broadcasts. We prove that $i_{bn}(G)\leq i_h(G)$ for all graphs $G$ and show that $i_h(G)/i_{bn}(G)$ is bounded. For both parameters, we show that the lower bn-independence number (h-independence number) of an arbitrary connected graph $G$ equals the minimum lower bn-independence number (h-independence number) among those of its spanning trees.

We further study the maximum cost of boundary independent broadcasts, denoted $\alpha_{bn}(G)$.
We show $\alpha_{bn}(G)$ can be bounded in terms of the independence number $\alpha(G)$, and  prove that the maximum bn-independent broadcast problem is NP-hard by a reduction from the independent set problem to an instance of the maximum bn-independent
broadcast problem.

With particular interest in caterpillars, we investigate bounds on $\alpha_{bn}(T)$ when $T$ is a tree in terms of its order and the number of vertices of degree at least 3, known as the \textit{branch vertices} of $T$. We conclude by describing a polynomial-time algorithm to determine $\alpha_{bn}(T)$ for a given tree $T$.

\end{abstract}



\section{Introduction}

There are several methods by which the concept of independent sets may be generalized to broadcast independence. If we require that no broadcasting vertex hears another, we obtain the definition of \textit{cost independent} broadcasts introduced by Erwin in \cite{erwin}, which we refer to as \textit{hearing independent}, abbreviated h-independent. 
The definition of boundary independent (or bn-independent) broadcasts, in which no broadcasts overlap on edges, was introduced by Neilson \cite{neilsonphd} and Mynhardt and Neilson \cite{mynhardtneilsonboundary} as an alternative to hearing independence. 
We further investigate the lower parameters $i_{h}(G)$ and $i_{bn}(G)$ on general graphs and trees, proving the comparability of these parameters and obtaining an upper bound on the ratio $i_h{(G)}/i_{bn}(G)$.

We present broadcast definitions and known results in Section 2, and preliminary results and observations in Section 3. In Section 4 we determine that for any connected graph $G$, 

\begin{center}
    $i_{bn}(G)=\min\{i_{bn}(T)\,:\,T\text{ is a spanning tree of G}\}$,
\end{center}

\noindent from which we prove that $i_{bn}(G)\leq i_h(G)$ for any graph $G$. As a corollary, by following the proof of the analogous result for the lower boundary independence number, we find that for any connected graph $G$,

\begin{center}
    $i_{h}(G)=\min\{i_{h}(T)\,:\,T\text{ is a spanning tree of G}\}$.
\end{center}

In Section 5 we show that $\frac{i_{h}
(G)}{i_{bn}(G)}\leq \frac{5}{4}$ for all graphs $G$.

In the latter half of this paper, our focus shifts to upper parameter for boundary independence.   
The maximum cost of a boundary independent broadcast on a given graph $G$ is referred to as its \textit{boundary independence number}, denoted $\alpha_{bn}(G)$. For a given integer $k\geq 0$, the problem of determining whether $\alpha_{bn}(G)\geq k$ is called the \textit{maximum bn-independent
broadcast problem}. The \textit{hearing independence number} $\alpha_h(G)$ and the \textit{maximum h-independent
broadcast problem} are defined similarly.

As any boundary independent broadcast is hearing independent, it follows from the definitions that $\alpha_{bn}(G)\leq \alpha_h(G)$ for all graphs $G$. 
In \cite{mynhardtneilsonboundary}, Mynhardt and Neilson showed that $\alpha_h(G)/ \alpha_{bn}(G)<2$, and that this bound is asymptotically best possible. They posed the problem of investigating the ratio $\alpha_{bn}(G)/\alpha(G)$ in \cite{mynhardt2021lowerexact}.

\begin{problem}\label{prob2}
Is it true that $\alpha_{bn}(G)<2\alpha(G)$ for any graph $G$? 
\end{problem}

It was shown in \cite{mynhardtneilsonboundary} that $\alpha_{bn}(G)\leq n-1$ for all graphs $G$ of order $n$, with equality if and only if $G$ is a path or a \textit{generalized spider}, a tree with exactly one vertex of degree greater than 2. 
It is easily observed that ${\alpha(G)\leq n-\delta(G)}$, where $\delta(G)$ denotes the minimum degree among the vertices of $G$. In \cite{mncomparingupper}, Mynhardt and Neilson asked whether a similar inequality existed for the maximum boundary independence number.

\begin{problem}\label{prob3}
Show that $\alpha_{bn}(G)\leq n-\delta(G)$ for any graph $G$ of order $n$. 
\end{problem}

For any tree $T$, the bound in Problem 1.2 follows immediately from the bound $\alpha_{bn}(T)\leq n-1$ and the fact that $\delta(T)=1$.

In Section 6, we show that  $\alpha_{bn}(G)<2\alpha(G)$ for all $G$, solving Problem \ref{prob2}. We further resolve Problem \ref{prob3} by showing that $\alpha_{bn}(G)\leq n-\delta(G)$ for any graph $G$.

In Section 7, by considering a transformation from independent sets to boundary independent broadcasts on graphs, we observe that determining whether $\alpha_{bn}(G)\geq k$ for a given integer $k$ is NP-Complete. 
In Section 8, we investigate the maximum boundary independence number of trees and determine $\alpha_{bn}(G)$ exactly for families of caterpillars. 

We continue our study of maximum boundary independence broadcasts in trees in Section 9. Using a method similar to the proof technique employed by Bessy and Rautenbach in \cite{bessy}, we derive an $O(n^9)$ time algorithm to determine $\alpha_{bn}(T)$ for a given tree $T$.

Open problems and directions for further research are discussed in Section 10.

\section{Definitions and Background}

Erwin \cite{erwin} defined
a \textit{broadcast} on a nontrivial connected graph $G$ as a function $f:V(G)\rightarrow \{0, 1, ..., \text{diam}(G)\}$ such that $f(v)$ is at most the eccentricity $e(v)$ for all vertices $v$. We say a vertex $v$ is \textit{broadcasting} if $f(v)\geq 1$, and that $f(v)$ is the \textit{strength} of $f$ from $v$. 
The \textit{cost} or \textit{weight} of $f$ is $\sigma(f)=\sum_{v\in V(G)}f(v)$. 

Given a broadcast $f$ on $G$ and a broadcasting vertex $v$, a vertex $u$ \textit{hears} $f$ from $v$ if $d_G(u, v)\leq f(v)$. 
We define the $f$-\textit{neighbourhood} of $v$, denoted by $N_f(v)$, as the set of all vertices which hear $f$ from $v$ (including $v$ itself). 

The $f$-\textit{private neighbourhood} of $v$, denoted by $PN_f(v)$, consists of those vertices that hear $f$ only from $v$. The $f$-\textit{boundary} of $v$ is $B_f(v)=\{u\in N_f(v)\,|\,d(u, v)= f(v)\}$. The $f$-\textit{private boundary} $PB_f(v)$ is defined analogously. In particular, $PB_f(v) = PN_f(v)\cap B_f(v)$. 
If $u\in N_f(v)\backslash B_f(v)$, $v$ is said to \textit{overdominate} $u$ by $k$, where $k=f(v)-d_G(u, v)$. A vertex which does not broadcast or hear $f$ from any broadcasting vertex is \textit{undominated}. 

\begin{sloppypar}
Throughout this paper, we partition the set of broadcasting vertices $V_f^+$ into 
${V_f^1=\{v\in V(G)\,|\,f(v)=1\}}$ and $V_f^{++}=\{v\in V(G)\,|\,f(v)>1\}$. We denote the set of undominated vertices by $U_f$. %For example, in Figure \ref{fig:firstexample}, $V_f^1=\{v_1\}$, $V_f^{++}=\{v_2\}$, and $U_f=\{w_1, w_2\}$. 
A broadcast $f$ is \textit{dominating} if $U_f=\emptyset$. The broadcast domination number, $\gamma_b(G)$, is the minimum weight of such a broadcast. 
An overview of broadcast domination in graphs is given by Henning, MacGillivray, and Yang in \cite{inbook}. 

\end{sloppypar}

We say an edge $e=uv$ \textit{hears} $f$ or \textit{is covered by} $w\in V_f^+$ if $u, v \in N_f(w)$ and at least one endpoint does not lie on the $f$-boundary of $w$. If no such vertex $w$ exists, then $e$ is $uncovered$. The set of uncovered edges is denoted $U_f^E$. 

An \textit{independent set} on a graph $G$ is a set of pairwise nonadjacent vertices. The minimum cardinality of a maximal independent set, called the \textit{independent domination number} of $G$, is denoted $i(G)$. 
A broadcast $f$ is \textit{hearing independent} if $x\notin N_f(v)$ for any $x, v \in V_f^+$. It is \textit{boundary independent} if $N_f(v)\backslash B_f(v)\subseteq PN_f(v)$ for all $v\in V_f^+$.

\begin{figure}[H]
    \centering

\tikzset{every picture/.style={line width=0.75pt}} %set default line width to 0.75pt        

\begin{tikzpicture}[x=0.75pt,y=0.75pt,yscale=-1,xscale=1]

%Straight Lines [id:da6085153950960471] 
\draw    (130,160) -- (170,160) ;
\draw [shift={(130,160)}, rotate = 0] [color={rgb, 255:red, 0; green, 0; blue, 0 }  ][fill={rgb, 255:red, 0; green, 0; blue, 0 }  ][line width=0.75]      (0, 0) circle [x radius= 3.35, y radius= 3.35]   ;
%Straight Lines [id:da29802293003402425] 
\draw    (290,160) -- (330,160) ;
\draw [shift={(290,160)}, rotate = 0] [color={rgb, 255:red, 0; green, 0; blue, 0 }  ][fill={rgb, 255:red, 0; green, 0; blue, 0 }  ][line width=0.75]      (0, 0) circle [x radius= 3.35, y radius= 3.35]   ;
%Straight Lines [id:da5063995986415559] 
\draw    (250,160) -- (290,160) ;
\draw [shift={(250,160)}, rotate = 0] [color={rgb, 255:red, 0; green, 0; blue, 0 }  ][fill={rgb, 255:red, 0; green, 0; blue, 0 }  ][line width=0.75]      (0, 0) circle [x radius= 3.35, y radius= 3.35]   ;
%Straight Lines [id:da1445924410338928] 
\draw    (210,160) -- (250,160) ;
\draw [shift={(210,160)}, rotate = 0] [color={rgb, 255:red, 0; green, 0; blue, 0 }  ][fill={rgb, 255:red, 0; green, 0; blue, 0 }  ][line width=0.75]      (0, 0) circle [x radius= 3.35, y radius= 3.35]   ;
%Straight Lines [id:da3853216277928395] 
\draw    (170,160) -- (210,160) ;
\draw [shift={(170,160)}, rotate = 0] [color={rgb, 255:red, 0; green, 0; blue, 0 }  ][fill={rgb, 255:red, 0; green, 0; blue, 0 }  ][line width=0.75]      (0, 0) circle [x radius= 3.35, y radius= 3.35]   ;
%Straight Lines [id:da015277429688481226] 
\draw    (330,160) -- (370,160) ;
\draw [shift={(330,160)}, rotate = 0] [color={rgb, 255:red, 0; green, 0; blue, 0 }  ][fill={rgb, 255:red, 0; green, 0; blue, 0 }  ][line width=0.75]      (0, 0) circle [x radius= 3.35, y radius= 3.35]   ;
%Straight Lines [id:da47965676222395315] 
\draw    (170,120) -- (170,160) ;
\draw [shift={(170,120)}, rotate = 90] [color={rgb, 255:red, 0; green, 0; blue, 0 }  ][fill={rgb, 255:red, 0; green, 0; blue, 0 }  ][line width=0.75]      (0, 0) circle [x radius= 3.35, y radius= 3.35]   ;
%Straight Lines [id:da9231611925994823] 
\draw    (290,120) -- (290,160) ;
\draw [shift={(290,120)}, rotate = 90] [color={rgb, 255:red, 0; green, 0; blue, 0 }  ][fill={rgb, 255:red, 0; green, 0; blue, 0 }  ][line width=0.75]      (0, 0) circle [x radius= 3.35, y radius= 3.35]   ;
%Straight Lines [id:da9900826382291812] 
\draw    (290,80) -- (290,120) ;
\draw [shift={(290,80)}, rotate = 90] [color={rgb, 255:red, 0; green, 0; blue, 0 }  ][fill={rgb, 255:red, 0; green, 0; blue, 0 }  ][line width=0.75]      (0, 0) circle [x radius= 3.35, y radius= 3.35]   ;
%Straight Lines [id:da11952158814169045] 
\draw    (370,160) -- (407.65,160) ;
\draw [shift={(410,160)}, rotate = 0] [color={rgb, 255:red, 0; green, 0; blue, 0 }  ][line width=0.75]      (0, 0) circle [x radius= 3.35, y radius= 3.35]   ;
\draw [shift={(370,160)}, rotate = 0] [color={rgb, 255:red, 0; green, 0; blue, 0 }  ][fill={rgb, 255:red, 0; green, 0; blue, 0 }  ][line width=0.75]      (0, 0) circle [x radius= 3.35, y radius= 3.35]   ;
%Shape: Arc [id:dp034805478154294356] 
\draw  [draw opacity=0][dash pattern={on 0.84pt off 2.51pt}] (128.32,160.11) .. controls (128.32,160.11) and (128.32,160.11) .. (128.32,160.11) .. controls (128.32,160.11) and (128.32,160.11) .. (128.32,160.11) .. controls (128.32,138.09) and (146.63,120.23) .. (169.22,120.23) .. controls (191.81,120.23) and (210.12,138.09) .. (210.12,160.11) -- (169.22,160.11) -- cycle ; \draw  [dash pattern={on 0.84pt off 2.51pt}] (128.32,160.11) .. controls (128.32,160.11) and (128.32,160.11) .. (128.32,160.11) .. controls (128.32,160.11) and (128.32,160.11) .. (128.32,160.11) .. controls (128.32,138.09) and (146.63,120.23) .. (169.22,120.23) .. controls (191.81,120.23) and (210.12,138.09) .. (210.12,160.11) ;  
%Shape: Arc [id:dp38480276940700486] 
\draw  [draw opacity=0][dash pattern={on 0.84pt off 2.51pt}] (210,160) .. controls (210,160) and (210,160) .. (210,160) .. controls (210,115.82) and (245.82,80) .. (290,80) .. controls (334.18,80) and (370,115.82) .. (370,160) -- (290,160) -- cycle ; \draw  [dash pattern={on 0.84pt off 2.51pt}] (210,160) .. controls (210,160) and (210,160) .. (210,160) .. controls (210,115.82) and (245.82,80) .. (290,80) .. controls (334.18,80) and (370,115.82) .. (370,160) ;  
%Straight Lines [id:da37148432064551185] 
\draw    (414,160) -- (447.65,160) ;
\draw [shift={(450,160)}, rotate = 0] [color={rgb, 255:red, 0; green, 0; blue, 0 }  ][line width=0.75]      (0, 0) circle [x radius= 3.35, y radius= 3.35]   ;

% Text Node
\draw (161,172.4) node [anchor=north west][inner sep=0.75pt]  [font=\small]  {$v_{1}$};
% Text Node
\draw (321,172.4) node [anchor=north west][inner sep=0.75pt]  [font=\small]  {$u$};
% Text Node
\draw (281,172.4) node [anchor=north west][inner sep=0.75pt]  [font=\small]  {$v_{2}$};
% Text Node
\draw (401,172.4) node [anchor=north west][inner sep=0.75pt]  [font=\small]  {$w_1$};
\draw (440,172.4) node [anchor=north west][inner sep=0.75pt]  [font=\small]  {$w_2$};

\end{tikzpicture}

    \caption{A boundary independent broadcast $f$ on a tree. Vertices $v_1$ and $v_2$ broadcast at strengths 1 and 2, respectively. The vertex $v_2$ overdominates $u$ by 1, whereas $w_1$ and $w_2$ are undominated.}
    \label{fig:firstexample}
\end{figure}
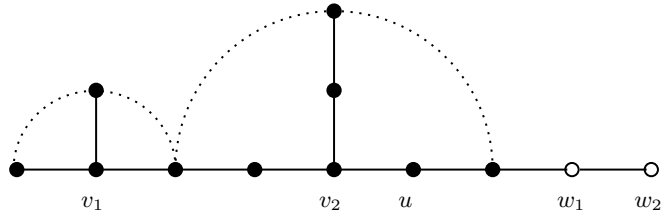

Although efficient broadcast domination was shown to be solvable in polynomial time for
every graph in \cite{heggernes}, the complexity of hearing independence was unknown even for
trees until an efficient algorithm was found by Bessy and Rautenbach in \cite{bessy}. 

\begin{theorem} \textup{\cite{bessy}}
For any tree $T$ of order $n$, $\alpha_h(T)$ can be determined in $O(n^9)$ time. 
\end{theorem}

\textit{Maximal} independent broadcasts are those for which the broadcast strength cannot be increased at any vertex without violating the independence condition. If $f$ and $g$ are broadcasts on a graph $G$, we say that $g\leq f$ if $g(v)\leq f(v)$ for all $v\in V(G)$. If in addition $g(v)<f(v)$ for some $v$, we write $g<f$. 
A boundary independent broadcast $f$ is \textit{maximal boundary independent} if there exists no boundary independent broadcast $g$ such that $g > f$. Equivalently, $f$ is a maximal boundary independent broadcast if $B_f(v)-PB_f(v)\neq \emptyset$ for all $v\in V_f^+$. Mynhardt and Neilson  \cite{mynhardtneilsonboundary} defined $i_{bn}(G)$ as the minimum weight of a maximal boundary independent broadcast on $G$,  the \textit{lower boundary independence number}. 
The minimum weight of a maximal hearing independent broadcast is denoted $i_h(G)$, or the \textit{lower hearing independence number}.

Hearing independence was further studied by Bessy and Rautenbach \cite{bessy3, bessy2} and by Dunbar et al. \cite{efficient}.
The more recent study of boundary independent broadcasts was continued by Mynhardt and Neilson in \cite{mynhardt2021sharp, mncomparingupper, mynhardt2021lowerexact} and by Marchessalt and Mynhardt in \cite{mynhardtmarchessault2021lower}. 
For terminology and general concepts in graphs theory not defined in this paper, see Chartrand, Lesniak, and Zhang \cite{graphsanddigraphs}.

\section{Preliminaries}

Observe that if a broadcast $f$ is h-independent or bn-independent but not dominating on a graph $G$, then $f$ may be extended to a dominating broadcast $g$ by successively broadcasting at strength 1 from an uncovered vertex in $V(G)$ until no such vertices remain. We state this fact below for reference. 

\begin{observation}
\label{obs:f_is_dominating}
If $f$ is a maximal boundary independent or maximal hearing independent broadcast, then $f$ is dominating. 
\end{observation}

Mynhardt and Neilson extended Observation \ref{obs:f_is_dominating} to a necessary and sufficient condition for a boundary independent broadcast to be maximal boundary independent.

\begin{proposition} \textup{\cite{mynhardtneilsonboundary}}
Let $f$ be a boundary independent broadcast on a connected graph $G$. Then $f$ is maximal bn-independent if and only if $f$ is dominating, and either

\begin{enumerate}
    \item[$(i)$] $|V_f^+|=1$, or
    \item[$(ii)$] $B_f(v)-PB_f(v)\neq \emptyset$ for each $v\in V_f^+$. 
\end{enumerate}
 
\label{prop:necsufbnindependent}
\end{proposition}

It is natural to consider the analogous result for maximal hearing independence.

\begin{proposition}
\label{prop:nec_suf_hearing} Let $f$ be a hearing independent broadcast on a connected graph $G$. Then $f$ is maximal hearing
independent if and only if $f$ is dominating, and either 

\begin{enumerate}
\item[$(i)$] $|V_{f}^{+}|=1$, or

\item[$(ii)$]  for each $v\in V_{f}^{+}$
there exist $u\in B_{f}(v)$ and $w\in V_{f}^{+}-\{v\}$ such that $uw\in E(G)$, i.e., each broadcasting vertex has a vertex on its boundary that is adjacent to another vertex in $V_{f}^{+}$.
\end{enumerate}
\end{proposition}

To illustrate \textit{(ii)}, observe that in Figure \ref{fig:hearingdom}, the dominating h-independent broadcast $f$ cannot be increased at $v$, otherwise the vertex $w\in V_f^+$ adjacent to $u\in B_f(v)$ would hear $f$ from $v$. Similarly, $f$ cannot be increased at either vertex broadcasting at strength~1.

\begin{figure}[H]
    \centering

\tikzset{every picture/.style={line width=0.75pt}} %set default line width to 0.75pt        

\begin{tikzpicture}[x=0.75pt,y=0.75pt,yscale=-1,xscale=1]

%Straight Lines [id:da3447472063549746] 
\draw    (70,140) -- (110,140) ;
\draw [shift={(70,140)}, rotate = 0] [color={rgb, 255:red, 0; green, 0; blue, 0 }  ][fill={rgb, 255:red, 0; green, 0; blue, 0 }  ][line width=0.75]      (0, 0) circle [x radius= 3.35, y radius= 3.35]   ;
%Straight Lines [id:da7356137931829796] 
\draw    (110,140) -- (150,140) ;
\draw [shift={(110,140)}, rotate = 0] [color={rgb, 255:red, 0; green, 0; blue, 0 }  ][fill={rgb, 255:red, 0; green, 0; blue, 0 }  ][line width=0.75]      (0, 0) circle [x radius= 3.35, y radius= 3.35]   ;
%Straight Lines [id:da39836299483426574] 
\draw    (150,140) -- (190,100) ;
\draw [shift={(150,140)}, rotate = 315] [color={rgb, 255:red, 0; green, 0; blue, 0 }  ][fill={rgb, 255:red, 0; green, 0; blue, 0 }  ][line width=0.75]      (0, 0) circle [x radius= 3.35, y radius= 3.35]   ;
%Straight Lines [id:da6967662667849024] 
\draw    (190,100) -- (230,100) ;
\draw [shift={(190,100)}, rotate = 0] [color={rgb, 255:red, 0; green, 0; blue, 0 }  ][fill={rgb, 255:red, 0; green, 0; blue, 0 }  ][line width=0.75]      (0, 0) circle [x radius= 3.35, y radius= 3.35]   ;
%Straight Lines [id:da9617946028428854] 
\draw    (350,140) -- (310,180) ;
\draw [shift={(350,140)}, rotate = 135] [color={rgb, 255:red, 0; green, 0; blue, 0 }  ][fill={rgb, 255:red, 0; green, 0; blue, 0 }  ][line width=0.75]      (0, 0) circle [x radius= 3.35, y radius= 3.35]   ;
%Straight Lines [id:da7266129671339787] 
\draw    (310,100) -- (350,140) ;
\draw [shift={(310,100)}, rotate = 45] [color={rgb, 255:red, 0; green, 0; blue, 0 }  ][fill={rgb, 255:red, 0; green, 0; blue, 0 }  ][line width=0.75]      (0, 0) circle [x radius= 3.35, y radius= 3.35]   ;
%Straight Lines [id:da7451264738958125] 
\draw    (270,100) -- (310,100) ;
\draw [shift={(270,100)}, rotate = 0] [color={rgb, 255:red, 0; green, 0; blue, 0 }  ][fill={rgb, 255:red, 0; green, 0; blue, 0 }  ][line width=0.75]      (0, 0) circle [x radius= 3.35, y radius= 3.35]   ;
%Straight Lines [id:da4215089383024391] 
\draw    (230,100) -- (270,100) ;
\draw [shift={(230,100)}, rotate = 0] [color={rgb, 255:red, 0; green, 0; blue, 0 }  ][fill={rgb, 255:red, 0; green, 0; blue, 0 }  ][line width=0.75]      (0, 0) circle [x radius= 3.35, y radius= 3.35]   ;
%Straight Lines [id:da8822938541511909] 
\draw    (190,180) -- (150,140) ;
\draw [shift={(190,180)}, rotate = 225] [color={rgb, 255:red, 0; green, 0; blue, 0 }  ][fill={rgb, 255:red, 0; green, 0; blue, 0 }  ][line width=0.75]      (0, 0) circle [x radius= 3.35, y radius= 3.35]   ;
%Straight Lines [id:da47901304954288704] 
\draw    (390,140) -- (350,140) ;
\draw [shift={(390,140)}, rotate = 180] [color={rgb, 255:red, 0; green, 0; blue, 0 }  ][fill={rgb, 255:red, 0; green, 0; blue, 0 }  ][line width=0.75]      (0, 0) circle [x radius= 3.35, y radius= 3.35]   ;
%Straight Lines [id:da03917319461429125] 
\draw    (250,180) -- (190,180) ;
\draw [shift={(250,180)}, rotate = 180] [color={rgb, 255:red, 0; green, 0; blue, 0 }  ][fill={rgb, 255:red, 0; green, 0; blue, 0 }  ][line width=0.75]      (0, 0) circle [x radius= 3.35, y radius= 3.35]   ;
%Straight Lines [id:da398919856644363] 
\draw    (310,180) -- (250,180) ;
\draw [shift={(310,180)}, rotate = 180] [color={rgb, 255:red, 0; green, 0; blue, 0 }  ][fill={rgb, 255:red, 0; green, 0; blue, 0 }  ][line width=0.75]      (0, 0) circle [x radius= 3.35, y radius= 3.35]   ;

% Text Node
\draw (267,82.4) node [anchor=north west][inner sep=0.75pt]  [font=\small]  {$1$};
% Text Node
\draw (347,122.4) node [anchor=north west][inner sep=0.75pt]  [font=\small]  {$1$};
% Text Node
\draw (141,122.4) node [anchor=north west][inner sep=0.75pt]  [font=\small]  {$2$};
% Text Node
\draw (141,146.4) node [anchor=north west][inner sep=0.75pt]  [font=\footnotesize]  {$v$};
% Text Node
\draw (267,106.4) node [anchor=north west][inner sep=0.75pt]  [font=\footnotesize]  {$w$};
% Text Node
\draw (227,106.4) node [anchor=north west][inner sep=0.75pt]  [font=\footnotesize]  {$u$};

\end{tikzpicture}

    \caption{ A dominating maximal h-independent broadcast $f$ with $|V_f^+|=3$.} 
    \label{fig:hearingdom}
\end{figure}
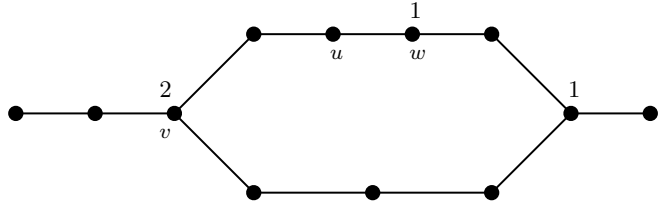

\begin{proof}
Let $f$ be a maximal h-independent broadcast on $G$. By Observation \ref{obs:f_is_dominating}, $f$ is dominating. Suppose there exists $v\in V_f^+$ such that $uw\notin E(G)$ for all $u\in B_f(v)$ and $w\in V_f^+-\{v\}$. Then either $|V_f^+|=1$ (in which case $f(v)=e(v)$, where $e(v)$ denotes the eccentricity of $v$), or we may define a new broadcast $f'$ where $f'(v)=f(v)+1$ and $f'(x)=f(x)$ for all $x\neq v$. Since $f'$ is dominating and no broadcasting vertex hears another, $f'$ is h-independent on $G$, contradicting the maximality of $f$. 

Conversely, let $f$ be a dominating h-independent broadcast such that (i) or (ii) hold. If $|V_f^+|=1$ and $V_f^+=\{v\}$, then $f(v)=e(v)$, otherwise $f$ would not be dominating.
%and so $f$ is maximal h-independent by definition. S
Suppose for a contradiction that $f$ satisfies (ii) and is not maximal h-independent. Then there exists $v\in V_f^+$ such that increasing the strength of the broadcast on $v$ by 1 results in a new h-independent broadcast $f'$. By (ii), since $v\in V_f^+$, there exists $u\in B_f(u)$ adjacent to a broadcasting vertex $w\in V_f^+ -\{v\}$. But then $w\in B_{f'}(v)$, a contradiction. It follows that $f$ is maximal h-independent. 
\end{proof}

From Proposition \ref{prop:nec_suf_hearing}, we derive conditions satisfied by a maximal h-independent broadcast that is also bn-independent.

\begin{corollary}
\label{cor:hearing_no_overlaps}
Let $f$ be a maximal hearing independent broadcast on a connected graph $G$. If $f$ is boundary independent, either $|V_f^+|=1$ or there exists $u\in B_f(v)$ adjacent to a vertex in $V_f^1$ for each $v\in V_f^+$.
\end{corollary}

\begin{proof}

Suppose $f$ is a maximal h-independent broadcast on $G$ such that $f$ is boundary independent; that is, no edge of $G$ hears more than one broadcasting vertex. 

Since $f$ is dominating, by Proposition \ref{prop:nec_suf_hearing} (ii), either $|V_f^+|=1$ or every $v\in V_f^+$ has a vertex $u$ on its $f$-boundary adjacent to another vertex $w\in V_f^+$. In the latter case, if $f(w)\geq 2$, then $w$ overdominates $u$, hence $N_f(v)$ and $N_f(w)$ intersect on an edge. Therefore $f(w)=1$.
\end{proof}

The following results of Marchessault and Mynhardt will be useful throughout this section.
For a path $P$ in a tree $T$, let $d(v, P)$ denote the minimum distance from a vertex $v\in V(T)$ to a vertex on $P$.

\begin{proposition} \textup{\cite{mynhardtmarchessault2021lower}}
Let P be a path in a tree T and let f be a broadcast on T. Let Touch$\,(P)$ denote the set of broadcasting vertices whose $f$-neighbourhoods intersect $P$, and let Off$\,(P)$ denote the remaining broadcasting vertices, that is, those that do not broadcast to any vertex of $P$. Suppose

\begin{center}
    $ \sigma_1 = \summ_{v\in \textup{Touch}(P)}d(v,P) $ and  $ \sigma_2 = \summ_{v\in \textup{Off}(P)}f(v)$.
\end{center}
Then
\begin{enumerate}
  \item $f$ covers at most $2\left (\summ_{v\in \textup{Touch}(P)}f(v)-\sigma_1\right )$ edges of $P$, and
  \item if f covers b edges of P, then $\sigma(f)\geq \lceil \frac{b}{2}\rceil +\sigma_1 + \sigma_2$. 
\end{enumerate}

\label{prop:marchessaultpath}
\end{proposition}

\begin{figure}[H]
    \centering

\tikzset{every picture/.style={line width=0.75pt}} %set default line width to 0.75pt        

\begin{tikzpicture}[x=0.75pt,y=0.75pt,yscale=-1,xscale=1]

%Straight Lines [id:da37542672917054] 
\draw    (110,180) -- (150,180) ;
\draw [shift={(110,180)}, rotate = 0] [color={rgb, 255:red, 0; green, 0; blue, 0 }  ][fill={rgb, 255:red, 0; green, 0; blue, 0 }  ][line width=0.75]      (0, 0) circle [x radius= 3.35, y radius= 3.35]   ;
%Straight Lines [id:da3406032174418747] 
\draw    (230,180) -- (270,180) ;
\draw [shift={(230,180)}, rotate = 0] [color={rgb, 255:red, 0; green, 0; blue, 0 }  ][fill={rgb, 255:red, 0; green, 0; blue, 0 }  ][line width=0.75]      (0, 0) circle [x radius= 3.35, y radius= 3.35]   ;
%Straight Lines [id:da34303118291789825] 
\draw    (190,180) -- (230,180) ;
\draw [shift={(190,180)}, rotate = 0] [color={rgb, 255:red, 0; green, 0; blue, 0 }  ][fill={rgb, 255:red, 0; green, 0; blue, 0 }  ][line width=0.75]      (0, 0) circle [x radius= 3.35, y radius= 3.35]   ;
%Straight Lines [id:da5357271166354656] 
\draw    (150,180) -- (190,180) ;
\draw [shift={(150,180)}, rotate = 0] [color={rgb, 255:red, 0; green, 0; blue, 0 }  ][fill={rgb, 255:red, 0; green, 0; blue, 0 }  ][line width=0.75]      (0, 0) circle [x radius= 3.35, y radius= 3.35]   ;
%Straight Lines [id:da027098878548082794] 
\draw  [dash pattern={on 0.84pt off 2.51pt}]  (190,140) -- (190,180) ;
\draw [shift={(190,140)}, rotate = 90] [color={rgb, 255:red, 0; green, 0; blue, 0 }  ][fill={rgb, 255:red, 0; green, 0; blue, 0 }  ][line width=0.75]      (0, 0) circle [x radius= 3.35, y radius= 3.35]   ;
%Straight Lines [id:da6446975933586057] 
\draw    (270,180) -- (310,180) ;
\draw [shift={(270,180)}, rotate = 0] [color={rgb, 255:red, 0; green, 0; blue, 0 }  ][fill={rgb, 255:red, 0; green, 0; blue, 0 }  ][line width=0.75]      (0, 0) circle [x radius= 3.35, y radius= 3.35]   ;
%Straight Lines [id:da2465846275903658] 
\draw    (310,180) -- (350,180) ;
\draw [shift={(350,180)}, rotate = 0] [color={rgb, 255:red, 0; green, 0; blue, 0 }  ][fill={rgb, 255:red, 0; green, 0; blue, 0 }  ][line width=0.75]      (0, 0) circle [x radius= 3.35, y radius= 3.35]   ;
\draw [shift={(310,180)}, rotate = 0] [color={rgb, 255:red, 0; green, 0; blue, 0 }  ][fill={rgb, 255:red, 0; green, 0; blue, 0 }  ][line width=0.75]      (0, 0) circle [x radius= 3.35, y radius= 3.35]   ;
%Straight Lines [id:da032452817588781935] 
\draw  [dash pattern={on 0.84pt off 2.51pt}]  (270,140) -- (270,180) ;
\draw [shift={(270,140)}, rotate = 90] [color={rgb, 255:red, 0; green, 0; blue, 0 }  ][fill={rgb, 255:red, 0; green, 0; blue, 0 }  ][line width=0.75]      (0, 0) circle [x radius= 3.35, y radius= 3.35]   ;
%Straight Lines [id:da7844930635720424] 
\draw  [dash pattern={on 0.84pt off 2.51pt}]  (270,100) -- (270,140) ;
\draw [shift={(270,100)}, rotate = 90] [color={rgb, 255:red, 0; green, 0; blue, 0 }  ][fill={rgb, 255:red, 0; green, 0; blue, 0 }  ][line width=0.75]      (0, 0) circle [x radius= 3.35, y radius= 3.35]   ;
%Straight Lines [id:da052222872384500185] 
\draw    (70,180) -- (110,180) ;
\draw [shift={(70,180)}, rotate = 0] [color={rgb, 255:red, 0; green, 0; blue, 0 }  ][fill={rgb, 255:red, 0; green, 0; blue, 0 }  ][line width=0.75]      (0, 0) circle [x radius= 3.35, y radius= 3.35]   ;

% Text Node
\draw (186,120.4) node [anchor=north west][inner sep=0.75pt]  [font=\small]  {$3$};
% Text Node
\draw (306,160.4) node [anchor=north west][inner sep=0.75pt]  [font=\small]  {$1$};
% Text Node
\draw (266,80.4) node [anchor=north west][inner sep=0.75pt]  [font=\small]  {$2$};
% Text Node
\draw (201,192.4) node [anchor=north west][inner sep=0.75pt]    {$P$};

\end{tikzpicture}

    \caption{Vertices broadcasting to a path $P$ of a tree. Observe that each broadcasting vertex $v$ covers at most $2(f(v)-d(v, P))$ edges of $P$.}
    \label{fig:my_label}
\end{figure}
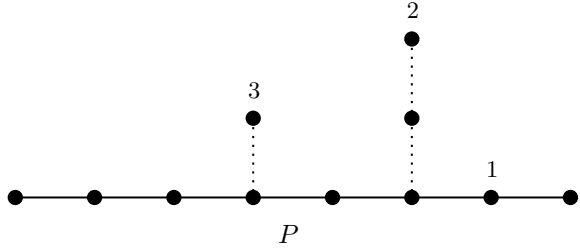

In particular, if $D$ is a diametrical path of a tree $T$ and  $f$ covers every edge of $D$, then $\sigma(f)\geq \text{rad}(T)$.

Recall that $U_f^E$ denotes the set of edges uncovered by a broadcast $f$.

\begin{proposition} \textup{\cite{mynhardtmarchessault2021lower}}
Let $f$ be a bn-independent broadcast on a connected graph $G$ such that $|V_f^+|\geq 2$. Then f is maximal bn-independent if and only if each component of $G-U_f^E$ contains at least two broadcasting vertices. 
\label{prop:marchessault2broadcasts}
\end{proposition}

Note that if each component of $G-U_f^E$ contains at least one broadcasting vertex, then $f$ is dominating, since $G-U_f^E$ is a spanning subgraph of $G$. 

It is clear that the first direction of Proposition \ref{prop:marchessault2broadcasts} must also hold for hearing independence, for if $f$ is maximal h-independent and some component $C$ of $G-U_f^E$ contains only a single broadcasting vertex $v$, all edges between $G-C$ and $B_f(v)$ are uncovered. But then increasing the broadcast strength of $v$ by 1 results in a new h-independent broadcast of greater cost, a contradiction. 

On the other hand, if $f$ is a hearing independent broadcast on a connected graph $G$ such that $|V_f^+|\geq 2$ and all components of $G-U_f^E$ contain at least two broadcasting vertices, $f$ is not necessarily maximal hearing independent as broadcasts may overlap on edges within components. Such a broadcast is illustrated in Figure \ref{fig:componentexample}.

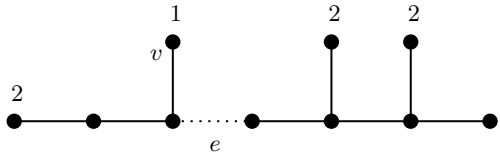
\begin{figure}[H]
    \centering

\tikzset{every picture/.style={line width=0.75pt}} %set default line width to 0.75pt        

\begin{tikzpicture}[x=0.75pt,y=0.75pt,yscale=-1,xscale=1]

%Straight Lines [id:da14883332089583812] 
\draw    (120,150) -- (160,150) ;
\draw [shift={(120,150)}, rotate = 0] [color={rgb, 255:red, 0; green, 0; blue, 0 }  ][fill={rgb, 255:red, 0; green, 0; blue, 0 }  ][line width=0.75]      (0, 0) circle [x radius= 3.35, y radius= 3.35]   ;
%Straight Lines [id:da3069590750568836] 
\draw    (160,150) -- (200,150) ;
\draw [shift={(160,150)}, rotate = 0] [color={rgb, 255:red, 0; green, 0; blue, 0 }  ][fill={rgb, 255:red, 0; green, 0; blue, 0 }  ][line width=0.75]      (0, 0) circle [x radius= 3.35, y radius= 3.35]   ;
%Straight Lines [id:da7743541564142027] 
\draw  [dash pattern={on 0.84pt off 2.51pt}]  (200,150) -- (240,150) ;
\draw [shift={(200,150)}, rotate = 0] [color={rgb, 255:red, 0; green, 0; blue, 0 }  ][fill={rgb, 255:red, 0; green, 0; blue, 0 }  ][line width=0.75]      (0, 0) circle [x radius= 3.35, y radius= 3.35]   ;
%Straight Lines [id:da31974264466645974] 
\draw    (240,150) -- (280,150) ;
\draw [shift={(240,150)}, rotate = 0] [color={rgb, 255:red, 0; green, 0; blue, 0 }  ][fill={rgb, 255:red, 0; green, 0; blue, 0 }  ][line width=0.75]      (0, 0) circle [x radius= 3.35, y radius= 3.35]   ;
%Straight Lines [id:da2392502698779646] 
\draw    (280,150) -- (320,150) ;
\draw [shift={(280,150)}, rotate = 0] [color={rgb, 255:red, 0; green, 0; blue, 0 }  ][fill={rgb, 255:red, 0; green, 0; blue, 0 }  ][line width=0.75]      (0, 0) circle [x radius= 3.35, y radius= 3.35]   ;
%Straight Lines [id:da4098912111365294] 
\draw    (200,110) -- (200,150) ;
\draw [shift={(200,110)}, rotate = 90] [color={rgb, 255:red, 0; green, 0; blue, 0 }  ][fill={rgb, 255:red, 0; green, 0; blue, 0 }  ][line width=0.75]      (0, 0) circle [x radius= 3.35, y radius= 3.35]   ;
%Straight Lines [id:da5119644304961206] 
\draw    (320,150) -- (360,150) ;
\draw [shift={(360,150)}, rotate = 0] [color={rgb, 255:red, 0; green, 0; blue, 0 }  ][fill={rgb, 255:red, 0; green, 0; blue, 0 }  ][line width=0.75]      (0, 0) circle [x radius= 3.35, y radius= 3.35]   ;
\draw [shift={(320,150)}, rotate = 0] [color={rgb, 255:red, 0; green, 0; blue, 0 }  ][fill={rgb, 255:red, 0; green, 0; blue, 0 }  ][line width=0.75]      (0, 0) circle [x radius= 3.35, y radius= 3.35]   ;
%Straight Lines [id:da8809635396261728] 
\draw    (320,110) -- (320,150) ;
\draw [shift={(320,110)}, rotate = 90] [color={rgb, 255:red, 0; green, 0; blue, 0 }  ][fill={rgb, 255:red, 0; green, 0; blue, 0 }  ][line width=0.75]      (0, 0) circle [x radius= 3.35, y radius= 3.35]   ;
%Straight Lines [id:da4325472876047438] 
\draw    (280,110) -- (280,150) ;
\draw [shift={(280,110)}, rotate = 90] [color={rgb, 255:red, 0; green, 0; blue, 0 }  ][fill={rgb, 255:red, 0; green, 0; blue, 0 }  ][line width=0.75]      (0, 0) circle [x radius= 3.35, y radius= 3.35]   ;

% Text Node
\draw (117,130.4) node [anchor=north west][inner sep=0.75pt]  [font=\small]  {$2$};
% Text Node
\draw (317,90.4) node [anchor=north west][inner sep=0.75pt]  [font=\small]  {$2$};
% Text Node
\draw (277,90.4) node [anchor=north west][inner sep=0.75pt]  [font=\small]  {$2$};
% Text Node
\draw (197,90.4) node [anchor=north west][inner sep=0.75pt]  [font=\small]  {$1$};
% Text Node
\draw (217,158.4) node [anchor=north west][inner sep=0.75pt]  [font=\small]  {$e$};
% Text Node
\draw (187,112.4) node [anchor=north west][inner sep=0.75pt]  [font=\small]  {$v$};

\end{tikzpicture}

    \caption{A hearing independent broadcast $f$ such that the removal of the $f$-uncovered edge $e$ leaves two components, each of which contains two broadcasting vertices. As increasing the strength of the broadcast from $v$ by 1 does not result in any broadcasting vertex hearing another, $f$ is not maximal h-independent.}
    \label{fig:componentexample}
\end{figure}

\section{The comparability of $i_{bn}$ and $i_h$}

Two graph parameters $p$ and $q$ are \textit{incomparable} if there exist graphs $G, G'$ for which $p(G)<q(G)$ and $p(G')>q(G')$. We write this as $p \diamond q$. In \cite{mynhardtneilsonboundary}, Mynhardt and Neilson observed that $i\diamond i_{bn}$ and~$i \diamond i_h$.

It is natural to ask whether there exist graphs for which $i_h(G)<i_{bn}(G)$. 
Suppose there exists a graph $G$ with two or more vertices of high degree such that broadcasts from each of these vertices will cover $G$ only if some broadcasts may overlap on edges. Assuming an $i_{bn}$-broadcast cannot be constructed by broadcasting from a single radial vertex, it seems reasonable to imagine a case in which a maximal h-independent broadcast has lower cost than a maximal bn-independent broadcast. 
We proceed to show that this is impossible, solving an open problem posed in \cite{mynhardtmarchessault2021lower}.

\begin{theorem}
\label{theorem:ibnih_gen2}
For any graph $G$, $i_{bn}(G)\leq i_h(G)$.
\end{theorem}

Since the cost of a broadcast is equal to the sum of the costs of the broadcasts on each of its components, it suffices to consider connected graphs. We begin by proving special cases of broadcasts or graphs, including when $G$ is a tree. The proof of Theorem \ref{theorem:ibnih_gen2} is presented in Subsection~4.3.

We first consider the case in which no vertices broadcast at strength greater than 1. 

\begin{proposition}
Let $f$ be a broadcast on $G$ such that $|V_f^{++}|=0$. Then $f$ is maximal boundary independent if and only if it is maximal hearing independent. 
\label{all1sbroadcast}
\end{proposition}

\begin{proof}
Suppose $f$ is a maximal bn-independent or maximal h-independent (and hence dominating) broadcast on $G$ such that $V_f^+=V_f^1$. If $|V_f^+|=1$, then $f$ is both maximal bn-independent and maximal h-independent by part (i) of Propositions \ref{prop:necsufbnindependent} and \ref{prop:nec_suf_hearing}. 

Otherwise, suppose $|V_f^1|\geq 2$ and let $v\in V_f^1$. Then $B_f(v)-PB_f(v)\neq \emptyset$ if and only if $v$ has a vertex on its boundary adjacent to another vertex broadcasting at strength 1, in other words, Propositions \ref{prop:necsufbnindependent} (ii) is equivalent to Proposition \ref{prop:nec_suf_hearing} (ii). Therefore $f$ is both maximal bn-independent and maximal h-independent. 
\end{proof}

\subsection{Trees}

Let $\ell(P)$ denote the length of the path $P$. The following result is a consequence of Proposition \ref{prop:marchessaultpath}, and is stated here for clarity. 

\begin{corollary}
\label{cor:connected_tree_case}
Suppose $f$ is a broadcast on a tree $T$ such that $T-U_f^E$ is connected. Then $\sigma(f) \geq \text{rad}(T)$.
\end{corollary}

\begin{proof}
If $T-U_f^E$ is connected, then every edge of $T$ is covered by $f$, otherwise the removal of an uncovered edge would disconnect the tree. Let $D$ be a diametrical path of $T$. By part 2 of Proposition \ref{prop:marchessaultpath} with $D=P$, $\sigma(f)\geq\left \lceil \frac{\ell(D)}{2}\right \rceil \geq \text{rad}(T)$. 
\end{proof}

As in Proposition \ref{prop:marchessaultpath}, given a path $P$ in a tree, let $Touch(P)$ denote the set of broadcasting vertices whose $f$-neighbourhoods intersect $P$. Recall that for $v\in Touch(P)$, we use $d(v, P)$ to denote the minimum distance from $v$ to a vertex on $P$.

\begin{proposition}
\label{prop:b_plus_k}
Let $P$ be a path of a tree $T$ and let $f$ be a broadcast on $T$. If $f$ covers $b$ edges of $P$, and $k$ edges are covered more
than once, then $\sigma(f)\geq\left\lceil \frac{b+k}{2}\right\rceil $.
\end{proposition}

\begin{proof}
Consider $v\in Touch(P)$ and let $u$ be the vertex on $P$ for which the distance to $v$ is smallest (possibly, $u=v$). Since  $T$ is a tree, there exists a unique $u-v$ path $P_{uv}$ which, by choice of $u$, intersects $P$ only on $u$. Thus, $v$ covers at most $2(f(v)-d(u, v))$ edges of $P$. It follows that 

\begin{center}
    $b+k\leq \summ_{v\in Touch(P)}2(f(v)-d(v, P))\leq \summ_{v\in Touch(P)}2f(v)\leq \summ_{v\in V_f^+}2f(v)$,
\end{center}

\noindent hence $\sigma(f)\geq \frac{b+k}{2}$. As $\sigma(f)$ is an integer, we have that $\sigma(f)\geq \lceil \frac{b+k}{2}\rceil $. \end{proof}

We show next that Theorem \ref{theorem:ibnih_gen2} holds for trees.

\begin{theorem}
For any tree $T$, $i_{bn}(T)\leq i_h(T)$.
\label{theorem:ibnihtrees}
\end{theorem}

\begin{proof}
Suppose $T$ is a tree such that $i_{h}(T)<i_{\mathrm{bn}}(T)$ and let $f$ be
an $i_{h}$-broadcast on $T$. By Corollary \ref{cor:connected_tree_case},  if $f$ covers every edge of $T$, then $\sigma(f)\geq \text{rad}(T)$. %%
Since $i_{bn}(T)\leq \text{rad}(T)$, the cost of $f$ must be strictly less than $\text{rad}(T)$, hence some edge of $T$ is uncovered. In particular, $T-U_f^E$ contains at least two components.

Let $T_1, T_2, ..., T_k$ be the components of $T-U_f^E$ and let $f_i$ denote the restriction of $f$ to $T_i$. By Proposition \ref{prop:marchessault2broadcasts}, since $f$ is maximal h-independent, each component $T_i$ contains at least two broadcasting vertices. Hence, if $f_{i}$ is
bn-independent, then Proposition \ref{prop:marchessault2broadcasts} implies that it is maximal bn-independent. Since $i_{h}%
(T)<i_{\mathrm{bn}}(T)$, at least one restricted broadcast $f_{i}$ is not bn-independent. 

Assume without loss of generality that $f_{1}$ is not bn-independent on $T_1$. Then since no edge of $T_1$ is uncovered, at least one edge hears more than one broadcasting vertex. If this edge lies along a diametrical path of $T_1$, then 
$\sigma(f_{1})\geq\left\lceil \frac{\operatorname{diam}(T_{1})+1}%
{2}\right\rceil $
by  Proposition \ref{prop:b_plus_k}. If no edge along the diametrical path is covered by multiple broadcasts, then some vertex off the diametrical path is broadcasting. By part 2 of Proposition \ref{prop:marchessaultpath}, we again have that
$\sigma(f_{1})\geq\left\lceil \frac{\operatorname{diam}(T_{1})}%
{2}\right\rceil+1\geq \left\lceil \frac{\operatorname{diam}(T_{1})+1}%
{2}\right\rceil$. 

Since $T-U_f^E$ has at least two components, for some $i\neq 1$, there exists $y\in V(T_i)$ and $x\in V(T_1)$ such that $xy\in E(T)$. If $\left\lceil \frac{\operatorname{diam}(T_{1})+1}%
{2}\right\rceil > \text{rad}(T_1)$, then $d(c, x)\leq \text{rad}(T_1)< \sigma(f_1)$ for any central vertex $c$ of $T_1$.
In the case where $\left\lceil \frac{\operatorname{diam}(T_{1})+1}%
{2}\right\rceil =\text{rad}(T_1)$, such as illustrated in Figure $\ref{fig:odd_diam_tree}$, $\text{diam}(T_1)$ is odd, and so we may choose a central vertex $c$ such that $d(c, x)<\text{rad}(T_1)$. 
Let 
$g_1$ be the broadcast on $T$ defined by $g_{1}=(f-f_{1})\cup\{(c, \sigma(f_1))\}$.
Observe that
$V(T_{1})\subseteq N_{g_{1}}(c)$ and, by choice of $c$, the vertex $y$ hears $g_1$ from $c$. 

\begin{figure}[H]
    \centering

\tikzset{every picture/.style={line width=0.75pt}} %set default line width to 0.75pt        

\begin{tikzpicture}[x=0.75pt,y=0.75pt,yscale=-1,xscale=1]

%Straight Lines [id:da6673567170114285] 
\draw    (120,150) -- (160,150) ;
\draw [shift={(120,150)}, rotate = 0] [color={rgb, 255:red, 0; green, 0; blue, 0 }  ][fill={rgb, 255:red, 0; green, 0; blue, 0 }  ][line width=0.75]      (0, 0) circle [x radius= 3.35, y radius= 3.35]   ;
%Straight Lines [id:da6934539354146196] 
\draw    (280,150) -- (320,150) ;
\draw [shift={(280,150)}, rotate = 0] [color={rgb, 255:red, 0; green, 0; blue, 0 }  ][fill={rgb, 255:red, 0; green, 0; blue, 0 }  ][line width=0.75]      (0, 0) circle [x radius= 3.35, y radius= 3.35]   ;
%Straight Lines [id:da33958529279273364] 
\draw    (240,150) -- (280,150) ;
\draw [shift={(240,150)}, rotate = 0] [color={rgb, 255:red, 0; green, 0; blue, 0 }  ][fill={rgb, 255:red, 0; green, 0; blue, 0 }  ][line width=0.75]      (0, 0) circle [x radius= 3.35, y radius= 3.35]   ;
%Straight Lines [id:da9352935213759415] 
\draw    (200,150) -- (240,150) ;
\draw [shift={(200,150)}, rotate = 0] [color={rgb, 255:red, 0; green, 0; blue, 0 }  ][fill={rgb, 255:red, 0; green, 0; blue, 0 }  ][line width=0.75]      (0, 0) circle [x radius= 3.35, y radius= 3.35]   ;
%Straight Lines [id:da8146262103029984] 
\draw    (160,148.5) -- (200,148.5)(160,151.5) -- (200,151.5) ;
\draw [shift={(160,150)}, rotate = 0] [color={rgb, 255:red, 0; green, 0; blue, 0 }  ][fill={rgb, 255:red, 0; green, 0; blue, 0 }  ][line width=0.75]      (0, 0) circle [x radius= 3.35, y radius= 3.35]   ;
%Straight Lines [id:da7470314766752357] 
\draw  [dash pattern={on 0.84pt off 2.51pt}]  (320,150) -- (360,150) ;
\draw [shift={(360,150)}, rotate = 0] [color={rgb, 255:red, 0; green, 0; blue, 0 }  ][fill={rgb, 255:red, 0; green, 0; blue, 0 }  ][line width=0.75]      (0, 0) circle [x radius= 3.35, y radius= 3.35]   ;
\draw [shift={(320,150)}, rotate = 0] [color={rgb, 255:red, 0; green, 0; blue, 0 }  ][fill={rgb, 255:red, 0; green, 0; blue, 0 }  ][line width=0.75]      (0, 0) circle [x radius= 3.35, y radius= 3.35]   ;
%Straight Lines [id:da2967351458992744] 
\draw    (110,110) -- (120,150) ;
\draw [shift={(110,110)}, rotate = 75.96] [color={rgb, 255:red, 0; green, 0; blue, 0 }  ][fill={rgb, 255:red, 0; green, 0; blue, 0 }  ][line width=0.75]      (0, 0) circle [x radius= 3.35, y radius= 3.35]   ;
%Straight Lines [id:da3669752722522672] 
\draw    (130,110) -- (120,150) ;
\draw [shift={(130,110)}, rotate = 104.04] [color={rgb, 255:red, 0; green, 0; blue, 0 }  ][fill={rgb, 255:red, 0; green, 0; blue, 0 }  ][line width=0.75]      (0, 0) circle [x radius= 3.35, y radius= 3.35]   ;
%Straight Lines [id:da952040294640057] 
\draw    (240,110) -- (240,150) ;
\draw [shift={(240,110)}, rotate = 90] [color={rgb, 255:red, 0; green, 0; blue, 0 }  ][fill={rgb, 255:red, 0; green, 0; blue, 0 }  ][line width=0.75]      (0, 0) circle [x radius= 3.35, y radius= 3.35]   ;
%Straight Lines [id:da1572994604125586] 
\draw    (240,70) -- (240,110) ;
\draw [shift={(240,70)}, rotate = 90] [color={rgb, 255:red, 0; green, 0; blue, 0 }  ][fill={rgb, 255:red, 0; green, 0; blue, 0 }  ][line width=0.75]      (0, 0) circle [x radius= 3.35, y radius= 3.35]   ;
%Straight Lines [id:da31556236187624465] 
\draw    (40,150) -- (80,150) ;
\draw [shift={(40,150)}, rotate = 0] [color={rgb, 255:red, 0; green, 0; blue, 0 }  ][fill={rgb, 255:red, 0; green, 0; blue, 0 }  ][line width=0.75]      (0, 0) circle [x radius= 3.35, y radius= 3.35]   ;
%Straight Lines [id:da6545913646733184] 
\draw    (80,150) -- (120,150) ;
\draw [shift={(80,150)}, rotate = 0] [color={rgb, 255:red, 0; green, 0; blue, 0 }  ][fill={rgb, 255:red, 0; green, 0; blue, 0 }  ][line width=0.75]      (0, 0) circle [x radius= 3.35, y radius= 3.35]   ;
%Straight Lines [id:da09679684680570921] 
\draw    (362.02,151.21) -- (410,180) ;
\draw [shift={(410,180)}, rotate = 30.96] [color={rgb, 255:red, 0; green, 0; blue, 0 }  ][fill={rgb, 255:red, 0; green, 0; blue, 0 }  ][line width=0.75]      (0, 0) circle [x radius= 3.35, y radius= 3.35]   ;
\draw [shift={(360,150)}, rotate = 30.96] [color={rgb, 255:red, 0; green, 0; blue, 0 }  ][line width=0.75]      (0, 0) circle [x radius= 3.35, y radius= 3.35]   ;
%Straight Lines [id:da8848679840101226] 
\draw    (420,90) -- (361.66,148.34) ;
\draw [shift={(360,150)}, rotate = 135] [color={rgb, 255:red, 0; green, 0; blue, 0 }  ][line width=0.75]      (0, 0) circle [x radius= 3.35, y radius= 3.35]   ;
\draw [shift={(420,90)}, rotate = 135] [color={rgb, 255:red, 0; green, 0; blue, 0 }  ][fill={rgb, 255:red, 0; green, 0; blue, 0 }  ][line width=0.75]      (0, 0) circle [x radius= 3.35, y radius= 3.35]   ;

% Text Node
\draw (171,172.4) node [anchor=north west][inner sep=0.75pt]    {$T_{1}$};
% Text Node
\draw (389,142.4) node [anchor=north west][inner sep=0.75pt]    {$T_{i}$};
% Text Node
\draw (127,133.4) node [anchor=north west][inner sep=0.75pt]  [font=\small]  {$2$};
% Text Node
\draw (245,133.4) node [anchor=north west][inner sep=0.75pt]  [font=\small]  {$2$};
% Text Node
\draw (197,154.4) node [anchor=north west][inner sep=0.75pt]  [font=\small]  {$c$};
% Text Node
%\draw (157,152.4) node [anchor=north west][inner sep=0.75pt]  [font=\small]  {$c'$};
% Text Node
\draw (351,154.4) node [anchor=north west][inner sep=0.75pt]  [font=\small]  {$y$};
% Text Node
\draw (311,154.4) node [anchor=north west][inner sep=0.75pt]  [font=\small]  {$x$};

\end{tikzpicture}

    \caption{A component $T_1$ with $\sigma(f_1)=\text{rad}(T_1)=4$. A broadcast of strength 4 from $c$ covers the edge $xy$ joining $T_1$ to $T_i$.}
    \label{fig:odd_diam_tree}
\end{figure}
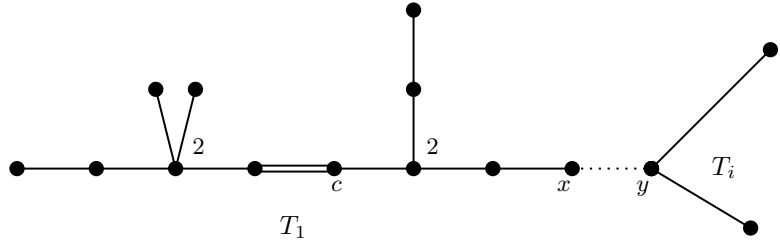

Let $G_{2}$ be the component of $T-U_{g_{1}}^{E}$
that contains $T_{1}$. As $y$ hears $g_1$ from a vertex in $T_1$, $G_2$ must also contain the component of $T-U_f^E$ containing $y$. It follows that $T-U_{g_1}^E$ has fewer components than $T-U_f^E$.

Let $h_{2}$ be the restriction of $g_{1}$ to $G_{2}$. Note that $h_{2}$ covers all edges of $G_{2}$, and all other components of
$T-U_{g_{1}}^{E}$ are trees $T_{i}$. 
In particular, each component of $T-U_{g_{1}}^{E}$
contains at least two vertices of $V_{g_{1}}^{+}$. Since $\sigma
(g_{1})=\sigma(f)<i_{\mathrm{bn}}(T)$, $g_{1}$ is not bn-independent, and so there exists at least one component of $T-U_{g_{1}}^{E}$ containing an edge that
hears two broadcasting vertices. Assume without loss of generality that $G_{2}$ contains such an edge.
Repeating the process, we again have that $\sigma(h_{2})\geq\left\lceil \frac{\operatorname{diam}(G_{2})+1}%
{2}\right\rceil $. 

As $T$ is finite, we may repeat the process until we eventually
obtain a broadcast $g_{\ell}$ on $T$ such that, for some central vertex $c$ of
$T$, $\sigma(g_{\ell})=g_{\ell}(c)=\sigma(f)=\operatorname{rad}(T)$, 
a contradiction. \end{proof}

\subsection{Spanning trees}

Our aim in this subsection is to show that the lower broadcast independent domination number of a graph $G$ is given by the minimum of this parameter among all spanning trees of $G$.
Recall that a vertex of $G$ is \textit{peripheral} if its eccentricity equals the diameter of $G$. Two peripheral vertices $p$ and $q$ are said to be \textit{antipodal} if $d(p, q)=\text{diam}(G)$.  

\begin{proposition}

\label{Prop_Rad}Let $f$ be a dominating broadcast on a connected graph $G$. If
$G-U_f^E$ is
connected, then $\sigma(f)\geq\operatorname{rad}(G)$. If, in addition, an edge
of $G$ hears a broadcast from more than one vertex, then for any
peripheral vertex $p$ of
$G$, there exists a dominating broadcast $g_{p}$ on
$G$ such that $\sigma(g_{p})=\sigma(f)$, $|V_{g_{p}}^{+}|=1$ and
$g_{p}$ overdominates $p$.

\end{proposition}

\begin{proof}
If $f$ is a dominating broadcast such that $|V_f^+|=1$, then $\sigma(f)\geq \text{rad}(G)$, 
so suppose $|V_{f}^{+}|\geq2$. Suppose $H_0=G-U_f^E$ is connected and let $p$ be an arbitrarily chosen peripheral vertex of $G$. Our goal is to define a sequence of equal-cost dominating broadcasts $f_0=f, f_1, f_2, ..., f_k=g_p$ on $G$ such that $|V_{f_{1}}^{+}|>|V_{f_{2}}^{+}|>\cdots>|V_{f_{k}}^{+}|=1$.

By careful construction of broadcasts, we will ensure that if an edge of $G$ hears $f$ from more than one broadcasting vertex, then some broadcast $f_i$ in the sequence overdominates $p$. Furthermore, we will show that if $p$ is overdominated in $f_i$, then $p$ is overdominated in $f_j$ for all $i\leq j\leq k$, such that the resulting radial broadcast $f_k$ overdominates $p$.

Let $u,w\in V_{f}^{+}$ such that $N_{f}(u)\cap N_{f}(w)\neq\emptyset$ and
let $P_{uw}$ be a $u-w$ geodesic in $H_0$. Then
$\ell(P_{uw})\leq f(u)+f(w)$. If an edge hears $f$ from both $u$ and
$w$, then $P_{uw}$ contains such an edge, in which case $\ell(P_{uw})<f(u)+f(w)$. 

If $p\in N_{f}(u)\cup N_{f}(w)$, assume without loss of generality that $p\in N_f(w)$.
Otherwise, if $p\notin N_{f}(u)\cup N_{f}(w)$, then $|V_f^+|=3$ since $f$ is dominating. Since $H_0$ is connected, some
vertex in $x\in N_{f}(u)\cup N_{f}(w)$ hears a vertex $y\in V_{f}^{+}%
\setminus\{u,w\}$. Assume without loss of generality that $x\in
N_{f}(w)$.

Let $b_1$ be the vertex on $P_{uw}$ at distance $f(w)$ from $u$ (and hence distance at most $f(u)$ from $w$). Observe that if $\ell(P_{uw})<f(u)+f(w)$, then $b_1$ is at distance at most $f(u)-1$ from $w$. 
Define the broadcast $f_{1}$ by 
\[
f_{1}(v)=\left\{
\begin{tabular}
[c]{ll}%
$f(v)$ & if $v\in V(G)\setminus\{u,w,b_{1}\}$\\
$0$ & if $v\in\{u,w\}$\\
$f(u)+f(w)$ & if $v=b_{1}.$%
\end{tabular}
\right.
\]

Clearly, $\sigma(f_{1})=\sigma(f)$. Since $f$ is dominating, to prove $f_{1}$ is dominating, it suffices to
show that each vertex in $N_{f}(u)\cup N_{f}(w)$ hears $f_{1}$.
For any $v\in N_{f}(u)$,
\[
d(v,b_{1})\leq d(v,u)+d(u,b_{1})\leq f(u)+f(w)=f_{1}(b_{1}),
\]
\noindent hence $v$ hears $f_{1}$ from $b_{1}$. Similarly, any vertex in $N_{f}(w)$
hears $b_{1}$. 

Let $H_1=G-U_{f_1}^E$. Note that $f_{1}$ covers all edges of $H_{1}$, 
and if $v\in H_0$ is overdominated by $f$, then $v$ is overdominated by $f_1$ in $H_1$.

We repeat the above procedure. At each step $i$, we define a dominating broadcast $f_i$ and a graph $H_i=G-U_{f_i}^E$ such that $\sigma(f_i)=\sigma(f_{i-1})$,  $|V_{f_i}^+|<|V_{f_{i-1}}^+|$, 
and if a vertex $v$ is overdominated under $f_{i-1}$, then $v$ is overdominated in $f_i$. Finally, if $H_{i-1}$ is connected, so is $H_i$. Thus, if $|V_{f_i}^+|\geq 2$, we may always find broadcasting vertices $u, w\in V_{f_i}^+$ such that $f_i(u)+f_i(w)\leq d(u, w)$. 

Let $k$ be the first $i$ such that $|V_{f_k}^+|=1$ and consider $f_{k-1}$. Let $V_{f_{k-1}}^+=\{u, w\}$, and assume without loss of generality that $p\in N_f(w)$. If an edge of $H_{k-1}$ hears $f_{k-1}$ from more than one broadcasting vertex, then such an edge lies along a $u-w$ geodesic $P_{uw}$. By definition of $f_k$, 

\begin{center}
    $d(p, b_k)\leq d(p, w)+d(w, b_k)<f_{k-1}(w)+f_{k-1}(u)=f_k(b_k)$,
\end{center}

\noindent hence $p$ is overdominated by $b_k$.

Let $g_p=f_k$. Since $f_{k}$ dominates $H_k$, and since $H_k$ is a connected spanning subgraph of $G$, we have that 
\[
\sigma(f)=\sigma(f_{k})\geq \text{rad}(H_k)\geq
\operatorname{rad}(G),
\]

\noindent therefore $g_p$ is a dominating broadcast on $G$. Furthermore, if an edge of $G$ hears $f$ from more than one broadcasting vertex, then $g_p$ overdominates $p$. 
\end{proof}

We proceed to show that the lower boundary independence number of an arbitrary connected graph $G$ equals the minimum lower boundary independence number among those of its spanning trees.

\begin{theorem}
\label{theorem:spanning_trees_ibn}
For any connected graph $G$,
\begin{center}
    $i_{bn}(G)=\min\{i_{bn}(T)\,:\,T\text{ is a spanning tree of G}\}$.
\end{center}
\end{theorem}

\begin{proof}

Suppose there exists a tree $T$ spanning $G$ such that ${i_{\mathrm{bn}}(T)<i_{\mathrm{bn}}(G)}$, and let $f$ be an $i_{\mathrm{bn}}%
$-broadcast on $T$. 

If $|V_f^+|=1$, by the minimality of $i_{bn}(T)$, $f(v)=\text{rad}(T)=\sigma(f)$ for some central vertex $v$. Therefore $i_{bn}(T)=\text{rad}(T)\geq \text{rad}(G)\geq i_{bn}(G)$, a contradiction. Suppose instead that $|V_f^+|\geq 2$ and the set of uncovered edges $U_f^E$ is nonempty. Then $T-U_f^E$ is disconnected. By Proposition \ref{prop:marchessault2broadcasts}, every component of $T-U_f^E$ contains at least two broadcasting vertices.

Since $f$ is maximal, $f$ is a dominating broadcast on $G$. However, since $i_{bn}(T)<i_{bn}(G)$, $f$ is not bn-independent on $G$, and so $G$ contains an edge that hears $f$ from more than one broadcasting vertex. 

Arbitrarily add edges of $G-T$ to $T$ until the addition of some edge, say $e_1$, results in a spanning subgraph $G_1$ of $G$ such that $f$ is not bn-independent on $G_1$. 
Let $H_{1}$ be the component of $G_{1}-U_{f}^{E}$ that contains $e_{1}$, and let
$h_{1}$ be the restriction of $f$ to $H_{1}$. By Proposition \ref{Prop_Rad}, there exists a dominating broadcast $g_{1}$ on $H_{1}$ such that $\sigma(g_1)=\sigma(h_1)$ and $V_{g_1}^+=\{b_1\}$ for some $b_1\in V(H_1)$. In particular, if $G_1-U_f^E$ contains more than one component, $g_1$ may be constructed such that $b_1$ covers an edge joining $H_1$ to different component. 

Define a new broadcast $f_1$ on $G$ by 

\[
f_{1}(x)=\left\{
\begin{tabular}
[c]{rl}%
$f(x)$ & if $x\in V(G)\setminus V(H_1)$
\\
$0$ & if $x\in V(H_1) -b_1$
\\
$\sigma(g_1)$ & if $x=b_{1}.$%
\end{tabular}
\right.
\]

Then $\sigma(f_1)=\sigma(f)$, and $G_1-U_{f_1}^E$ has fewer components than $G_1-U_{f}^E$. If $f_1$ is bn-independent on $G_1$, since $\sigma(f_1)=\sigma(f)<i_{bn}(G)$, we may continue adding edges of $G-G_1$ to $G_1$ until the addition of $e_2$ results in a spanning subgraph $G_2$ of $G$ such that $f_1$ is not bn-independent on $G_2$. 

By Proposition \ref{Prop_Rad}, we may repeat the process until we obtain a broadcast $f_k$ such that $\sigma(f)=\sigma(f_k)$ and $G-U_{f_k}^E$ has only one component.
Since the resulting broadcast is dominating and the spanning subgraph of $G$ induced by the set of covered edges is connected, there exists a dominating broadcast $f'$ on $G$ such that $\sigma(f')=\text{rad}(G)$ and $|V_{f'}^+|=1$. But then $f'$ is maximal bn-independent on $G$, a contradiction.

It remains to show that there exists a tree $T$ spanning $G$ such that $i_{bn}(T)=i_{bn}(G)$. Let $f$ be an $i_{\mathrm{bn}}$-broadcast on $G$ and suppose $V_{f}^{+}=\{v_{1},...,v_{k}\}$. For $i=1,...,k$, consider the subgraph
$G_{i}$ of $G$ induced by $N_{f}(v_{i})$. If $G_i$ is acyclic, let $T_i=G_i$; otherwise, successively delete edges from cycles lying in $N_f(v_i)$ to obtain a spanning tree $T_i$. For each such cycle, we may always choose an edge furthest from $v_i$ such that the remaining edges are covered by $v_i$. Thus, the restriction of $f$ to
$T_{i}$, denoted $f_i$, covers all edges of $T_{i}$. Let $H$ be the subgraph of $G$ induced by
$\bigcup_{i=1}^{k}E(T_{i})$. 

Suppose $H$ contains a cycle $C$. By construction, the edges of $C$ are covered by a set of broadcasting vertices $V_C\subseteq V_f^+$ such that $|V_C|\geq 2$. Observe that each $v_i\in V_C$ covers an even number of edges on $C$. In particular, there exist $v_i, v_j\in V_C$ such that $B_f(v_i)\cap  B_f(v_j)$ contains a vertex $b\in V(C)$. Let $a$ be the vertex on $C$ adjacent to $b$ in $T_i$ and let $H_1 = H-ab$. Since $B_f(v_i)\cap C$ contains at least two vertices, there exists $b'\in V(B_f(v_i)\cap C)\setminus \{b\}$ lying on the boundary of another broadcasting vertex in $H_1$. The same holds for $v_j$. Thus, $f$ is maximal bn-independent on $H_1$.

\begin{figure}[H]
    \centering

\tikzset{every picture/.style={line width=0.75pt}} %set default line width to 0.75pt        

\begin{tikzpicture}[x=0.75pt,y=0.75pt,yscale=-1,xscale=1]

%Straight Lines [id:da23405377479387113] 
\draw    (170,210) -- (200,180) ;
\draw [shift={(170,210)}, rotate = 315] [color={rgb, 255:red, 0; green, 0; blue, 0 }  ][fill={rgb, 255:red, 0; green, 0; blue, 0 }  ][line width=0.75]      (0, 0) circle [x radius= 3.35, y radius= 3.35]   ;
%Straight Lines [id:da6777376652395912] 
\draw    (140,180) -- (170,210) ;
\draw [shift={(140,180)}, rotate = 45] [color={rgb, 255:red, 0; green, 0; blue, 0 }  ][fill={rgb, 255:red, 0; green, 0; blue, 0 }  ][line width=0.75]      (0, 0) circle [x radius= 3.35, y radius= 3.35]   ;
%Straight Lines [id:da4239063036655317] 
\draw    (140,140) -- (140,180) ;
\draw [shift={(140,140)}, rotate = 90] [color={rgb, 255:red, 0; green, 0; blue, 0 }  ][fill={rgb, 255:red, 0; green, 0; blue, 0 }  ][line width=0.75]      (0, 0) circle [x radius= 3.35, y radius= 3.35]   ;
%Straight Lines [id:da7555741638022784] 
\draw    (200,100) -- (170,70) ;
\draw [shift={(200,100)}, rotate = 225] [color={rgb, 255:red, 0; green, 0; blue, 0 }  ][fill={rgb, 255:red, 0; green, 0; blue, 0 }  ][line width=0.75]      (0, 0) circle [x radius= 3.35, y radius= 3.35]   ;
%Straight Lines [id:da11699306942937837] 
\draw  [dash pattern={on 4.5pt off 4.5pt}]  (170,70) -- (140,100) ;
\draw [shift={(170,70)}, rotate = 135] [color={rgb, 255:red, 0; green, 0; blue, 0 }  ][fill={rgb, 255:red, 0; green, 0; blue, 0 }  ][line width=0.75]      (0, 0) circle [x radius= 3.35, y radius= 3.35]   ;
%Straight Lines [id:da1918361542383975] 
\draw    (140,100) -- (140,140) ;
\draw [shift={(140,100)}, rotate = 90] [color={rgb, 255:red, 0; green, 0; blue, 0 }  ][fill={rgb, 255:red, 0; green, 0; blue, 0 }  ][line width=0.75]      (0, 0) circle [x radius= 3.35, y radius= 3.35]   ;
%Straight Lines [id:da9043290345978887] 
\draw    (200,180) -- (200,140) ;
\draw [shift={(200,180)}, rotate = 270] [color={rgb, 255:red, 0; green, 0; blue, 0 }  ][fill={rgb, 255:red, 0; green, 0; blue, 0 }  ][line width=0.75]      (0, 0) circle [x radius= 3.35, y radius= 3.35]   ;
%Straight Lines [id:da8561610330612448] 
\draw    (200,140) -- (200,100) ;
\draw [shift={(200,140)}, rotate = 270] [color={rgb, 255:red, 0; green, 0; blue, 0 }  ][fill={rgb, 255:red, 0; green, 0; blue, 0 }  ][line width=0.75]      (0, 0) circle [x radius= 3.35, y radius= 3.35]   ;
%Straight Lines [id:da5675707029370713] 
\draw    (320,210) -- (350,180) ;
\draw [shift={(320,210)}, rotate = 315] [color={rgb, 255:red, 0; green, 0; blue, 0 }  ][fill={rgb, 255:red, 0; green, 0; blue, 0 }  ][line width=0.75]      (0, 0) circle [x radius= 3.35, y radius= 3.35]   ;
%Straight Lines [id:da03432301837345264] 
\draw    (290,180) -- (320,210) ;
\draw [shift={(290,180)}, rotate = 45] [color={rgb, 255:red, 0; green, 0; blue, 0 }  ][fill={rgb, 255:red, 0; green, 0; blue, 0 }  ][line width=0.75]      (0, 0) circle [x radius= 3.35, y radius= 3.35]   ;
%Straight Lines [id:da4087641169930678] 
\draw    (290,140) -- (290,180) ;
\draw [shift={(290,140)}, rotate = 90] [color={rgb, 255:red, 0; green, 0; blue, 0 }  ][fill={rgb, 255:red, 0; green, 0; blue, 0 }  ][line width=0.75]      (0, 0) circle [x radius= 3.35, y radius= 3.35]   ;
%Straight Lines [id:da19470731248612938] 
\draw    (350,100) -- (320,70) ;
\draw [shift={(350,100)}, rotate = 225] [color={rgb, 255:red, 0; green, 0; blue, 0 }  ][fill={rgb, 255:red, 0; green, 0; blue, 0 }  ][line width=0.75]      (0, 0) circle [x radius= 3.35, y radius= 3.35]   ;
%Straight Lines [id:da6823686323242277] 
\draw  [dash pattern={on 4.5pt off 4.5pt}]  (320,70) -- (290,100) ;
\draw [shift={(320,70)}, rotate = 135] [color={rgb, 255:red, 0; green, 0; blue, 0 }  ][fill={rgb, 255:red, 0; green, 0; blue, 0 }  ][line width=0.75]      (0, 0) circle [x radius= 3.35, y radius= 3.35]   ;
%Straight Lines [id:da8509083123109087] 
\draw    (290,100) -- (290,140) ;
\draw [shift={(290,100)}, rotate = 90] [color={rgb, 255:red, 0; green, 0; blue, 0 }  ][fill={rgb, 255:red, 0; green, 0; blue, 0 }  ][line width=0.75]      (0, 0) circle [x radius= 3.35, y radius= 3.35]   ;
%Straight Lines [id:da08498291798581303] 
\draw    (350,180) -- (350,140) ;
\draw [shift={(350,180)}, rotate = 270] [color={rgb, 255:red, 0; green, 0; blue, 0 }  ][fill={rgb, 255:red, 0; green, 0; blue, 0 }  ][line width=0.75]      (0, 0) circle [x radius= 3.35, y radius= 3.35]   ;
%Straight Lines [id:da2127244446331651] 
\draw    (350,140) -- (350,100) ;
\draw [shift={(350,140)}, rotate = 270] [color={rgb, 255:red, 0; green, 0; blue, 0 }  ][fill={rgb, 255:red, 0; green, 0; blue, 0 }  ][line width=0.75]      (0, 0) circle [x radius= 3.35, y radius= 3.35]   ;
%Straight Lines [id:da8552566656129228] 
\draw    (100,140) -- (140,140) ;
\draw [shift={(100,140)}, rotate = 0] [color={rgb, 255:red, 0; green, 0; blue, 0 }  ][fill={rgb, 255:red, 0; green, 0; blue, 0 }  ][line width=0.75]      (0, 0) circle [x radius= 3.35, y radius= 3.35]   ;

% Text Node
\draw (207,132.4) node [anchor=north west][inner sep=0.75pt]  [font=\small]  {$2$};
% Text Node
\draw (101,122.4) node [anchor=north west][inner sep=0.75pt]  [font=\small]  {$3$};
% Text Node
\draw (157,56.4) node [anchor=north west][inner sep=0.75pt]  [font=\footnotesize]  {$b$};
% Text Node
\draw (307,56.4) node [anchor=north west][inner sep=0.75pt]  [font=\footnotesize]  {$b$};
% Text Node
\draw (167,216.4) node [anchor=north west][inner sep=0.75pt]  [font=\footnotesize]  {$b'$};
% Text Node
\draw (271,132.4) node [anchor=north west][inner sep=0.75pt]  [font=\small]  {$2$};
% Text Node
\draw (351,82.4) node [anchor=north west][inner sep=0.75pt]  [font=\small]  {$1$};
% Text Node
\draw (357,174.4) node [anchor=north west][inner sep=0.75pt, shape=rectangle]  [font=\small]  {$1$};
% Text Node
\draw (277,86.4) node [anchor=north west][inner sep=0.75pt]  [font=\footnotesize]  {$a$};
% Text Node
\draw (127,86.4) node [anchor=north west][inner sep=0.75pt]  [font=\footnotesize]  {$a$};
% Text Node
\draw (317,216.4) node [anchor=north west][inner sep=0.75pt]  [font=\footnotesize]  {$b'$};

\end{tikzpicture}

    \caption{Two cycles whose edges are covered by multiple broadcasting vertices. We may always remove an edge $ab$ from such a cycle without violating the maximal boundary independence condition.}
    \label{fig:my_label123}
\end{figure}
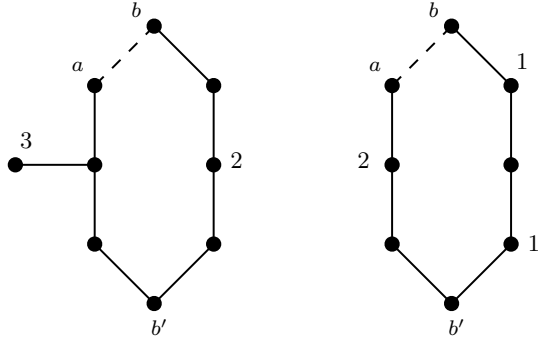

If $H_1$ contains a cycle, repeat the process, successively removing edges from cycles until the resulting graph $H_r$ is acyclic. If $H_r$ is connected, let $T=H_r$. Otherwise, since $G$ is connected, we may add edges of $G-H_r$ to $H_r$ joining components of $H_r$ without creating cycles until we obtain a tree $T$ spanning $G$. 

Since $f$ is maximal bn-independent on $H_r$, the construction ensures that $f$ is a maximal bn-independent broadcast on $T$, as $B_f(v)-PB_f(v)\neq \emptyset$ in $H_r$ for each $v\in V_f^+$. Hence $i_{\mathrm{bn}}(T)\leq\sigma(f)=i_{\mathrm{bn}}(G)$. But we have
already shown that $i_{\mathrm{bn}}(G)\leq i_{\mathrm{bn}}(T^{\prime})$ for
any spanning tree $T^{\prime}$ of $G$. Consequently, $i_{\mathrm{bn}%
}(G)=i_{\mathrm{bn}}(T)=\min\{i_{\mathrm{bn}}(T^{\prime}):T^{\prime}$ is a
spanning tree of $G\}$. 
\end{proof}

\subsection{Proof of Theorem 4.1}

Theorem \ref{theorem:ibnihtrees} may be extended to general graphs by an application of Theorem \ref{theorem:spanning_trees_ibn}. Suppose again that there exists a graph $G$ such that $i_{h}(G)<i_{\mathrm{bn}}(G)$. Then
$i_{h}(G)<\operatorname{rad}(G)$. Let $f$ be a hearing independent broadcast
on $G$ such that $|V_{f}^{+}|\geq2$, and such that some
edge of $G$ hears at least two broadcasting vertices.

If $u$ and $v$ are vertices in $V_{f}^{+}$ such that $u$ is adjacent to a
vertex in $B_{f}(v)$ (Figure \ref{fig:certif_digraph}), we write $u\rightarrow v$ and say that $u$
\emph{provides a certificate that the broadcast cannot be increased at }$v$,
or, in short, that $u$ \emph{certifies }$v$. We now define two graphs and a
digraph associated with $G$ and $f$.

\begin{itemize}
\item The \emph{neighbourhood graph }$\mathcal{N}_f(G)$ has as its
vertex set the set $V_{f}^{+}$, and two vertices $u,v\in V_{f}^{+}$ are
adjacent in $\mathcal{N}_f(G)$ if and only if $N_{f}(u)\cap
N_{f}(v)\neq\emptyset.$

\item The \emph{certification digraph} $\mathcal{C}_f(G)$ has as its
vertex set the set $V_{f}^{+}$, and $(v,u)$ is an arc of $\mathcal{C}_f(G)$ if and only if $v\rightarrow u$. Note that if $(v,u)$ is an
arc of $\mathcal{C}_f(G)$, then $(u,v)$ may or may not be an arc as
well. We say that $(v,u)$ is a \emph{double arc }if $(u,v)$ is also an arc,
otherwise we say that $(v,u)$ is a \emph{single arc}. Note that if $(v,u)$ is
an arc of $\mathcal{C}_f(G)$, then $(v,u)$ is
a double arc if and only if $f(v)=f(u)$.
\end{itemize}

\begin{figure}[H]
    \centering

\tikzset{every picture/.style={line width=0.75pt}} %set default line width to 0.75pt        

\begin{tikzpicture}[x=0.75pt,y=0.75pt,yscale=-1,xscale=1]

%Straight Lines [id:da9372095291076055] 
\draw    (120,80) -- (105,105) ;
\draw [shift={(105,105)}, rotate = 120.96] [color={rgb, 255:red, 0; green, 0; blue, 0 }  ][fill={rgb, 255:red, 0; green, 0; blue, 0 }  ][line width=0.75]      (0, 0) circle [x radius= 3.35, y radius= 3.35]   ;
\draw [shift={(120,80)}, rotate = 120.96] [color={rgb, 255:red, 0; green, 0; blue, 0 }  ][fill={rgb, 255:red, 0; green, 0; blue, 0 }  ][line width=0.75]      (0, 0) circle [x radius= 3.35, y radius= 3.35]   ;
%Straight Lines [id:da21489420864536002] 
\draw    (105,105) -- (90,130) ;
\draw [shift={(90,130)}, rotate = 120.96] [color={rgb, 255:red, 0; green, 0; blue, 0 }  ][fill={rgb, 255:red, 0; green, 0; blue, 0 }  ][line width=0.75]      (0, 0) circle [x radius= 3.35, y radius= 3.35]   ;
\draw [shift={(105,105)}, rotate = 120.96] [color={rgb, 255:red, 0; green, 0; blue, 0 }  ][fill={rgb, 255:red, 0; green, 0; blue, 0 }  ][line width=0.75]      (0, 0) circle [x radius= 3.35, y radius= 3.35]   ;
%Straight Lines [id:da8739345107860024] 
\draw    (90,130) -- (75,155) ;
\draw [shift={(75,155)}, rotate = 120.96] [color={rgb, 255:red, 0; green, 0; blue, 0 }  ][fill={rgb, 255:red, 0; green, 0; blue, 0 }  ][line width=0.75]      (0, 0) circle [x radius= 3.35, y radius= 3.35]   ;
\draw [shift={(90,130)}, rotate = 120.96] [color={rgb, 255:red, 0; green, 0; blue, 0 }  ][fill={rgb, 255:red, 0; green, 0; blue, 0 }  ][line width=0.75]      (0, 0) circle [x radius= 3.35, y radius= 3.35]   ;
%Straight Lines [id:da8899757915285347] 
\draw    (75,155) -- (60,180) ;
\draw [shift={(60,180)}, rotate = 120.96] [color={rgb, 255:red, 0; green, 0; blue, 0 }  ][fill={rgb, 255:red, 0; green, 0; blue, 0 }  ][line width=0.75]      (0, 0) circle [x radius= 3.35, y radius= 3.35]   ;
\draw [shift={(75,155)}, rotate = 120.96] [color={rgb, 255:red, 0; green, 0; blue, 0 }  ][fill={rgb, 255:red, 0; green, 0; blue, 0 }  ][line width=0.75]      (0, 0) circle [x radius= 3.35, y radius= 3.35]   ;
%Straight Lines [id:da3359735601748879] 
\draw    (180,180) -- (165,155) ;
\draw [shift={(180,180)}, rotate = 239.04] [color={rgb, 255:red, 0; green, 0; blue, 0 }  ][fill={rgb, 255:red, 0; green, 0; blue, 0 }  ][line width=0.75]      (0, 0) circle [x radius= 3.35, y radius= 3.35]   ;
%Straight Lines [id:da3882176867152669] 
\draw    (165,155) -- (150,130) ;
\draw [shift={(165,155)}, rotate = 239.04] [color={rgb, 255:red, 0; green, 0; blue, 0 }  ][fill={rgb, 255:red, 0; green, 0; blue, 0 }  ][line width=0.75]      (0, 0) circle [x radius= 3.35, y radius= 3.35]   ;
%Straight Lines [id:da9858859521399965] 
\draw    (150,130) -- (135,105) ;
\draw [shift={(150,130)}, rotate = 239.04] [color={rgb, 255:red, 0; green, 0; blue, 0 }  ][fill={rgb, 255:red, 0; green, 0; blue, 0 }  ][line width=0.75]      (0, 0) circle [x radius= 3.35, y radius= 3.35]   ;
%Straight Lines [id:da1331903622953139] 
\draw    (135,105) -- (120,80) ;
\draw [shift={(135,105)}, rotate = 239.04] [color={rgb, 255:red, 0; green, 0; blue, 0 }  ][fill={rgb, 255:red, 0; green, 0; blue, 0 }  ][line width=0.75]      (0, 0) circle [x radius= 3.35, y radius= 3.35]   ;
%Straight Lines [id:da052730280201758184] 
\draw    (60,180) -- (100,180) ;
\draw [shift={(100,180)}, rotate = 0] [color={rgb, 255:red, 0; green, 0; blue, 0 }  ][fill={rgb, 255:red, 0; green, 0; blue, 0 }  ][line width=0.75]      (0, 0) circle [x radius= 3.35, y radius= 3.35]   ;
%Straight Lines [id:da8635809217681731] 
\draw    (290,80) -- (230,180) ;
\draw [shift={(230,180)}, rotate = 120.96] [color={rgb, 255:red, 0; green, 0; blue, 0 }  ][fill={rgb, 255:red, 0; green, 0; blue, 0 }  ][line width=0.75]      (0, 0) circle [x radius= 3.35, y radius= 3.35]   ;
\draw [shift={(290,80)}, rotate = 120.96] [color={rgb, 255:red, 0; green, 0; blue, 0 }  ][fill={rgb, 255:red, 0; green, 0; blue, 0 }  ][line width=0.75]      (0, 0) circle [x radius= 3.35, y radius= 3.35]   ;
%Straight Lines [id:da7800812826702253] 
\draw    (350,180) -- (290,80) ;
\draw [shift={(350,180)}, rotate = 239.04] [color={rgb, 255:red, 0; green, 0; blue, 0 }  ][fill={rgb, 255:red, 0; green, 0; blue, 0 }  ][line width=0.75]      (0, 0) circle [x radius= 3.35, y radius= 3.35]   ;
%Straight Lines [id:da7008051470751417] 
\draw    (230,180) -- (350,180) ;
%Straight Lines [id:da8958976592207606] 
\draw    (460,80) -- (400,180) ;
\draw [shift={(400,180)}, rotate = 120.96] [color={rgb, 255:red, 0; green, 0; blue, 0 }  ][fill={rgb, 255:red, 0; green, 0; blue, 0 }  ][line width=0.75]      (0, 0) circle [x radius= 3.35, y radius= 3.35]   ;
\draw [shift={(433.34,124.43)}, rotate = 120.96] [fill={rgb, 255:red, 0; green, 0; blue, 0 }  ][line width=0.08]  [draw opacity=0] (8.93,-4.29) -- (0,0) -- (8.93,4.29) -- cycle    ;
\draw [shift={(460,80)}, rotate = 120.96] [color={rgb, 255:red, 0; green, 0; blue, 0 }  ][fill={rgb, 255:red, 0; green, 0; blue, 0 }  ][line width=0.75]      (0, 0) circle [x radius= 3.35, y radius= 3.35]   ;
%Straight Lines [id:da5675925428117514] 
\draw    (520,180) -- (460,80) ;
\draw [shift={(487.43,125.71)}, rotate = 59.04] [fill={rgb, 255:red, 0; green, 0; blue, 0 }  ][line width=0.08]  [draw opacity=0] (8.93,-4.29) -- (0,0) -- (8.93,4.29) -- cycle    ;
\draw [shift={(520,180)}, rotate = 239.04] [color={rgb, 255:red, 0; green, 0; blue, 0 }  ][fill={rgb, 255:red, 0; green, 0; blue, 0 }  ][line width=0.75]      (0, 0) circle [x radius= 3.35, y radius= 3.35]   ;
%Straight Lines [id:da9725614457192302] 
\draw    (100,180) -- (140,180) ;
\draw [shift={(140,180)}, rotate = 0] [color={rgb, 255:red, 0; green, 0; blue, 0 }  ][fill={rgb, 255:red, 0; green, 0; blue, 0 }  ][line width=0.75]      (0, 0) circle [x radius= 3.35, y radius= 3.35]   ;
%Straight Lines [id:da8621909500849447] 
\draw    (140,180) -- (180,180) ;
%Straight Lines [id:da13275920539923813] 
\draw    (400,180) -- (520,180) ;
\draw [shift={(453.5,180)}, rotate = 0] [fill={rgb, 255:red, 0; green, 0; blue, 0 }  ][line width=0.08]  [draw opacity=0] (8.93,-4.29) -- (0,0) -- (8.93,4.29) -- cycle    ;
%Curve Lines [id:da13675393099679867] 
\draw    (400,180) .. controls (440,150) and (479,150) .. (520,180) ;
\draw [shift={(464.99,157.68)}, rotate = 181.73] [fill={rgb, 255:red, 0; green, 0; blue, 0 }  ][line width=0.08]  [draw opacity=0] (8.93,-4.29) -- (0,0) -- (8.93,4.29) -- cycle    ;

% Text Node
\draw (117,59.4) node [anchor=north west][inner sep=0.75pt]  [font=\footnotesize]  {$3$};
% Text Node
\draw (178,161.4) node [anchor=north west][inner sep=0.75pt]  [font=\footnotesize]  {$2$};
% Text Node
\draw (47,161.4) node [anchor=north west][inner sep=0.75pt]  [font=\footnotesize]  {$2$};
% Text Node
\draw (113,192.4) node [anchor=north west][inner sep=0.75pt]    {$G$};
% Text Node
\draw (268,192.4) node [anchor=north west][inner sep=0.75pt]    {$\mathcal{N}_{f}( G)$};
% Text Node
\draw (438,192.4) node [anchor=north west][inner sep=0.75pt]    {$C_{f}( G)$};
% Text Node
\draw (44,176.4) node [anchor=north west][inner sep=0.75pt]  [font=\footnotesize]  {$u$};
% Text Node
\draw (105,74.4) node [anchor=north west][inner sep=0.75pt]  [font=\footnotesize]  {$v$};
% Text Node
\draw (184,176.4) node [anchor=north west][inner sep=0.75pt]  [font=\footnotesize]  {$w$};

\end{tikzpicture}

    \caption{A maximal h-independent broadcast $f$ on a graph $G$ along with the corresponding neighbourhood graph and certification digraph. Observe that $v$ is certified by $u$ and $w$, and $u$ and $w$ certify each other.}
    \label{fig:certif_digraph}
\end{figure}
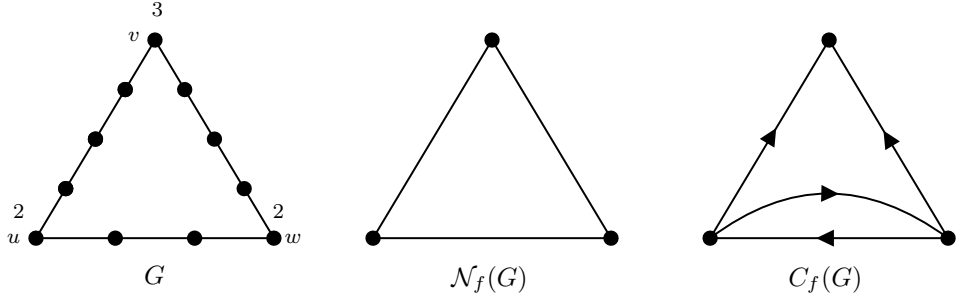

\begin{itemize}
\item The \emph{underlying graph} of $\mathcal{C}_f(G)$ (or of a
subgraph $\mathcal{H}$ of $\mathcal{C}_f(G)$) is the graph obtained by replacing arcs of  $\mathcal{C}_f(G)$ (or
$\mathcal{H}$) by edges and identifying double edges.
\end{itemize}

The underlying graph of $\mathcal{C}_f(G)$ in the example shown in Figure \ref{fig:certif_digraph} is a triangle. Note that the underlying graph of $\mathcal{C}_f(G)$ need not contain all edges of $\mathcal{N}_f(G)$.

\begin{proposition}
\label{prop:mar_prop6}
If $f$ is a maximal hearing independent broadcast on $G$ such
that $|V_{f}^{+}|\geq2$, then each vertex $u\in V_{f}^{+}$ is adjacent, in
$\mathcal{N}_{f}(G)$, to a vertex $v\in V_{f}^{+}$ such that $f(v)\leq
f(u)$.
\end{proposition}

\begin{proof}
Suppose there exists a vertex $u\in V_f^+$ such that 
$d_{G}(u,v)\geq
f(v)+1>f(u)+1$ for all $v\in V_{f}^{+}-\{u\}$. Then the broadcast obtained by
increasing the strength of $f$ at $u$ by $1$ is also hearing
independent, contradicting the maximality of $f$. 
\end{proof}

\begin{proposition}
\label{prop:mar_prop7}Let $f$ be a hearing independent broadcast on $G$ such that
$|V_{f}^{+}|\geq2$. Then $f$ is maximal hearing independent if and only if $f$ is dominating 
and each vertex of $\mathcal{C}_{f}(G)$ has positive in-degree.
\end{proposition}

\begin{proof}

If $v$ has in-degree 0 for some $v\in \mathcal{C}_f(G)$, by definition of $\mathcal{C}_f(G)$, no vertex on the $f$-boundary of $v$ is adjacent to a vertex in $V_f^+-\{v\}$. By part (ii) of Proposition \ref{prop:nec_suf_hearing}, $f$ is not maximal.

Conversely, if $f$ is dominating and every vertex of $\mathcal{C}_f(G)$ has positive in-degree, then $f$ is maximal by part (ii) of Proposition \ref{prop:nec_suf_hearing}.
\end{proof}

\begin{proposition}
\label{prop:mar_prop8}
Let $f$ be a maximal hearing independent broadcast on $G$ such
that $|V_{f}^{+}|\geq2$. Suppose $C$ is a cycle in the underlying graph of
$\mathcal{C}_{f}(G)$. Then the subgraph of $\mathcal{C}_f(G)$ with arcs corresponding to $E(C)$ contains a directed cycle of length at least 3 if and only if every edge of $C$ corresponds to a double arc.
\end{proposition}

\begin{proof}
Suppose the subgraph of $\mathcal{C}_f(G)$ corresponding to $C$ contains a directed cycle. Label the vertices of $C$ as $v_1, v_2, ..., v_k$ such that $v_i$ certifies $v_{i+1}$ for all $1\leq i\leq k-1$ and $v_k$ certifies $v_1$. Then $f(v_i)\leq f(v_{i+1})$ for all $1\leq i\leq k-1$ and $f(v_k)\leq f(v_1)$.

Without loss of generality, suppose for a contradiction that $v_1\rightarrow v_2$ but $v_2\not\rightarrow v_1$. Since $(v_1, v_2)$ is a single arc if and only if $f(v_1)<f(v_2)$, we have that
\begin{center}
    $f(v_2)>f(v_1)\geq f(v_k) \geq f(v_{k-1})\geq \cdots \geq f(v_2)$,
\end{center}
which is impossible. 

The converse is obvious. 
\end{proof}

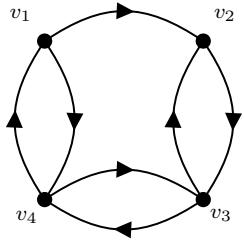
\begin{figure}[H]
    \centering

\tikzset{every picture/.style={line width=0.75pt}} %set default line width to 0.75pt        

\begin{tikzpicture}[x=0.75pt,y=0.75pt,yscale=-1,xscale=1]

%Straight Lines [id:da5380945354956626] 
\draw    (140,160) ;
\draw [shift={(140,160)}, rotate = 0] [color={rgb, 255:red, 0; green, 0; blue, 0 }  ][fill={rgb, 255:red, 0; green, 0; blue, 0 }  ][line width=0.75]      (0, 0) circle [x radius= 3.35, y radius= 3.35]   ;
%Straight Lines [id:da5478539078914899] 
\draw    (220,160) ;
\draw [shift={(220,160)}, rotate = 0] [color={rgb, 255:red, 0; green, 0; blue, 0 }  ][fill={rgb, 255:red, 0; green, 0; blue, 0 }  ][line width=0.75]      (0, 0) circle [x radius= 3.35, y radius= 3.35]   ;
%Curve Lines [id:da5140071181827996] 
\draw    (140,80) .. controls (170,59.8) and (190,59.8) .. (220,80) ;
\draw [shift={(220,80)}, rotate = 33.95] [color={rgb, 255:red, 0; green, 0; blue, 0 }  ][fill={rgb, 255:red, 0; green, 0; blue, 0 }  ][line width=0.75]      (0, 0) circle [x radius= 3.35, y radius= 3.35]   ;
\draw [shift={(185.26,65.15)}, rotate = 182.77] [fill={rgb, 255:red, 0; green, 0; blue, 0 }  ][line width=0.08]  [draw opacity=0] (8.93,-4.29) -- (0,0) -- (8.93,4.29) -- cycle    ;
\draw [shift={(140,80)}, rotate = 326.05] [color={rgb, 255:red, 0; green, 0; blue, 0 }  ][fill={rgb, 255:red, 0; green, 0; blue, 0 }  ][line width=0.75]      (0, 0) circle [x radius= 3.35, y radius= 3.35]   ;
%Shape: Boxed Bezier Curve [id:dp4947379688088125] 
\draw    (220,80) .. controls (240.2,110) and (240.2,130) .. (220,160) ;
\draw [shift={(234.85,125.26)}, rotate = 272.77] [fill={rgb, 255:red, 0; green, 0; blue, 0 }  ][line width=0.08]  [draw opacity=0] (8.93,-4.29) -- (0,0) -- (8.93,4.29) -- cycle    ;
%Shape: Boxed Bezier Curve [id:dp6550943745050994] 
\draw    (220,160) .. controls (190,180.2) and (170,180.2) .. (140,160) ;
\draw [shift={(174.74,174.85)}, rotate = 2.77] [fill={rgb, 255:red, 0; green, 0; blue, 0 }  ][line width=0.08]  [draw opacity=0] (8.93,-4.29) -- (0,0) -- (8.93,4.29) -- cycle    ;
%Shape: Boxed Bezier Curve [id:dp9305512592718834] 
\draw    (140,160) .. controls (119.8,130) and (119.8,110) .. (140,80) ;
\draw [shift={(125.15,114.74)}, rotate = 92.77] [fill={rgb, 255:red, 0; green, 0; blue, 0 }  ][line width=0.08]  [draw opacity=0] (8.93,-4.29) -- (0,0) -- (8.93,4.29) -- cycle    ;
%Shape: Boxed Bezier Curve [id:dp37956472763289195] 
\draw    (220,160) .. controls (199.8,130) and (199.8,110) .. (220,80) ;
\draw [shift={(205.15,114.74)}, rotate = 92.77] [fill={rgb, 255:red, 0; green, 0; blue, 0 }  ][line width=0.08]  [draw opacity=0] (8.93,-4.29) -- (0,0) -- (8.93,4.29) -- cycle    ;
%Shape: Boxed Bezier Curve [id:dp5834969065766857] 
\draw    (140,80) .. controls (160.2,110) and (160.2,130) .. (140,160) ;
\draw [shift={(154.85,125.26)}, rotate = 272.77] [fill={rgb, 255:red, 0; green, 0; blue, 0 }  ][line width=0.08]  [draw opacity=0] (8.93,-4.29) -- (0,0) -- (8.93,4.29) -- cycle    ;
%Curve Lines [id:da7306035008590157] 
\draw    (140,160) .. controls (170,139.8) and (190,139.8) .. (220,160) ;
\draw [shift={(185.26,145.15)}, rotate = 182.77] [fill={rgb, 255:red, 0; green, 0; blue, 0 }  ][line width=0.08]  [draw opacity=0] (8.93,-4.29) -- (0,0) -- (8.93,4.29) -- cycle    ;

% Text Node
\draw (121,62.4) node [anchor=north west][inner sep=0.75pt]  [font=\footnotesize]  {$v_{1}$};
% Text Node
\draw (222,163.4) node [anchor=north west][inner sep=0.75pt]  [font=\footnotesize]  {$v_{3}$};
% Text Node
\draw (224,62.4) node [anchor=north west][inner sep=0.75pt]  [font=\footnotesize]  {$v_{2}$};
% Text Node
\draw (124,164.4) node [anchor=north west][inner sep=0.75pt]  [font=\footnotesize]  {$v_{4}$};

\end{tikzpicture}

    \caption{An example of the situation considered in Proposition \ref{prop:mar_prop8}. Since 
    $(v_1, v_2)$ is a single arc, we have that $f(v_1)<f(v_2)$, contradicting $f(v_2)=f(v_3)=f(v_4)=f(v_1)$ along the double arcs.}
    \label{fig:dir_cycle}
\end{figure}

\begin{corollary}
\label{cor:opposite_arcs}
Let $f$ be a maximal hearing independent broadcast on $G$ such
that $|V_{f}^{+}|\geq2$, and suppose $C$ is a cycle in the underlying graph of
$\mathcal{C}_{f}(G)$. If $C$ contains an edge $e$ that does not corresponds to a double arc in $\mathcal{C}_f(G)$, then $C$ contains an edge $e'$ corresponding to a single arc such that $e$ and $e'$ are oriented in opposite directions in the subgraph of $\mathcal{C}_f(G)$ with arcs corresponding to the edges of $C$.

\end{corollary}

\begin{proof}
Suppose not. Then the vertices of $C$ may be labelled $v_1, v_2, ..., v_k$ such that $e$ correspond to the single arc $(v_k, v_1)$, and for all $1\leq i\leq k-1$, $v_i$ certifies $v_{i+1}$. But then the subgraph of $\mathcal{C}_f(G)$ contains a directed cycle $v_1 \rightarrow v_2 \rightarrow \cdots \rightarrow v_k \rightarrow v_1$, contradicting Proposition \ref{prop:mar_prop8}.
 \end{proof}

\noindent We may now prove our main result, restated here for convenience. 
\vspace{8pt}

\noindent \textbf{Theorem 4.1.}  
    For any graph $G$, $i_{bn}(G)\leq i_h(G)$.

\begin{proof}

Let $f$ be an $i_h$-broadcast on a connected graph $G$. 
If $\sigma(f)=\text{rad}(G)$, then the claim follows from the fact that $i_{bn}(G)\leq \text{rad}(G)$. Suppose $\sigma(f)<\text{rad}(G)$ (in which case $|V_f^+|\geq 2$), and let $G_0=G-U_f^E$. By the maximality of $f$, each component of $G_0$ contains at least two broadcasting vertices. 

Suppose $i_h(G)<i_{bn}(G)$. 
By Proposition \ref{prop:nec_suf_hearing} (ii), for each $v\in V_f^+$ there exists $u\in V_f^+$ such that $u$ is adjacent to a vertex in $B_f(v)$. 
Let $P_{u\rightarrow v}$ denote an arbitrarily chosen $u-v$ geodesic in $G_0$, which must exist as every edge on the shortest path between $u$ and $v$ is covered by $f$.

We aim to find a tree $T$ containing at least one geodesic $P_{u\rightarrow v}$ for every $v\in V_f^+$, such that $T$ spans $G$ and the restriction of $f$ to $T$, denoted $f_T$, is a dominating broadcast on $T$.
Then, since $\sigma(f_T)<\text{rad}(G)\leq \text{rad}(T)$, Proposition \ref{prop:nec_suf_hearing} (ii) will imply that $f_T$ is a maximal h-independent broadcast on $T$ such that $\sigma(f)=\sigma(f_T)$. We may then follow the proof of Theorem \ref{theorem:ibnihtrees} and apply Theorem \ref{theorem:spanning_trees_ibn} to obtain a contradiction.

Consider the certification digraph of $G_0$, denoted $\mathcal{C}_f(G_0)$, and let $H_0$ be its underlying graph. Suppose $C_0$ is a cycle in $H_0$.

If $C_0$ corresponds to a directed cycle in $\mathcal{C}_f(G_0)$, then each of its arcs corresponds to a double arc by Proposition \ref{prop:mar_prop8}. Let $xy$ be an arbitrary edge of $C_0$ and let $\mathcal{H}_1$ be the subgraph of $\mathcal{C}_f(G_0)$ obtained by deleting the arcs $(x, y)$ and $(y, x)$ from $\mathcal{C}_f(G_0)$. 

If $C_0$ does not correspond to a directed cycle, then by Corollary \ref{cor:opposite_arcs}, there exist two single arcs $(x, y)$ and $(x', y')$ oriented in opposite directions along the subgraph of $\mathcal{C}_f(G_0)$ corresponding to $C_0$. We may select $(x, y)$ and $(x', y')$ such that either $y=y'$ or all edges between $y$ and $y'$ on $C_0$ correspond to double arcs, so that $y$ and $y'$ each have in-degree at least $2$ in $\mathcal{C}_f(G_0)$. 
Let $\mathcal{H}_1$ be the subgraph obtained by deleting $(x, y)$.  By Proposition \ref{prop:mar_prop7}, every vertex of $\mathcal{C}_f(G_0)$ must have positive in-degree, hence $x$ is certified by $w$ for some $w\in V_f^+ -\{x, y\}$ (which may or may not lie on~$C_0$.)

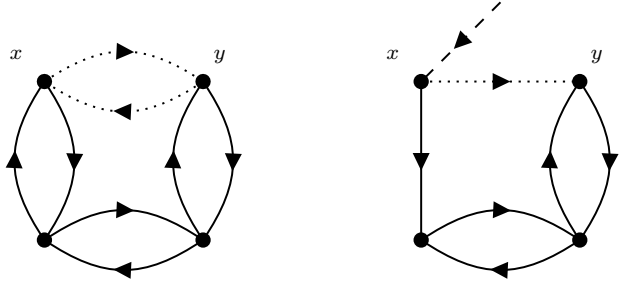
\begin{figure}[H]
    \centering

\tikzset{every picture/.style={line width=0.75pt}} %set default line width to 0.75pt        

\begin{tikzpicture}[x=0.75pt,y=0.75pt,yscale=-1,xscale=1]

%Straight Lines [id:da5380945354956626] 
\draw    (140,160) ;
\draw [shift={(140,160)}, rotate = 0] [color={rgb, 255:red, 0; green, 0; blue, 0 }  ][fill={rgb, 255:red, 0; green, 0; blue, 0 }  ][line width=0.75]      (0, 0) circle [x radius= 3.35, y radius= 3.35]   ;
%Straight Lines [id:da5478539078914899] 
\draw    (220,160) ;
\draw [shift={(220,160)}, rotate = 0] [color={rgb, 255:red, 0; green, 0; blue, 0 }  ][fill={rgb, 255:red, 0; green, 0; blue, 0 }  ][line width=0.75]      (0, 0) circle [x radius= 3.35, y radius= 3.35]   ;
%Curve Lines [id:da5140071181827996] 
\draw [line width=0.75]  [dash pattern={on 0.84pt off 2.51pt}]  (140,80) .. controls (170,59.8) and (190,59.8) .. (220,80) ;
\draw [shift={(220,80)}, rotate = 33.95] [color={rgb, 255:red, 0; green, 0; blue, 0 }  ][fill={rgb, 255:red, 0; green, 0; blue, 0 }  ][line width=0.75]      (0, 0) circle [x radius= 3.35, y radius= 3.35]   ;
\draw [shift={(185.26,65.15)}, rotate = 182.77] [fill={rgb, 255:red, 0; green, 0; blue, 0 }  ][line width=0.08]  [draw opacity=0] (8.93,-4.29) -- (0,0) -- (8.93,4.29) -- cycle    ;
\draw [shift={(140,80)}, rotate = 326.05] [color={rgb, 255:red, 0; green, 0; blue, 0 }  ][fill={rgb, 255:red, 0; green, 0; blue, 0 }  ][line width=0.75]      (0, 0) circle [x radius= 3.35, y radius= 3.35]   ;
%Shape: Boxed Bezier Curve [id:dp4947379688088125] 
\draw    (220,80) .. controls (240.2,110) and (240.2,130) .. (220,160) ;
\draw [shift={(234.85,125.26)}, rotate = 272.77] [fill={rgb, 255:red, 0; green, 0; blue, 0 }  ][line width=0.08]  [draw opacity=0] (8.93,-4.29) -- (0,0) -- (8.93,4.29) -- cycle    ;
%Shape: Boxed Bezier Curve [id:dp6550943745050994] 
\draw    (220,160) .. controls (190,180.2) and (170,180.2) .. (140,160) ;
\draw [shift={(174.74,174.85)}, rotate = 2.77] [fill={rgb, 255:red, 0; green, 0; blue, 0 }  ][line width=0.08]  [draw opacity=0] (8.93,-4.29) -- (0,0) -- (8.93,4.29) -- cycle    ;
%Shape: Boxed Bezier Curve [id:dp9305512592718834] 
\draw    (140,160) .. controls (119.8,130) and (119.8,110) .. (140,80) ;
\draw [shift={(125.15,114.74)}, rotate = 92.77] [fill={rgb, 255:red, 0; green, 0; blue, 0 }  ][line width=0.08]  [draw opacity=0] (8.93,-4.29) -- (0,0) -- (8.93,4.29) -- cycle    ;
%Shape: Boxed Bezier Curve [id:dp37956472763289195] 
\draw    (220,160) .. controls (199.8,130) and (199.8,110) .. (220,80) ;
\draw [shift={(205.15,114.74)}, rotate = 92.77] [fill={rgb, 255:red, 0; green, 0; blue, 0 }  ][line width=0.08]  [draw opacity=0] (8.93,-4.29) -- (0,0) -- (8.93,4.29) -- cycle    ;
%Shape: Boxed Bezier Curve [id:dp5834969065766857] 
\draw    (140,80) .. controls (160.2,110) and (160.2,130) .. (140,160) ;
\draw [shift={(154.85,125.26)}, rotate = 272.77] [fill={rgb, 255:red, 0; green, 0; blue, 0 }  ][line width=0.08]  [draw opacity=0] (8.93,-4.29) -- (0,0) -- (8.93,4.29) -- cycle    ;
%Curve Lines [id:da7306035008590157] 
\draw    (140,160) .. controls (170,139.8) and (190,139.8) .. (220,160) ;
\draw [shift={(185.26,145.15)}, rotate = 182.77] [fill={rgb, 255:red, 0; green, 0; blue, 0 }  ][line width=0.08]  [draw opacity=0] (8.93,-4.29) -- (0,0) -- (8.93,4.29) -- cycle    ;
%Curve Lines [id:da8880675215442113] 
\draw [line width=0.75]  [dash pattern={on 0.84pt off 2.51pt}]  (220,80) .. controls (190,100.2) and (170,100.2) .. (140,80) ;
\draw [shift={(174.74,94.85)}, rotate = 2.77] [fill={rgb, 255:red, 0; green, 0; blue, 0 }  ][line width=0.08]  [draw opacity=0] (8.93,-4.29) -- (0,0) -- (8.93,4.29) -- cycle    ;
%Straight Lines [id:da20125141922154177] 
\draw    (330,160) ;
\draw [shift={(330,160)}, rotate = 0] [color={rgb, 255:red, 0; green, 0; blue, 0 }  ][fill={rgb, 255:red, 0; green, 0; blue, 0 }  ][line width=0.75]      (0, 0) circle [x radius= 3.35, y radius= 3.35]   ;
%Straight Lines [id:da8659122349255626] 
\draw    (410,160) ;
\draw [shift={(410,160)}, rotate = 0] [color={rgb, 255:red, 0; green, 0; blue, 0 }  ][fill={rgb, 255:red, 0; green, 0; blue, 0 }  ][line width=0.75]      (0, 0) circle [x radius= 3.35, y radius= 3.35]   ;
%Shape: Boxed Bezier Curve [id:dp43255561599942927] 
\draw    (410,80) .. controls (430.2,110) and (430.2,130) .. (410,160) ;
\draw [shift={(424.85,125.26)}, rotate = 272.77] [fill={rgb, 255:red, 0; green, 0; blue, 0 }  ][line width=0.08]  [draw opacity=0] (8.93,-4.29) -- (0,0) -- (8.93,4.29) -- cycle    ;
%Shape: Boxed Bezier Curve [id:dp298984893168184] 
\draw    (410,160) .. controls (380,180.2) and (360,180.2) .. (330,160) ;
\draw [shift={(364.74,174.85)}, rotate = 2.77] [fill={rgb, 255:red, 0; green, 0; blue, 0 }  ][line width=0.08]  [draw opacity=0] (8.93,-4.29) -- (0,0) -- (8.93,4.29) -- cycle    ;
%Shape: Boxed Bezier Curve [id:dp9024343913265078] 
\draw    (410,160) .. controls (389.8,130) and (389.8,110) .. (410,80) ;
\draw [shift={(395.15,114.74)}, rotate = 92.77] [fill={rgb, 255:red, 0; green, 0; blue, 0 }  ][line width=0.08]  [draw opacity=0] (8.93,-4.29) -- (0,0) -- (8.93,4.29) -- cycle    ;
%Curve Lines [id:da7111800678657569] 
\draw    (330,160) .. controls (360,139.8) and (380,139.8) .. (410,160) ;
\draw [shift={(375.26,145.15)}, rotate = 182.77] [fill={rgb, 255:red, 0; green, 0; blue, 0 }  ][line width=0.08]  [draw opacity=0] (8.93,-4.29) -- (0,0) -- (8.93,4.29) -- cycle    ;
%Straight Lines [id:da39164699301673944] 
\draw  [dash pattern={on 0.84pt off 2.51pt}]  (330,80) -- (410,80) ;
\draw [shift={(410,80)}, rotate = 0] [color={rgb, 255:red, 0; green, 0; blue, 0 }  ][fill={rgb, 255:red, 0; green, 0; blue, 0 }  ][line width=0.75]      (0, 0) circle [x radius= 3.35, y radius= 3.35]   ;
\draw [shift={(375,80)}, rotate = 180] [fill={rgb, 255:red, 0; green, 0; blue, 0 }  ][line width=0.08]  [draw opacity=0] (8.93,-4.29) -- (0,0) -- (8.93,4.29) -- cycle    ;
\draw [shift={(330,80)}, rotate = 0] [color={rgb, 255:red, 0; green, 0; blue, 0 }  ][fill={rgb, 255:red, 0; green, 0; blue, 0 }  ][line width=0.75]      (0, 0) circle [x radius= 3.35, y radius= 3.35]   ;
%Straight Lines [id:da5375692048486223] 
\draw    (330,80) -- (330,160) ;
\draw [shift={(330,125)}, rotate = 270] [fill={rgb, 255:red, 0; green, 0; blue, 0 }  ][line width=0.08]  [draw opacity=0] (8.93,-4.29) -- (0,0) -- (8.93,4.29) -- cycle    ;
%Straight Lines [id:da3275372405198884] 
\draw  [dash pattern={on 4.5pt off 4.5pt}]  (370,40) -- (330,80) ;
\draw [shift={(346.46,63.54)}, rotate = 315] [fill={rgb, 255:red, 0; green, 0; blue, 0 }  ][line width=0.08]  [draw opacity=0] (8.93,-4.29) -- (0,0) -- (8.93,4.29) -- cycle    ;

% Text Node
\draw (121,62.4) node [anchor=north west][inner sep=0.75pt]  [font=\footnotesize]  {$x$};
% Text Node
\draw (224,62.4) node [anchor=north west][inner sep=0.75pt]  [font=\footnotesize]  {$y$};
% Text Node
\draw (311,62.4) node [anchor=north west][inner sep=0.75pt]  [font=\footnotesize]  {$x$};
% Text Node
\draw (414,62.4) node [anchor=north west][inner sep=0.75pt]  [font=\footnotesize]  {$y$};

\end{tikzpicture}

    \caption{ The two cases considered. 
In either case, both $x$ and $y$ have positive in-degree in $\mathcal{H}_1$. }
    \label{fig:cycles_in_h}
\end{figure}

Repeat the process: at each step $i$, 
if the underlying graph $H_i$ of $\mathcal{H}_i$ contains a cycle $C_i$, delete corresponding arcs as described for $C_0$.
Eventually, we obtain a spanning subgraph
$\mathcal{H}_{k}$ of $\mathcal{C}_{f}({G}_{0})$ such that its underlying graph $H_k$ is a tree or forest. 

Construct $T$ as follows. For each arc $(u,v)$
of $\mathcal{H}_{k}$, let $P_{u\rightarrow v}$ be a $u-v$ geodesic in $G$,
where only one such path is chosen if $(u,v)$ is a double arc. 
Since $H_k$ is acyclic, so is the spanning subgraph $T_0$ obtained by removing all edges of $G$ not lying on one of the chosen paths. 

If $H_k$ is a forest, then $T_0$ contains at least 2 components dominated by the restriction of $f$ to $T_0$, denoted $f_{T_0}$. If there exists an edge $a_0b_0\in E(G)$ joining these two components, let $T_1=T_0\cup \{a_0b_0\}$ and let $f_{T_1}$ denote the restriction of $f$ to $T_1$. Repeat the process. At each step $i$, we construct a spanning forest $T_i$ of $G$ such that $T_{i}=T_{i-1}\cup \{a_{i-1}b_{i-1}\}$ for some edge $a_{i-1} b_{i-1}$ joining two components dominated by $f_{i-1}$. 

Suppose $T_k$ is the first spanning subgraph in the sequence consisting of a tree dominated by $f_k$ and a set of isolated vertices $S$. If $S=\emptyset$, let $T_k=T$. Otherwise, suppose $|S|=m$ and let  $u_1, u_2, ..., u_m$ be an ordering of the vertices in $S$ such that if $i\leq j$, then $\text{max}\{f(v)-d_G(v, u_i)\,|\,v\in V_f^+\}\leq  \text{max}\{f(v)-d_G(v, u_j)\,|\,v\in V_f^+\} $. Successively join each vertex $u_i\in S$ to a vertex $w$ such that $d_G(w, v)<f(v)$ for some broadcasting vertex $v$ in $G$, which must exist as $f$ dominates $u_i$ in $G$. Continue joining edges along a $w-v$ geodesic until $u_i$ hears $f_k$ from a broadcasting vertex in $T_k$.  
Repeat the process for all remaining vertices of $ S$, and let $T$ be the resulting spanning tree of~$G$. 

\begin{figure}[H]
    \centering

\tikzset{every picture/.style={line width=0.75pt}} %set default line width to 0.75pt        

\begin{tikzpicture}[x=0.75pt,y=0.75pt,yscale=-1,xscale=1]

%Straight Lines [id:da47233760100819167] 
\draw    (160,60) -- (160,100) ;
\draw [shift={(160,60)}, rotate = 90] [color={rgb, 255:red, 0; green, 0; blue, 0 }  ][fill={rgb, 255:red, 0; green, 0; blue, 0 }  ][line width=0.75]      (0, 0) circle [x radius= 3.35, y radius= 3.35]   ;
%Straight Lines [id:da2058605454827036] 
\draw    (160,180) -- (160,220) ;
\draw [shift={(160,220)}, rotate = 90] [color={rgb, 255:red, 0; green, 0; blue, 0 }  ][fill={rgb, 255:red, 0; green, 0; blue, 0 }  ][line width=0.75]      (0, 0) circle [x radius= 3.35, y radius= 3.35]   ;
\draw [shift={(160,180)}, rotate = 90] [color={rgb, 255:red, 0; green, 0; blue, 0 }  ][fill={rgb, 255:red, 0; green, 0; blue, 0 }  ][line width=0.75]      (0, 0) circle [x radius= 3.35, y radius= 3.35]   ;
%Straight Lines [id:da24930205347625178] 
\draw    (160,140) -- (160,180) ;
\draw [shift={(160,140)}, rotate = 90] [color={rgb, 255:red, 0; green, 0; blue, 0 }  ][fill={rgb, 255:red, 0; green, 0; blue, 0 }  ][line width=0.75]      (0, 0) circle [x radius= 3.35, y radius= 3.35]   ;
%Straight Lines [id:da5179459855325821] 
\draw    (160,100) -- (160,140) ;
\draw [shift={(160,100)}, rotate = 90] [color={rgb, 255:red, 0; green, 0; blue, 0 }  ][fill={rgb, 255:red, 0; green, 0; blue, 0 }  ][line width=0.75]      (0, 0) circle [x radius= 3.35, y radius= 3.35]   ;
%Straight Lines [id:da15761533767818237] 
\draw  [dash pattern={on 0.84pt off 2.51pt}]  (200,140) -- (200,180) ;
\draw [shift={(200,180)}, rotate = 90] [color={rgb, 255:red, 0; green, 0; blue, 0 }  ][fill={rgb, 255:red, 0; green, 0; blue, 0 }  ][line width=0.75]      (0, 0) circle [x radius= 3.35, y radius= 3.35]   ;
\draw [shift={(200,140)}, rotate = 90] [color={rgb, 255:red, 0; green, 0; blue, 0 }  ][fill={rgb, 255:red, 0; green, 0; blue, 0 }  ][line width=0.75]      (0, 0) circle [x radius= 3.35, y radius= 3.35]   ;
%Straight Lines [id:da8349799871485879] 
\draw  [dash pattern={on 0.84pt off 2.51pt}]  (200,100) -- (200,140) ;
\draw [shift={(200,100)}, rotate = 90] [color={rgb, 255:red, 0; green, 0; blue, 0 }  ][fill={rgb, 255:red, 0; green, 0; blue, 0 }  ][line width=0.75]      (0, 0) circle [x radius= 3.35, y radius= 3.35]   ;
%Straight Lines [id:da554944688628471] 
\draw  [dash pattern={on 0.84pt off 2.51pt}]  (160,60) -- (200,100) ;
%Straight Lines [id:da6394461615850073] 
\draw  [dash pattern={on 0.84pt off 2.51pt}]  (160,220) -- (200,180) ;
%Straight Lines [id:da4693246459378686] 
\draw  [dash pattern={on 0.84pt off 2.51pt}]  (200,140) -- (240,140) ;
\draw [shift={(240,140)}, rotate = 0] [color={rgb, 255:red, 0; green, 0; blue, 0 }  ][fill={rgb, 255:red, 0; green, 0; blue, 0 }  ][line width=0.75]      (0, 0) circle [x radius= 3.35, y radius= 3.35]   ;

% Text Node
\draw (141,44.4) node [anchor=north west][inner sep=0.75pt]  [font=\small]  {$3$};
% Text Node
\draw (141,204.4) node [anchor=north west][inner sep=0.75pt]  [font=\small]  {$2$};
% Text Node
\draw (157,42.4) node [anchor=north west][inner sep=0.75pt]  [font=\small]  {$v$};
% Text Node
\draw (202,182.4) node [anchor=north west][inner sep=0.75pt]  [font=\small]  {$u_{3}$};
% Text Node
\draw (202,103.4) node [anchor=north west][inner sep=0.75pt]  [font=\small]  {$u_{4}$};
% Text Node
\draw (242,143.4) node [anchor=north west][inner sep=0.75pt]  [font=\small]  {$u_{1}$};
% Text Node
\draw (202,143.4) node [anchor=north west][inner sep=0.75pt]  [font=\small]  {$u_{2}$};
% Text Node
\draw (162,223.4) node [anchor=north west][inner sep=0.75pt]  [font=\small]  {$v'$};

\end{tikzpicture}

    \caption{Isolated vertices are joined by edges in increasing order of how much they are overdominated in $G$. Observe that if the edges $u_3 v'$ and $u_2 u_3 $ were added to $T_k$ before $u_1 u_2$, then $f_k$ would not dominate $u_1$ unless both edges $u_2 u_4$ and $u_4$ were added, creating a cycle.}
    \label{fig:my_label2}
\end{figure}
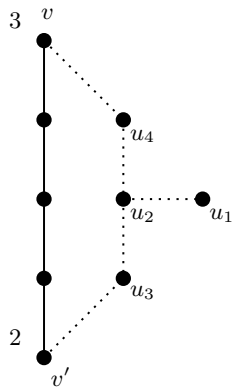

Let $f_{T}$ be the restriction of $f$ to $T$. Note that $|V_{f_{T}}%
^{+}|=|V_{f}^{+}|$ and $f_{T}(v)=f(v)$ for each $v\in V(G)$. Since $f$ is
hearing independent and dominating, so is $f_{T}$. Thus, since each vertex of
$\mathcal{H}_{k}$ has positive in-degree, Proposition \ref{prop:mar_prop7} implies that
$f_{T}$ is maximal hearing independent. Therefore $i_{h}(T)\leq
\sigma(f_{T})=\sigma(f)<\text{rad}(T)$.

Consider the restriction of $f_T$ to each component of $T-U_{f_T}^E$. By the maximality of $f_T$, each component contains at least two broadcasting vertices. If $f_T$ is maximal boundary independent, then $i_{bn}(T)\leq \sigma(f_T)=i_h(G)<i_{bn}(G)$, contradicting Theorem \ref{theorem:spanning_trees_ibn}. Therefore an edge in some component of $T-U_{f_T}^E$ hears $f_T$ from more than one broadcasting vertex. Following the proof of Theorem \ref{theorem:ibnihtrees}, we may obtain a dominating broadcast $g_l$ on $G$ such that for some central vertex $y$ of $G$, $\sigma(g_l)=g_l(y)=\sigma(f_T)\geq \text{rad}(T)$, a contradiction. 
\end{proof}

Applying the same approach as in the first half of the proof of Theorem \ref{theorem:spanning_trees_ibn}, we show that $i_h(T)\geq i_h(G)$ for any spanning tree of a connected graph $G$. We may then prove the analogous result to Theorem \ref{theorem:spanning_trees_ibn} for hearing independence as a corollary of Theorem~\ref{theorem:ibnih_gen2}.

\begin{theorem}
For any connected graph $G$, 
\begin{center}
    $i_h(G)=\textup{min}\{i_h(T)\,:\,T \text{ is a spanning tree of } G\}$.
\end{center}
\end{theorem}

\begin{proof}

Suppose there exists a tree $T$ spanning $G$ such that ${i_{\mathrm{h}}(T)<i_{\mathrm{h}}(G)}$, and let $f$ be an $i_{\mathrm{h}}%
$-broadcast on $T$. If $f$ is maximal boundary independent on $T$, then $\sigma(f)\leq i_{h}(T)\leq i_{h}(G)$ by Theorem \ref{theorem:spanning_trees_ibn}, so suppose at least one edge of $T$ hears $f$ from more than one broadcasting vertex. 

It follows that $|V_f^+|\geq 2$. Since $T$ is a tree and $i_h(T)<i_h(G)\leq \text{rad}(G)$, the set of uncovered edges $U_f^E$ is nonempty. Then $T-U_f^E$ is disconnected. By the maximality of $f$, every component of $T-U_f^E$ contains at least two broadcasting vertices.

Since $f$ is dominating but not h-independent on $G$, there exist vertices $u, v\in V_f^+$ such that $d_G(u, v)\leq f(v)$.
Arbitrarily add edges of $G-T$ to $T$ until the addition of some edge, say $e_1$, results in a spanning subgraph $G_1$ of $G$ such that $f$ is not h-independent on $G_1$. 
Proceed as in the proof of Theorem \ref{theorem:spanning_trees_ibn} to obtain a dominating broadcast $f'$ on $G$ such that $\sigma(f')=\text{rad}(G)$ and $|V_{f'}^+|=1$. But then $f'$ is maximal h-independent on $G$, a contradiction. 

\bigskip 
To show there exists a tree $T$ spanning $G$ such that $i_h(T)=i_h(G)$, let $f$ be an $i_h$-broadcast on $G$ and construct $T$ as in the proof of Theorem \ref{theorem:ibnih_gen2}, such that the restriction $f_T$ is maximal hearing independent on $T$. Then $i_h(T)\leq \sigma(f_T)=\sigma(f)=i_h(G)$. But we have already shown that $i_{h}(G)\leq i_h(T)$, hence $i_h(T)=i_h(G)$.
\end{proof}

In \cite{mynhardtmarchessault2021lower}, Marchessault and Mynhardt asked whether the difference $i_h(G)-i_{bn}(G)$ may be arbitrarily large. A construction of such an infinite family of graphs is presented in \cite{thesis}. 
Marchessault and Mynhardt also posed the problem of bounding the ratio $i_h(G)/i_{bn}(G)$ for general graphs, which we consider in the following section. 

\section{The Ratio $i_h(G)/i_{bn}(G)$}

Recall that $i\,\diamond \,i_{bn}$ and $i \,\diamond \, i_h$.
The ratios $\frac{i(G)}{i_{bn}(G)}$ and  $\frac{i(G)}{i_{h}(G)}$, in general, may be arbitrarily large: for example, $i(K_{n, n})=n$ whereas $i_{bn}(K_{n,n})=i_h(K_{n,n})=2$ for all $n\geq 2$.
In \cite{mynhardtmarchessault2021lower}, Marchessault and Mynhardt found that $i_{bn}(G)\leq \lceil \frac{4i(G)}{3}\rceil$ and asked whether the ratio $\frac{i_h(G)}{i_{bn}(G)}$ may be similarly bounded.

In the previous section, we found that $i_h$ and $i_{bn}$ are comparable. In particular, since $i_{bn}(G)\leq i_{h}(G)$ for all $G$, $\frac{i_{bn}(G)}{i_h(G)}\leq 1$.
We now prove that $\frac{i_{h}(G)}{i_{bn}(G)}\leq \frac{5}{4}$ for all graphs $G.$

\begin{proposition}\textup{\cite{mynhardtmarchessault2021lower}}
\label{prop:subtrees}
If $T'$ is a subtree of a tree $T$, then $i_{bn}(T')\leq i_{bn}(T)$. 
\end{proposition}

Recall that $P_n$ denotes the path on $n$ vertices. Since a tree $T$ with diameter $d$ contains the path $P_{d+1}$ as a subtree, we may bound $i_{bn}(T)$ below by the value $i_{bn}(P_{d+1})$, which was determined exactly by Neilson in \cite{neilsonphd}.

\begin{proposition}
\textup{\cite{neilsonphd}}
\label{prop:ibn_paths}
For any $n\neq 3$, $i_{bn}(P_n)=\lceil \frac{2n}{5}\rceil$.
\end{proposition}

The exception is $P_3$, which admits a maximal bn-independent broadcast of cost 1.

\begin{theorem}
For any graph $G$, $1\leq i_{h}(G)/i_{bn}(G)\leq 5/4$.
\end{theorem}

\begin{proof}
Since $i_{bn}(G)$ is equal to the sum of the costs of $i_{bn}$-broadcasts on all components of $G$, it suffices to consider graph with one component, so assume $G$ is connected. By Theorem \ref{theorem:spanning_trees_ibn}, there exists a tree $T$ spanning $G$ such that $i_{bn}(T)=i_{bn}(G)$. 

Let $d=\text{diam}(T)$ and let $D \cong P_{d+1}$ be a diametrical path of $T$. It follows from Proposition \ref{prop:subtrees} and Proposition \ref{prop:ibn_paths} that $i_{bn}(T)\geq \lceil\frac{2(d+1)}{5}\rceil\geq \frac{2(d+1)}{5}$. 
Since $T$ spans $G$, $\text{rad}(G)\leq \text{rad}(T)$.
Finally, since $i_h(G)\leq \text{rad}(G)$ for any connected graph $G$, we have that

\begin{align*}
     \frac{i_h(G)}{i_{bn}(G)}=\frac{i_h(G)}{i_{bn}(T)}\leq \frac{5\cdot\text{rad}(G)}{2(d+1)} \leq \frac{5\cdot \text{rad}(G)}{4\cdot\text{rad}(T)}\leq \frac{5\cdot\text{rad}(T)}{4\cdot\text{rad}(T)}=\frac{5}{4}.
\end{align*}
         
The lower bound follows from Theorem \ref{theorem:ibnih_gen2}. \end{proof}

\section{Upper and Lower Bounds on $\alpha_{bn}(G)$}

We turn our attention to the upper parameter corresponding to boundary independent broadcasts. 
Recall that $\alpha_{bn}(G)$ denotes the maximum weight of a boundary independent broadcast on a graph $G$, called the \textit{upper boundary independence number} or simply the \textit{boundary independence number} of $G$. The corresponding parameter for the maximum weight of a hearing independent broadcast, denoted $\alpha_h(G)$, is called the \textit{hearing independence number} of $G$. 
Our focus in this section is to establish bounds on $\alpha_{bn}(G)$ for general graphs, and determine parameters comparable to $\alpha_{bn}$ to place the parameter within a chain of inequalities.

In \cite{bessy2}, Bessy and Rautenbach found that $\alpha_h(G)< 4\alpha(G)$, adapting a proof technique used by Neilson \cite{neilsonphd} to show that $\alpha_h(G)< 2\alpha_{bn}(G)$. 
 Mynhardt and Neilson further studied the ratio $\frac{\alpha_{bn}(G)}{\alpha(G)}$ in \cite{mynhardt2021lowerexact} and asked whether it can be shown that $\frac{\alpha_{bn}(G)}{\alpha(G)}<2$ for all graphs $G$.

\begin{theorem}
\label{theorem:2alpha}
For any graph $G$, $\alpha_{bn}(G)< 2\alpha(G)$.
\end{theorem}

\begin{proof} Let $f$ be an $\alpha_{bn}$-broadcast on $G$.
%such that $V_f^1$ is maximal. 
If $V_f^+ = V_f^1$, $V_f^+$ is an independent set, so assume there exists $v$ such that $f(v)\geq 2$. Let $u\in B_f(v)$ and consider a subgraph $T_u$ consisting of $u$ and unique geodesics from $u$ to each vertex in $S=\{x \in V_f^+ : u \in B_f(x)\}$, the set of all broadcasting vertices heard by $u$ (if $u\in PB_f(v)$, $S=\{v\}$). 
%Let $T_u$ be the subgraph consisting of $u$ and unique paths from $u$ to each vertex in $S$. 
Since no edge can be covered by two different broadcasts, $T_u$ is an induced path or spider, hence $|V(T_u)| = 1+ \sum\limits_{w\in S} f(w)$. 

Consider a proper two-colouring of $T_u$ and define a new boundary independent broadcast by deleting all broadcasts from the leaves of $T_u$ and adding strength-one broadcasts from each vertex in the color set of highest cardinality (if they have equal cardinality, select one arbitrarily). 

Since broadcasts overlap only on boundaries,
the broadcasting vertices of the resulting bn-independent broadcast form an independent set of cardinality at least $|V(T_u)|/2$. 
Repeating the process until no vertices broadcasting at strength greater than 1 remain yields an independent set on $G$. Since $|V(T_u)|<\sum\limits_{w\in S}f(w)$, it follows that ${\alpha_{bn}(G)/2 < \alpha(G)}$.
\end{proof}

As $\alpha_{bn}(G)\leq \alpha_h(G)$ for any graph $G$, Bessy and Rautenbach's bound now follows easily from Neilson's result and Theorem \ref{theorem:2alpha}.

\begin{corollary}
For any graph $G$, $\alpha_{h}(G)< 4\alpha(G)$.
\end{corollary}

Let $\delta(G)$ denote the minimum degree of $G$. It is easy to see that ${\alpha(G)\leq n-\delta(G)}$, as the inclusion of any vertex to an independent set excludes its neighbours. We determine the analogous result for maximum boundary independence, thereby solving an open problem posed in  \cite{mncomparingupper}.

\begin{theorem}
For any graph $G$ of order $n$, $\alpha_{bn}(G)\leq n-\delta(G)$.
\end{theorem} 

\begin{proof}

The result is clear if $|V(G)|=1$, so assume the theorem holds for all graphs $G$ such that $|V(G)|\leq n-1$ and consider a graph $G$ of order $n$. 
Let $f$ be an $\alpha_{bn}$-broadcast on $G$ with $|V_{f}^1|$ maximum. For any $v\in V_f^{++}$, $PB_f(v)=\emptyset$, otherwise a new boundary independent broadcast of equal weight could be constructed by reducing $f(v)$ by 1, and broadcasting at strength 1 from a vertex in the $f$-private boundary of $v$. 

For some $v\in V_f^{+}$, consider the graph $G'=G-PN_f(v)$ of order $n'$, and let $f_{G'}$ be $f$ restricted to $G'$. By induction, $\sigma(f_{G'})\leq \alpha_{bn}(G')\leq n'-\delta(G')$, %Since $f(v)\leq k$ where $k=|PN_f(v)|$,
hence $\alpha_{bn}(G)=\sigma(f_{G'})+f(v)\leq f(v)+n'-\delta(G')$.

Consider a vertex $u$ of minimum degree in $G'$. First, suppose $u\in B_f(v)$ and let $k=|PN_f(v)\cap N_G(u)|$. Let $P$ be a $u-v$ geodesic in $G$. Since $P$ has length $f(v)$ and $u$ is adjacent to $k-1$ vertices in $PN_f(v)-V(P$), we have that $k-1+f(v)\leq |PN_f(v)|$. In particular, if $k+f(v)= |PN_f(v)|+1$, then either $PN_f(v)$ consists of a path on $ f(v)$ vertices (and thus $d(v)=1$), or $f(v)=2$ and $d_G(v)=k$. In either case, since $u$ hears more than one broadcasting vertex under $f$,  $\delta(G)<d_G(u)$. Thus,

\begin{center}
    $\alpha_{bn}(G)
\leq f(v)+n'-\delta(G')
\leq f(v)+n'-d_G(u)+k\leq n+1-d_G(u)\leq n-\delta(G)$.
\end{center}

Otherwise, if $k+f(v)\leq |PN_f(v)|$, 

\[ 
\alpha_{bn}(G)\leq f(v)+n'-d_G(u)+k\leq  n-\delta(G).
\]

Finally, if $u\notin B_f(v)$, then $d_G(u)=d_{G'}(u)$, 
hence $\delta(G)\leq \delta(G')$ and $\alpha_{bn}(G)\leq n-\delta(G')\leq n-\delta(G)$ as desired. 
\end{proof}

\section{The Hardness of Determining $\alpha_{bn}(G)$}

Given a graph $G$ and a positive integer $k$, we may verify that a broadcast $f$ on $G$ of cost at least $k$ is boundary independent in polynomial time by checking that $f(u)+f(v)\leq d_G(u, v)$ for every pair of distinct broadcasting vertices $u, v\in V_f^+$. It follows that the maximum bn-independent broadcast problem is NP.
We proceed to show that the problem is NP-complete by a transformation from the independent set problem on $G$, which was shown to be NP-complete by Karp \cite{karp}, to the maximum bn-independent broadcast problem on a corresponding graph $C_G$.

The \textit{corona} $G \bigodot H$ of graphs $G$ and $H$ is constructed from $G$ and $n=|V(G)|$ copies of $H$ by joining the $i$th vertex in $G$ by edges to every vertex in the $i$th copy of $H$. Let $C_G=G\bigodot K_1$.

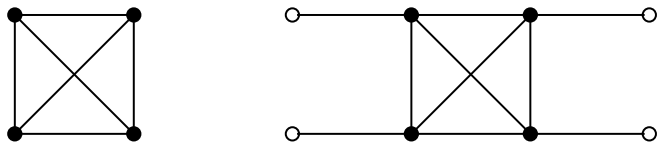
\begin{figure}[H]
    \centering

\tikzset{every picture/.style={line width=0.75pt}} %set default line width to 0.75pt        

\begin{tikzpicture}[x=0.75pt,y=0.75pt,yscale=-1,xscale=1]

%Straight Lines [id:da5174761141371971] 
\draw    (110,110) -- (170,110) ;
\draw [shift={(110,110)}, rotate = 0] [color={rgb, 255:red, 0; green, 0; blue, 0 }  ][fill={rgb, 255:red, 0; green, 0; blue, 0 }  ][line width=0.75]      (0, 0) circle [x radius= 3.35, y radius= 3.35]   ;
%Straight Lines [id:da5755190127717684] 
\draw    (170,170) -- (110,170) ;
\draw [shift={(170,170)}, rotate = 180] [color={rgb, 255:red, 0; green, 0; blue, 0 }  ][fill={rgb, 255:red, 0; green, 0; blue, 0 }  ][line width=0.75]      (0, 0) circle [x radius= 3.35, y radius= 3.35]   ;
%Straight Lines [id:da9315949837524446] 
\draw    (170,110) -- (170,170) ;
\draw [shift={(170,110)}, rotate = 90] [color={rgb, 255:red, 0; green, 0; blue, 0 }  ][fill={rgb, 255:red, 0; green, 0; blue, 0 }  ][line width=0.75]      (0, 0) circle [x radius= 3.35, y radius= 3.35]   ;
%Straight Lines [id:da5477480830539563] 
\draw    (110,170) -- (110,110) ;
\draw [shift={(110,170)}, rotate = 270] [color={rgb, 255:red, 0; green, 0; blue, 0 }  ][fill={rgb, 255:red, 0; green, 0; blue, 0 }  ][line width=0.75]      (0, 0) circle [x radius= 3.35, y radius= 3.35]   ;
%Straight Lines [id:da5597875005108821] 
\draw    (310,110) -- (370,110) ;
\draw [shift={(310,110)}, rotate = 0] [color={rgb, 255:red, 0; green, 0; blue, 0 }  ][fill={rgb, 255:red, 0; green, 0; blue, 0 }  ][line width=0.75]      (0, 0) circle [x radius= 3.35, y radius= 3.35]   ;
%Straight Lines [id:da08093595153123001] 
\draw    (370,170) -- (310,170) ;
\draw [shift={(370,170)}, rotate = 180] [color={rgb, 255:red, 0; green, 0; blue, 0 }  ][fill={rgb, 255:red, 0; green, 0; blue, 0 }  ][line width=0.75]      (0, 0) circle [x radius= 3.35, y radius= 3.35]   ;
%Straight Lines [id:da9931697126996755] 
\draw    (370,110) -- (370,170) ;
\draw [shift={(370,110)}, rotate = 90] [color={rgb, 255:red, 0; green, 0; blue, 0 }  ][fill={rgb, 255:red, 0; green, 0; blue, 0 }  ][line width=0.75]      (0, 0) circle [x radius= 3.35, y radius= 3.35]   ;
%Straight Lines [id:da3068360856065282] 
\draw    (310,170) -- (310,110) ;
\draw [shift={(310,170)}, rotate = 270] [color={rgb, 255:red, 0; green, 0; blue, 0 }  ][fill={rgb, 255:red, 0; green, 0; blue, 0 }  ][line width=0.75]      (0, 0) circle [x radius= 3.35, y radius= 3.35]   ;
%Straight Lines [id:da8300327836062393] 
\draw    (252.35,170) -- (310,170) ;
\draw [shift={(250,170)}, rotate = 0] [color={rgb, 255:red, 0; green, 0; blue, 0 }  ][line width=0.75]      (0, 0) circle [x radius= 3.35, y radius= 3.35]   ;
%Straight Lines [id:da4474182865699121] 
\draw    (252.35,110) -- (310,110) ;
\draw [shift={(250,110)}, rotate = 0] [color={rgb, 255:red, 0; green, 0; blue, 0 }  ][line width=0.75]      (0, 0) circle [x radius= 3.35, y radius= 3.35]   ;
%Straight Lines [id:da019497959888931105] 
\draw    (427.65,170) -- (370,170) ;
\draw [shift={(430,170)}, rotate = 180] [color={rgb, 255:red, 0; green, 0; blue, 0 }  ][line width=0.75]      (0, 0) circle [x radius= 3.35, y radius= 3.35]   ;
%Straight Lines [id:da10229367495256914] 
\draw    (427.65,110) -- (370,110) ;
\draw [shift={(430,110)}, rotate = 180] [color={rgb, 255:red, 0; green, 0; blue, 0 }  ][line width=0.75]      (0, 0) circle [x radius= 3.35, y radius= 3.35]   ;
%Straight Lines [id:da49679989857176454] 
\draw    (110,110) -- (170,170) ;
%Straight Lines [id:da4967163853093546] 
\draw    (310,110) -- (370,170) ;
%Straight Lines [id:da8564322057734488] 
\draw    (170,110) -- (110,170) ;
%Straight Lines [id:da5342742492330574] 
\draw    (370,110) -- (310,170) ;

\end{tikzpicture}

    \caption{The construction of the corona $K_4 \bigodot K_1$.}
    \label{fig:my_label}
\end{figure}

\begin{proposition}
Let $G$ be a connected graph on $n\geq 2$ vertices and let  $C_G=G\bigodot K_1$. Then $\alpha_{bn}(C_G)= n+\alpha(G)$.
\label{prop:coronas}
\end{proposition}

\begin{proof}
Let $l_1,...,l_n$ be a labelling of the leaves of $C_G$, and let $v_1, ..., v_n$ be a labelling of the remaining vertices such that $l_i$ is adjacent to $v_i$ for all $i$. 
Let $S\subseteq \{v_1, ..., v_n\}$ be a maximum independent set on the subgraph of $C_G$ corresponding to $G$. 

Define a broadcast $f$ on $C_G$ by
\begin{center}
    $f(u)=
 \begin{cases}
2 & \text{ if } u=l_i \text{ for some \textit{i} and } v_i\in S\\ 
1 & \text{ if } u=l_i \text{ for some \textit{i} and } v_i\notin S\\ 
0 & \text{ otherwise.} 
\end{cases}$
\end{center}
Then $\sum_{u}f(u) =n+\alpha(G)$.
Since only leaves broadcast, and since $d_{C_G}(l_i, l_j)\leq f(l_i)+f(l_j)$ for any pair of leaves $l_i\neq l_j$, $f$ is boundary independent. Therefore $\alpha_{bn}(C_G)\geq n+\alpha(G)$. 

To show that $\alpha_{bn}(C_G)\leq n+\alpha(G)$, let $f$ be a $\alpha_{bn}$-broadcast on $C_G$. If a vertex $v_i$ is broadcasting, the bn-independent broadcast $f'$ on $C_G$ defined by $f'(l_i)=f(v_i)+1$, $f'(v_i)=0$, and $f'(u)=f(u)$ otherwise has greater cost than $f$, a contradiction. It follows that $V_f^+\subseteq\{ l_1, ..., l_n\}$.  

Suppose some leaf $l_i$ broadcasts at strength $k\geq 3$ under $f$. Since $f(l_i)\leq e(l_i)$, there exist at least $k-2$ leaves which hear $l_i$ from a distance greater than $2$. 
Define a new maximum bn-independent broadcast
$f'$ by
\begin{center}
      $f'(l_i)=
 \begin{cases}
f(l_i) & \text{ if } f(l_i)\in \{1, 2\}
\\ 
2 & \text{ if } f(l_i)=k>2
\\ 
1 & \text{ if } f(l_i)=0 
\end{cases}$
\end{center}
and consider $R\subseteq\{v_1, ..., v_n\}$, where $v_i\in R$ if and only if $f'(l_i)=2$. Two vertices adjacent to leaves broadcasting at strength 2 cannot be adjacent, hence $R$ is an independent set on the subgraph of $C_G$ corresponding to $G$. Since every leaf broadcasts at strength 1 or 2 under $f'$, we have that $|R|=\alpha_{bn}(C_G)-n$. Therefore $\alpha_{bn}(C_G)\leq n+|R|\leq n+\alpha(G)$.
\end{proof}

\begin{corollary}
Let $G$ be a graph and $k$ a positive integer. The problem of deciding whether $\alpha_{bn}(G)\geq k$ is NP-complete.
\end{corollary}

\begin{proof}
We have already observed that the maximum bn-independent broadcast problem is NP. As the transformation described in Proposition \ref{prop:coronas} can be carried out in polynomial time, the problem is NP-complete. 
\end{proof}

It follows that for any graph class $\mathcal{C}$ for which $G\in \mathcal{C}$ implies $C_G\in \mathcal{C}$, if the independent set problem is known to be NP-complete for all graphs in $\mathcal{C}$, then so is the maximum bn-independent set problem.

Graph classes for which the independent set problem is known to be NP-complete are listed in \cite{graphclasses}. In particular, the maximum bn-independent set problem is NP-complete for planar and toroidal graphs, graphs with maximum degree $\Delta \in \{3, 4, 5, 6\}$, triangle-free and $C_n$-free graphs for $n\in \{4, 5, 6\}$, $K_n$-free graphs for $n\in\{4, 6, 7\}$, and house-free graphs. 

\section{Trees}

Let $T$ be a tree. Its \textit{branch-leaf representation} $\mathcal{BL}(T)$, also known as the \textit{homeomorphic reduction} of $T$, is obtained by successively removing a vertex of degree 2 and adding an edge between its two neighbours until no such vertices remain. Deleting all leaves yields the \textit{branch representation} of $T$, denoted $\mathcal{B}(T)$.

\begin{figure}[H]
    \centering

\tikzset{every picture/.style={line width=0.75pt}} %set default line width to 0.75pt        

\begin{tikzpicture}[x=0.75pt,y=0.75pt,yscale=-1,xscale=1]

%Straight Lines [id:da768116565337633] 
\draw    (120,50) -- (150,90) ;
\draw [shift={(150,90)}, rotate = 53.13] [color={rgb, 255:red, 0; green, 0; blue, 0 }  ][fill={rgb, 255:red, 0; green, 0; blue, 0 }  ][line width=0.75]      (0, 0) circle [x radius= 3.35, y radius= 3.35]   ;
\draw [shift={(120,50)}, rotate = 53.13] [color={rgb, 255:red, 0; green, 0; blue, 0 }  ][fill={rgb, 255:red, 0; green, 0; blue, 0 }  ][line width=0.75]      (0, 0) circle [x radius= 3.35, y radius= 3.35]   ;
%Straight Lines [id:da3616842515739038] 
\draw    (120,90) -- (120,50) ;
\draw [shift={(120,90)}, rotate = 270] [color={rgb, 255:red, 0; green, 0; blue, 0 }  ][fill={rgb, 255:red, 0; green, 0; blue, 0 }  ][line width=0.75]      (0, 0) circle [x radius= 3.35, y radius= 3.35]   ;
%Straight Lines [id:da384886045606043] 
\draw    (120,170) -- (120,130) ;
\draw [shift={(120,170)}, rotate = 270] [color={rgb, 255:red, 0; green, 0; blue, 0 }  ][fill={rgb, 255:red, 0; green, 0; blue, 0 }  ][line width=0.75]      (0, 0) circle [x radius= 3.35, y radius= 3.35]   ;
%Straight Lines [id:da8964946775679585] 
\draw    (120,130) -- (120,90) ;
\draw [shift={(120,130)}, rotate = 270] [color={rgb, 255:red, 0; green, 0; blue, 0 }  ][fill={rgb, 255:red, 0; green, 0; blue, 0 }  ][line width=0.75]      (0, 0) circle [x radius= 3.35, y radius= 3.35]   ;
%Straight Lines [id:da9199894020232557] 
\draw    (90,90) -- (120,50) ;
\draw [shift={(90,90)}, rotate = 306.87] [color={rgb, 255:red, 0; green, 0; blue, 0 }  ][fill={rgb, 255:red, 0; green, 0; blue, 0 }  ][line width=0.75]      (0, 0) circle [x radius= 3.35, y radius= 3.35]   ;
%Straight Lines [id:da1679945103181566] 
\draw    (60,130) -- (90,90) ;
\draw [shift={(60,130)}, rotate = 306.87] [color={rgb, 255:red, 0; green, 0; blue, 0 }  ][fill={rgb, 255:red, 0; green, 0; blue, 0 }  ][line width=0.75]      (0, 0) circle [x radius= 3.35, y radius= 3.35]   ;
%Straight Lines [id:da12922970835360337] 
\draw    (170,130) -- (150,90) ;
\draw [shift={(170,130)}, rotate = 243.43] [color={rgb, 255:red, 0; green, 0; blue, 0 }  ][fill={rgb, 255:red, 0; green, 0; blue, 0 }  ][line width=0.75]      (0, 0) circle [x radius= 3.35, y radius= 3.35]   ;
%Straight Lines [id:da9545503590957665] 
\draw    (150,130) -- (150,90) ;
\draw [shift={(150,130)}, rotate = 270] [color={rgb, 255:red, 0; green, 0; blue, 0 }  ][fill={rgb, 255:red, 0; green, 0; blue, 0 }  ][line width=0.75]      (0, 0) circle [x radius= 3.35, y radius= 3.35]   ;
%Straight Lines [id:da9596608309067567] 
\draw    (60,130) -- (80,170) ;
\draw [shift={(80,170)}, rotate = 63.43] [color={rgb, 255:red, 0; green, 0; blue, 0 }  ][fill={rgb, 255:red, 0; green, 0; blue, 0 }  ][line width=0.75]      (0, 0) circle [x radius= 3.35, y radius= 3.35]   ;
%Straight Lines [id:da08290293248334879] 
\draw    (40,170) -- (60,130) ;
\draw [shift={(40,170)}, rotate = 296.57] [color={rgb, 255:red, 0; green, 0; blue, 0 }  ][fill={rgb, 255:red, 0; green, 0; blue, 0 }  ][line width=0.75]      (0, 0) circle [x radius= 3.35, y radius= 3.35]   ;
%Straight Lines [id:da7325990054508562] 
\draw    (60,170) -- (60,130) ;
\draw [shift={(60,170)}, rotate = 270] [color={rgb, 255:red, 0; green, 0; blue, 0 }  ][fill={rgb, 255:red, 0; green, 0; blue, 0 }  ][line width=0.75]      (0, 0) circle [x radius= 3.35, y radius= 3.35]   ;
%Straight Lines [id:da08571009498353654] 
\draw    (300,50) -- (340,90) ;
\draw [shift={(300,50)}, rotate = 45] [color={rgb, 255:red, 0; green, 0; blue, 0 }  ][fill={rgb, 255:red, 0; green, 0; blue, 0 }  ][line width=0.75]      (0, 0) circle [x radius= 3.35, y radius= 3.35]   ;
%Straight Lines [id:da7313572207266776] 
\draw    (300,170) -- (300,50) ;
\draw [shift={(300,170)}, rotate = 270] [color={rgb, 255:red, 0; green, 0; blue, 0 }  ][fill={rgb, 255:red, 0; green, 0; blue, 0 }  ][line width=0.75]      (0, 0) circle [x radius= 3.35, y radius= 3.35]   ;
%Straight Lines [id:da2997317176491545] 
\draw    (240,130) -- (300,50) ;
\draw [shift={(240,130)}, rotate = 306.87] [color={rgb, 255:red, 0; green, 0; blue, 0 }  ][fill={rgb, 255:red, 0; green, 0; blue, 0 }  ][line width=0.75]      (0, 0) circle [x radius= 3.35, y radius= 3.35]   ;
%Straight Lines [id:da14288810562309417] 
\draw    (340,90) -- (360,130) ;
\draw [shift={(360,130)}, rotate = 63.43] [color={rgb, 255:red, 0; green, 0; blue, 0 }  ][fill={rgb, 255:red, 0; green, 0; blue, 0 }  ][line width=0.75]      (0, 0) circle [x radius= 3.35, y radius= 3.35]   ;
\draw [shift={(340,90)}, rotate = 63.43] [color={rgb, 255:red, 0; green, 0; blue, 0 }  ][fill={rgb, 255:red, 0; green, 0; blue, 0 }  ][line width=0.75]      (0, 0) circle [x radius= 3.35, y radius= 3.35]   ;
%Straight Lines [id:da9849455018941728] 
\draw    (340,170) -- (340,90) ;
\draw [shift={(340,170)}, rotate = 270] [color={rgb, 255:red, 0; green, 0; blue, 0 }  ][fill={rgb, 255:red, 0; green, 0; blue, 0 }  ][line width=0.75]      (0, 0) circle [x radius= 3.35, y radius= 3.35]   ;
%Straight Lines [id:da7065682153987944] 
\draw    (520,90) -- (490,50) ;
\draw [shift={(490,50)}, rotate = 233.13] [color={rgb, 255:red, 0; green, 0; blue, 0 }  ][fill={rgb, 255:red, 0; green, 0; blue, 0 }  ][line width=0.75]      (0, 0) circle [x radius= 3.35, y radius= 3.35]   ;
\draw [shift={(520,90)}, rotate = 233.13] [color={rgb, 255:red, 0; green, 0; blue, 0 }  ][fill={rgb, 255:red, 0; green, 0; blue, 0 }  ][line width=0.75]      (0, 0) circle [x radius= 3.35, y radius= 3.35]   ;
%Straight Lines [id:da007225304447845415] 
\draw    (430,130) -- (490,50) ;
\draw [shift={(430,130)}, rotate = 306.87] [color={rgb, 255:red, 0; green, 0; blue, 0 }  ][fill={rgb, 255:red, 0; green, 0; blue, 0 }  ][line width=0.75]      (0, 0) circle [x radius= 3.35, y radius= 3.35]   ;
%Straight Lines [id:da9175441570484537] 
\draw    (150,170) -- (150,130) ;
\draw [shift={(150,170)}, rotate = 270] [color={rgb, 255:red, 0; green, 0; blue, 0 }  ][fill={rgb, 255:red, 0; green, 0; blue, 0 }  ][line width=0.75]      (0, 0) circle [x radius= 3.35, y radius= 3.35]   ;
%Straight Lines [id:da9974519629539114] 
\draw    (240,130) -- (260,170) ;
\draw [shift={(260,170)}, rotate = 63.43] [color={rgb, 255:red, 0; green, 0; blue, 0 }  ][fill={rgb, 255:red, 0; green, 0; blue, 0 }  ][line width=0.75]      (0, 0) circle [x radius= 3.35, y radius= 3.35]   ;
%Straight Lines [id:da1526822448844969] 
\draw    (220,170) -- (240,130) ;
\draw [shift={(220,170)}, rotate = 296.57] [color={rgb, 255:red, 0; green, 0; blue, 0 }  ][fill={rgb, 255:red, 0; green, 0; blue, 0 }  ][line width=0.75]      (0, 0) circle [x radius= 3.35, y radius= 3.35]   ;
%Straight Lines [id:da9683625920040162] 
\draw    (240,170) -- (240,130) ;
\draw [shift={(240,170)}, rotate = 270] [color={rgb, 255:red, 0; green, 0; blue, 0 }  ][fill={rgb, 255:red, 0; green, 0; blue, 0 }  ][line width=0.75]      (0, 0) circle [x radius= 3.35, y radius= 3.35]   ;

% Text Node
\draw (111,182.4) node [anchor=north west][inner sep=0.75pt]    {$T$};
% Text Node
\draw (276,182.4) node [anchor=north west][inner sep=0.75pt]    {$\mathcal{BL}( T)$};
% Text Node
\draw (471,182.4) node [anchor=north west][inner sep=0.75pt]    {$\mathcal{B}( T)$};
% Text Node
\draw (111,32.4) node [anchor=north west][inner sep=0.75pt]  [font=\small]  {$v$};
% Text Node
\draw (481,32.4) node [anchor=north west][inner sep=0.75pt]  [font=\small]  {$v$};
% Text Node
\draw (291,32.4) node [anchor=north west][inner sep=0.75pt]  [font=\small]  {$v$};

\end{tikzpicture}

    \caption{The branch-leaf representation and branch representation of a tree $T$.}
    \label{fig:branch_leaf}
\end{figure}
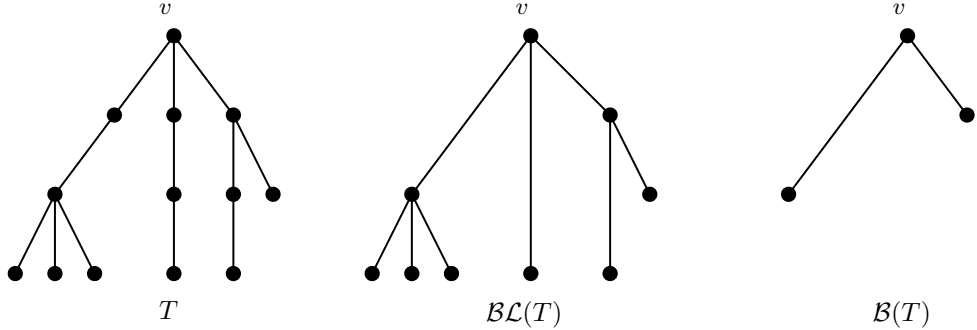

We partition the vertices of $T$ into sets $B_T, L_T, $ and $\tau_T$, where $B_T$ denotes the set of branch vertices of $T$, $L_T$ the set of leaves, and $\tau_T$ the set of vertices of degree 2, called the \textit{trunks} of $T$. We further define a subset $R_T$ of $B_T$ consisting of branch vertices adjacent to at most one leaf in $\mathcal{BL}(T)$. In other words, a branch vertex $b$ belongs to $R_T$ if there exists at most one leaf $l\in L_T$ such that for every $b'\in B_T-b$, the unique $b'-l$ path contains $b$. 
%$d_T(b, l)<d_T(b', l)$. 
In Figure \ref{fig:branch_leaf} above, $|B_T|=3$ and $R_T=\{v\}$.

Maximum boundary independent broadcasts on trees were studied by Neilson in \cite{neilsonphd}, who determined that $\alpha_{bn}(T)\leq n-|B_T|+|R_T|$ for any tree $T$ with $B_T\neq \emptyset$. Neilson asked whether this upper bound could be improved to $\alpha_{bn}(T)\leq n-|B_T|+\alpha(T[R_T])$, where $T[R_T]$ is the subgraph of $T$ induced by $R_T$.

A \textit{caterpillar} is a tree with a diametrical path $D$ such that every vertex lies on $D$ or is adjacent to a vertex on $D$. 
We proceed to show that $\alpha_{bn}(C)\leq n-|B_C|+\alpha(C[R_C])$ for all caterpillars $C$, and determine $\alpha_{bn}(C)$ exactly for certain subclasses of caterpillars. Observe that if $C$ is a caterpillar, $C[R_C]$ is a path or a forest of paths.

The following lemma will be useful throughout this section.

\begin{lemma} \label{lemma:leaf_hears_f2} \textup{\cite{neilsonphd}} If $f$ is an $\alpha_{bn}$-broadcast on a tree $T$, no leaf of $T$ hears a broadcast from any non-leaf vertex. 
\end{lemma}

In particular, if $f$ is a maximum boundary independent broadcast on a caterpillar, only leaves and trunks broadcast.

Given a caterpillar $C$, its \textit{spine} is an (arbitrarily chosen) diametrical path. Label the leaves of $C$ as $l_1, ..., l_m$ such that $l_1$ and $l_m$ are the endpoints the spine and for all $i\leq j$, $d_C(l_1, l_i)\leq d_C(l_1, l_j)$. Call $l_2, ..., l_{m-1}$ the \textit{inner leaves} of $C$.

\begin{lemma}
\label{lemma:cat_endpoints_broadcast}
Let $f$ be an $\alpha_{bn}$-broadcast on a caterpillar $C$ such that $|V_f^1|$ is maximized. Then $l_1, l_m\in V_f^+$.

\end{lemma} 

\begin{proof}
Suppose not. If $l_1\in N_f(v)$ for some $v\in V_f^+$, define a new broadcast $f'$ by $f'(v)=f(v-1)$, $f'(l_1)=1$, and $f'(w)=f(w)$ for all $w\neq v, l_1$. If some vertex $u$ does not hear $f'$, broadcasting at strength 1 from $u$ produces a boundary independent broadcast of greater cost than $f$, a contradiction. Therefore $f'$ is an $\alpha_{bn}$-broadcast on $C$. By Lemma \ref{lemma:leaf_hears_f2}, $v$ is a leaf, hence $d_C(l_1, v)\geq 2$ and so $v\notin V_f^1$. But then $|V_{f'}^1|>|V_f^1|$, a contradiction. 

It follows that $l_1\in V_f^+$. Similarly, $l_m\in V_f^+$. 
\end{proof}

\begin{theorem}
\label{theorem:caterpillars} 
If $C$ is a caterpillar, then $\alpha_{bn}(C)\leq n-|B_C|+\alpha(C[R_C])$.
\end{theorem}

\begin{proof}

Let $f$ be an $\alpha_{bn}$-broadcast on $C$ such that $|V_f^1|$ is maximized, and let $l_i$ be an inner leaf of $C$ adjacent to a branch vertex $b_i$. 
If $f(l_i)\geq 3$, then since $l_1$ and $l_m$ are broadcasting, $l_i$ covers at least $2(f(l_i)-1)$ edges on the spine. 

If $f(l_i)$ is odd, let $f'$ be a broadcast defined by $f'(l_i)=1$, $f'(v)=1$ for all $v$ at an even distance from $l_i$, and $f'(v)=f(v)$ for all $v\notin N_f[l_i]$. 

If $f(l_i)$ is even, define $f'$ by $f'(l_i)=2$, $f'(v)=1$ for all $v$ at odd distance at least 3 from $l_i$, and $f'(v)=f(v)$ for all $v\notin N_f[l_i]$.

Figure \ref{fig:even_odd} illustrates the construction of $f'$.

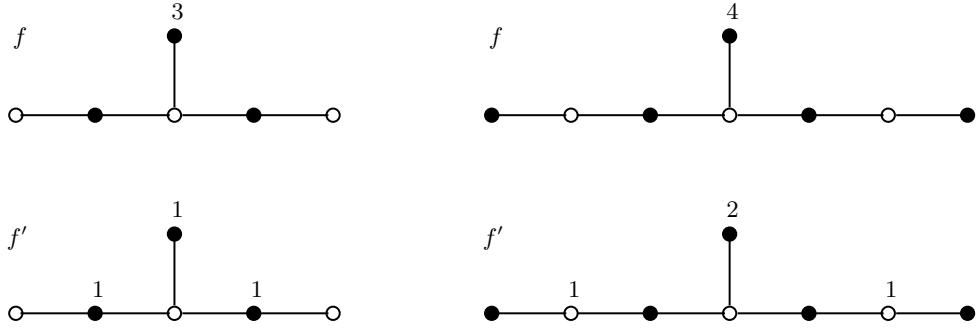
\begin{figure}[H]
    \centering

\tikzset{every picture/.style={line width=0.75pt}} 
\begin{tikzpicture}[x=0.75pt,y=0.75pt,yscale=-1,xscale=1]

%Straight Lines [id:da2303585626077811] 
\draw    (280,130) -- (317.65,130) ;
\draw [shift={(320,130)}, rotate = 0] [color={rgb, 255:red, 0; green, 0; blue, 0 }  ][line width=0.75]      (0, 0) circle [x radius= 3.35, y radius= 3.35]   ;
\draw [shift={(280,130)}, rotate = 0] [color={rgb, 255:red, 0; green, 0; blue, 0 }  ][fill={rgb, 255:red, 0; green, 0; blue, 0 }  ][line width=0.75]      (0, 0) circle [x radius= 3.35, y radius= 3.35]   ;
%Straight Lines [id:da09957425012471455] 
\draw    (360,130) -- (397.65,130) ;
\draw [shift={(400,130)}, rotate = 0] [color={rgb, 255:red, 0; green, 0; blue, 0 }  ][line width=0.75]      (0, 0) circle [x radius= 3.35, y radius= 3.35]   ;
\draw [shift={(360,130)}, rotate = 0] [color={rgb, 255:red, 0; green, 0; blue, 0 }  ][fill={rgb, 255:red, 0; green, 0; blue, 0 }  ][line width=0.75]      (0, 0) circle [x radius= 3.35, y radius= 3.35]   ;
%Straight Lines [id:da48029294751570983] 
\draw    (440,130) -- (477.65,130) ;
\draw [shift={(480,130)}, rotate = 0] [color={rgb, 255:red, 0; green, 0; blue, 0 }  ][line width=0.75]      (0, 0) circle [x radius= 3.35, y radius= 3.35]   ;
\draw [shift={(440,130)}, rotate = 0] [color={rgb, 255:red, 0; green, 0; blue, 0 }  ][fill={rgb, 255:red, 0; green, 0; blue, 0 }  ][line width=0.75]      (0, 0) circle [x radius= 3.35, y radius= 3.35]   ;
%Straight Lines [id:da7703368938538104] 
\draw    (323,130) -- (359,130) ;
%Straight Lines [id:da0940836690295912] 
\draw    (404,130) -- (440,130) ;
%Straight Lines [id:da2587599212630156] 
\draw    (484,130) -- (520,130) ;
\draw [shift={(520,130)}, rotate = 0] [color={rgb, 255:red, 0; green, 0; blue, 0 }  ][fill={rgb, 255:red, 0; green, 0; blue, 0 }  ][line width=0.75]      (0, 0) circle [x radius= 3.35, y radius= 3.35]   ;
%Straight Lines [id:da6461910608518426] 
\draw    (400,126) -- (400,90) ;
\draw [shift={(400,90)}, rotate = 270] [color={rgb, 255:red, 0; green, 0; blue, 0 }  ][fill={rgb, 255:red, 0; green, 0; blue, 0 }  ][line width=0.75]      (0, 0) circle [x radius= 3.35, y radius= 3.35]   ;
%Straight Lines [id:da2383423066864554] 
\draw    (280,230) -- (317.65,230) ;
\draw [shift={(320,230)}, rotate = 0] [color={rgb, 255:red, 0; green, 0; blue, 0 }  ][line width=0.75]      (0, 0) circle [x radius= 3.35, y radius= 3.35]   ;
\draw [shift={(280,230)}, rotate = 0] [color={rgb, 255:red, 0; green, 0; blue, 0 }  ][fill={rgb, 255:red, 0; green, 0; blue, 0 }  ][line width=0.75]      (0, 0) circle [x radius= 3.35, y radius= 3.35]   ;
%Straight Lines [id:da6711838365442209] 
\draw    (360,230) -- (397.65,230) ;
\draw [shift={(400,230)}, rotate = 0] [color={rgb, 255:red, 0; green, 0; blue, 0 }  ][line width=0.75]      (0, 0) circle [x radius= 3.35, y radius= 3.35]   ;
\draw [shift={(360,230)}, rotate = 0] [color={rgb, 255:red, 0; green, 0; blue, 0 }  ][fill={rgb, 255:red, 0; green, 0; blue, 0 }  ][line width=0.75]      (0, 0) circle [x radius= 3.35, y radius= 3.35]   ;
%Straight Lines [id:da367789242856025] 
\draw    (440,230) -- (477.65,230) ;
\draw [shift={(480,230)}, rotate = 0] [color={rgb, 255:red, 0; green, 0; blue, 0 }  ][line width=0.75]      (0, 0) circle [x radius= 3.35, y radius= 3.35]   ;
\draw [shift={(440,230)}, rotate = 0] [color={rgb, 255:red, 0; green, 0; blue, 0 }  ][fill={rgb, 255:red, 0; green, 0; blue, 0 }  ][line width=0.75]      (0, 0) circle [x radius= 3.35, y radius= 3.35]   ;
%Straight Lines [id:da27578886758368637] 
\draw    (323,230) -- (359,230) ;
%Straight Lines [id:da6134185221554112] 
\draw    (404,230) -- (440,230) ;
%Straight Lines [id:da3197980571071821] 
\draw    (484,230) -- (520,230) ;
\draw [shift={(520,230)}, rotate = 0] [color={rgb, 255:red, 0; green, 0; blue, 0 }  ][fill={rgb, 255:red, 0; green, 0; blue, 0 }  ][line width=0.75]      (0, 0) circle [x radius= 3.35, y radius= 3.35]   ;
%Straight Lines [id:da8847507455681567] 
\draw    (400,226) -- (400,190) ;
\draw [shift={(400,190)}, rotate = 270] [color={rgb, 255:red, 0; green, 0; blue, 0 }  ][fill={rgb, 255:red, 0; green, 0; blue, 0 }  ][line width=0.75]      (0, 0) circle [x radius= 3.35, y radius= 3.35]   ;
%Straight Lines [id:da5661001187068431] 
\draw    (80,130) -- (117.65,130) ;
\draw [shift={(120,130)}, rotate = 0] [color={rgb, 255:red, 0; green, 0; blue, 0 }  ][line width=0.75]      (0, 0) circle [x radius= 3.35, y radius= 3.35]   ;
\draw [shift={(80,130)}, rotate = 0] [color={rgb, 255:red, 0; green, 0; blue, 0 }  ][fill={rgb, 255:red, 0; green, 0; blue, 0 }  ][line width=0.75]      (0, 0) circle [x radius= 3.35, y radius= 3.35]   ;
%Straight Lines [id:da48077435002084457] 
\draw    (160,130) -- (197.65,130) ;
\draw [shift={(200,130)}, rotate = 0] [color={rgb, 255:red, 0; green, 0; blue, 0 }  ][line width=0.75]      (0, 0) circle [x radius= 3.35, y radius= 3.35]   ;
\draw [shift={(160,130)}, rotate = 0] [color={rgb, 255:red, 0; green, 0; blue, 0 }  ][fill={rgb, 255:red, 0; green, 0; blue, 0 }  ][line width=0.75]      (0, 0) circle [x radius= 3.35, y radius= 3.35]   ;
%Straight Lines [id:da4119523179436442] 
\draw    (42.35,130) -- (79,130) ;
\draw [shift={(40,130)}, rotate = 0] [color={rgb, 255:red, 0; green, 0; blue, 0 }  ][line width=0.75]      (0, 0) circle [x radius= 3.35, y radius= 3.35]   ;
%Straight Lines [id:da41943400772605166] 
\draw    (124,130) -- (160,130) ;
%Straight Lines [id:da2580518973187742] 
\draw    (120,126) -- (120,90) ;
\draw [shift={(120,90)}, rotate = 270] [color={rgb, 255:red, 0; green, 0; blue, 0 }  ][fill={rgb, 255:red, 0; green, 0; blue, 0 }  ][line width=0.75]      (0, 0) circle [x radius= 3.35, y radius= 3.35]   ;
%Straight Lines [id:da23124216926460406] 
\draw    (80,230) -- (117.65,230) ;
\draw [shift={(120,230)}, rotate = 0] [color={rgb, 255:red, 0; green, 0; blue, 0 }  ][line width=0.75]      (0, 0) circle [x radius= 3.35, y radius= 3.35]   ;
\draw [shift={(80,230)}, rotate = 0] [color={rgb, 255:red, 0; green, 0; blue, 0 }  ][fill={rgb, 255:red, 0; green, 0; blue, 0 }  ][line width=0.75]      (0, 0) circle [x radius= 3.35, y radius= 3.35]   ;
%Straight Lines [id:da7432009314493073] 
\draw    (160,230) -- (197.65,230) ;
\draw [shift={(200,230)}, rotate = 0] [color={rgb, 255:red, 0; green, 0; blue, 0 }  ][line width=0.75]      (0, 0) circle [x radius= 3.35, y radius= 3.35]   ;
\draw [shift={(160,230)}, rotate = 0] [color={rgb, 255:red, 0; green, 0; blue, 0 }  ][fill={rgb, 255:red, 0; green, 0; blue, 0 }  ][line width=0.75]      (0, 0) circle [x radius= 3.35, y radius= 3.35]   ;
%Straight Lines [id:da28718242652449755] 
\draw    (42.35,230) -- (79,230) ;
\draw [shift={(40,230)}, rotate = 0] [color={rgb, 255:red, 0; green, 0; blue, 0 }  ][line width=0.75]      (0, 0) circle [x radius= 3.35, y radius= 3.35]   ;
%Straight Lines [id:da11648641996149811] 
\draw    (124,230) -- (160,230) ;
%Straight Lines [id:da19493370023294188] 
\draw    (120,226) -- (120,190) ;
\draw [shift={(120,190)}, rotate = 270] [color={rgb, 255:red, 0; green, 0; blue, 0 }  ][fill={rgb, 255:red, 0; green, 0; blue, 0 }  ][line width=0.75]      (0, 0) circle [x radius= 3.35, y radius= 3.35]   ;

% Text Node
\draw (397,72.4) node [anchor=north west][inner sep=0.75pt]  [font=\small]  {$4$};
% Text Node
\draw (117,72.4) node [anchor=north west][inner sep=0.75pt]  [font=\small]  {$3$};
% Text Node
\draw (117,172.4) node [anchor=north west][inner sep=0.75pt]  [font=\small]  {$1$};
% Text Node
\draw (317,212.4) node [anchor=north west][inner sep=0.75pt]  [font=\small]  {$1$};
% Text Node
\draw (157,212.4) node [anchor=north west][inner sep=0.75pt]  [font=\small]  {$1$};
% Text Node
\draw (77,212.4) node [anchor=north west][inner sep=0.75pt]  [font=\small]  {$1$};
% Text Node
\draw (477,212.4) node [anchor=north west][inner sep=0.75pt]  [font=\small]  {$1$};
% Text Node
\draw (397,172.4) node [anchor=north west][inner sep=0.75pt]  [font=\small]  {$2$};
% Text Node
\draw (37,84.4) node [anchor=north west][inner sep=0.75pt]  [font=\small]  {$f$};
% Text Node
\draw (277,84.4) node [anchor=north west][inner sep=0.75pt]  [font=\small]  {$f$};
% Text Node
\draw (34,184.4) node [anchor=north west][inner sep=0.75pt]  [font=\small]  {$f'$};
% Text Node
\draw (274,184.4) node [anchor=north west][inner sep=0.75pt]  [font=\small]  {$f'$};

\end{tikzpicture}

    \caption{The odd and even cases. }
    \label{fig:even_odd}
\end{figure}

It follows that $f'$ is a boundary independent broadcast on $C$ such that $\sigma(f')= \sigma(f)$ and $|V_{f'}^1|>|V_f^1|$, a contradiction. Thus, $f(l_i)\leq 2$ for each inner leaf $l_i$. Since $f$ is boundary independent, the set of branch vertices adjacent to leaves that broadcast at strength 2 must form an independent set. 

By Lemma 8.1, no branch vertices broadcast under $f$. Let $C'$ denote the subgraph of $C$ induced by removing all inner leaves and branch vertices, and let $f_{C'}$ denote the restriction of $f$ to $C'$. 
Since $f_{C'}$ is boundary independent, $\sigma(f_{C'})\leq \alpha_{bn}(C')\leq |V(C')|-\delta(C')=|V(C')|-1$.
Taking the sum over all broadcasting vertices, we find that 
\\
${\alpha_{bn}(C)\leq (|V(C')|-1)+(L_C-2)+\alpha(C[R_C])\leq n-|B_C|+\alpha(C[R_C])}$. \end{proof}

\begin{corollary}
Let $C$ be a caterpillar with $|V(C)|\geq 3$.
\begin{itemize}
    \item[i.] If $\tau_C=\emptyset$, then $\alpha_{bn}(C)=|L_C|+\alpha(C[R_C])$.
    \item[ii.]  If $C$ has no two adjacent trunks and no vertices of degree 3, then $\alpha_{bn}(C)=|L_C|+|\tau_C|$.
\end{itemize}
\end{corollary}

\begin{proof}

Suppose $\tau_C=\emptyset$, so that every vertex of $C$ is either a branch vertex or a leaf. By Theorem  \ref{theorem:caterpillars}, $\alpha_{bn}(C)\leq n-|B_C|+\alpha(C[R_C]) = |L_C|+\alpha(C[R_C])$. Observe that $C[R_C]$ is a forest of paths, and let $S$ be a maximum independent set of $C[R_C]$. 
Define a broadcast $f$ on $C$ by 

\[
f(v)=\left\{
\begin{tabular}
[c]{rl}%
2 & if $v\in L_T$ and $v$ is adjacent to a vertex in $S$
\\
1 & if $v\in L_T$ and $v$ is not adjacent to a vertex in $S$
\\
0 & if $v\notin L_T$.
\end{tabular}
\right.
\]

As only leaves broadcast, two broadcasts from $l_i$ and $l_j$ may overlap only if one or both vertices broadcast at strength 2. 
Suppose they do. Without loss of generality, assume $f(l_i)=2$. Let $b_i$ be the neighbour of $l_i$ in $S$. But then since $b_i$ has only one leaf, $f(l_j)=2$ and $b_j\in S$, hence $d(l_i, l_j)\geq 4$. It follows that $f$ is a boundary independent broadcast on $C$. Thus, $\sigma(f)=|L_C|+\alpha(C[R_C])\leq \alpha_{bn}(C)$.

Suppose instead that $C$ has no two adjacent trunks and no vertices of degree 3. Since $R_C=\emptyset$, $\alpha_{bn}(C)\leq n-|B_C|+\alpha(C[R_C]) = |L_C|+\tau_C$ by Theorem  \ref{theorem:caterpillars}.
Define a broadcast $f$ on $C$ by 

\[
f(v)=\left\{
\begin{tabular}
[c]{rl}%
1 & if $v\in L_C$ or $v\in \tau_C$
\\
0 & otherwise.
\end{tabular}
\right.
\]

By definition, no two broadcasting vertices are adjacent. It follows that $f$ is a boundary independent broadcast, hence $\sigma(f)= |L_C|+|\tau_C|\leq \alpha_{bn}(C)$.
\end{proof}

\section{A polynomial time algorithm for trees}

In Section 7, we found that the problem of determining the cost of a bn-independent broadcast on a general graph $G$ is NP-complete.
However, the complexity of boundary independence in trees is unknown.

In \cite{bessy}, Bessy and Rautenbach proved that the hearing independence number $\alpha_h(T)$ can be determined for a tree $T$ in $O(n^9)$ time. In this section, we show that their algorithm can be modified to determine the boundary independence number $\alpha_{bn}(T)$ in $O(n^9)$ time.

A \textit{rooted tree} $(T, r)$ is a tree in which a distinguished vertex $r$ serves as a point of reference for all vertices of $T$. A vertex $v$ is said to be a \textit{descendant} of $u$ if $u$ lies along the unique $r-v$ path, in which case $u$ is an \textit{ancestor} of $v$. The descendants adjacent to $u$ are known as the \textit{children} of $u$. 

\subsection{Definitions and Notation}

Let $T$ be a tree of order $n\geq 3$ in which an arbitrary non-leaf vertex $r$ is chosen to be the root.
%Let $L(T)$ denote the set of leaves of $T$, and 
For each $u\in V(T)$, fix an arbitrary linear order on its children. If $v_1, ..., v_k$ are the children of $u$ in this linear order, define $T_{u, i}$ as the subtree of $T$ induced by $u$ and all vertices $w$ such that the unique $u-w$ paths contains one of $v_1, ..., v_i$ (see Figure \ref{fig:tui}).
%, and all descendants of $v_1, ..., v_i$ (see Figure 5.1).
In addition, define $T_{u, 0}$ as the subtree consisting only of the vertex $u$. If $\text{c}_T(v)$ denotes the number of children of $u$ in $T$ rooted at $r$, then, in total, there are at most
\begin{equation}
    \sum\limits_{u\in V(T)} (\text{c}_T(u)+1) 
    = deg(r)+1+\sum\limits_{u\in V(T)\backslash \{r\}}deg(u)
    = 2n-1
    %< n\Delta(T)+n 
    =O(n)
\end{equation} such subtrees $T_{u, i}$. (Note that if $u$ is a leaf, then $c_T(u)=0$ and the only subtree counted is $T_{u, 0}=\{u\}$).

\begin{figure}[H]
    \centering

\tikzset{every picture/.style={line width=0.75pt}} %set default line width to 0.75pt        

\begin{tikzpicture}[x=0.75pt,y=0.75pt,yscale=-1,xscale=1]

%Straight Lines [id:da43182035758001924] 
\draw    (290,50) -- (240,90) -- (200,130) -- (190,170) ;
\draw [shift={(190,170)}, rotate = 104.04] [color={rgb, 255:red, 0; green, 0; blue, 0 }  ][fill={rgb, 255:red, 0; green, 0; blue, 0 }  ][line width=0.75]      (0, 0) circle [x radius= 3.35, y radius= 3.35]   ;
\draw [shift={(290,50)}, rotate = 141.34] [color={rgb, 255:red, 0; green, 0; blue, 0 }  ][fill={rgb, 255:red, 0; green, 0; blue, 0 }  ][line width=0.75]      (0, 0) circle [x radius= 3.35, y radius= 3.35]   ;
%Straight Lines [id:da5435029762019488] 
\draw    (290,50) -- (340,90) ;
\draw [shift={(340,90)}, rotate = 38.66] [color={rgb, 255:red, 0; green, 0; blue, 0 }  ][fill={rgb, 255:red, 0; green, 0; blue, 0 }  ][line width=0.75]      (0, 0) circle [x radius= 3.35, y radius= 3.35]   ;
\draw [shift={(290,50)}, rotate = 38.66] [color={rgb, 255:red, 0; green, 0; blue, 0 }  ][fill={rgb, 255:red, 0; green, 0; blue, 0 }  ][line width=0.75]      (0, 0) circle [x radius= 3.35, y radius= 3.35]   ;
%Straight Lines [id:da35584131723670276] 
\draw    (200,130) -- (210,170) ;
\draw [shift={(210,170)}, rotate = 75.96] [color={rgb, 255:red, 0; green, 0; blue, 0 }  ][fill={rgb, 255:red, 0; green, 0; blue, 0 }  ][line width=0.75]      (0, 0) circle [x radius= 3.35, y radius= 3.35]   ;
\draw [shift={(200,130)}, rotate = 75.96] [color={rgb, 255:red, 0; green, 0; blue, 0 }  ][fill={rgb, 255:red, 0; green, 0; blue, 0 }  ][line width=0.75]      (0, 0) circle [x radius= 3.35, y radius= 3.35]   ;
%Straight Lines [id:da29008798228095856] 
\draw    (240,90) -- (240,130) ;
\draw [shift={(240,130)}, rotate = 90] [color={rgb, 255:red, 0; green, 0; blue, 0 }  ][fill={rgb, 255:red, 0; green, 0; blue, 0 }  ][line width=0.75]      (0, 0) circle [x radius= 3.35, y radius= 3.35]   ;
\draw [shift={(240,90)}, rotate = 90] [color={rgb, 255:red, 0; green, 0; blue, 0 }  ][fill={rgb, 255:red, 0; green, 0; blue, 0 }  ][line width=0.75]      (0, 0) circle [x radius= 3.35, y radius= 3.35]   ;
%Straight Lines [id:da8130261312235743] 
\draw    (240,90) -- (280,130) ;
\draw [shift={(280,130)}, rotate = 45] [color={rgb, 255:red, 0; green, 0; blue, 0 }  ][fill={rgb, 255:red, 0; green, 0; blue, 0 }  ][line width=0.75]      (0, 0) circle [x radius= 3.35, y radius= 3.35]   ;
%Straight Lines [id:da18248346063223586] 
\draw    (270,210) -- (280,170) -- (290,210) ;
\draw [shift={(290,210)}, rotate = 75.96] [color={rgb, 255:red, 0; green, 0; blue, 0 }  ][fill={rgb, 255:red, 0; green, 0; blue, 0 }  ][line width=0.75]      (0, 0) circle [x radius= 3.35, y radius= 3.35]   ;
\draw [shift={(270,210)}, rotate = 284.04] [color={rgb, 255:red, 0; green, 0; blue, 0 }  ][fill={rgb, 255:red, 0; green, 0; blue, 0 }  ][line width=0.75]      (0, 0) circle [x radius= 3.35, y radius= 3.35]   ;
%Straight Lines [id:da2890070672393301] 
\draw    (280,130) -- (280,170) ;
\draw [shift={(280,170)}, rotate = 90] [color={rgb, 255:red, 0; green, 0; blue, 0 }  ][fill={rgb, 255:red, 0; green, 0; blue, 0 }  ][line width=0.75]      (0, 0) circle [x radius= 3.35, y radius= 3.35]   ;
%Straight Lines [id:da602423775411832] 
\draw    (240,130) -- (240,170) ;
\draw [shift={(240,170)}, rotate = 90] [color={rgb, 255:red, 0; green, 0; blue, 0 }  ][fill={rgb, 255:red, 0; green, 0; blue, 0 }  ][line width=0.75]      (0, 0) circle [x radius= 3.35, y radius= 3.35]   ;
%Straight Lines [id:da0007886720027521221] 
\draw    (340,130) -- (340,90) -- (380,130) ;
\draw [shift={(380,130)}, rotate = 45] [color={rgb, 255:red, 0; green, 0; blue, 0 }  ][fill={rgb, 255:red, 0; green, 0; blue, 0 }  ][line width=0.75]      (0, 0) circle [x radius= 3.35, y radius= 3.35]   ;
\draw [shift={(340,130)}, rotate = 270] [color={rgb, 255:red, 0; green, 0; blue, 0 }  ][fill={rgb, 255:red, 0; green, 0; blue, 0 }  ][line width=0.75]      (0, 0) circle [x radius= 3.35, y radius= 3.35]   ;
%Straight Lines [id:da3535575960265509] 
\draw    (360,170) -- (380,130) -- (400,170) ;
\draw [shift={(400,170)}, rotate = 63.43] [color={rgb, 255:red, 0; green, 0; blue, 0 }  ][fill={rgb, 255:red, 0; green, 0; blue, 0 }  ][line width=0.75]      (0, 0) circle [x radius= 3.35, y radius= 3.35]   ;
\draw [shift={(360,170)}, rotate = 296.57] [color={rgb, 255:red, 0; green, 0; blue, 0 }  ][fill={rgb, 255:red, 0; green, 0; blue, 0 }  ][line width=0.75]      (0, 0) circle [x radius= 3.35, y radius= 3.35]   ;
%Straight Lines [id:da11189048889273878] 
\draw    (380,130) -- (380,149.8) -- (380,170) ;
\draw [shift={(380,170)}, rotate = 90] [color={rgb, 255:red, 0; green, 0; blue, 0 }  ][fill={rgb, 255:red, 0; green, 0; blue, 0 }  ][line width=0.75]      (0, 0) circle [x radius= 3.35, y radius= 3.35]   ;
%Shape: Polygon [id:ds3803745058209991] 
\draw  [dash pattern={on 0.84pt off 2.51pt}] (150,70) -- (150,230) -- (300,230) -- (300,70) -- cycle ;
%Shape: Polygon [id:ds8187168626057397] 
\draw  [dash pattern={on 0.84pt off 2.51pt}] (170,78) -- (170,180) -- (260,180) -- (260,78) -- cycle ;

% Text Node
\draw (275,40.4) node [anchor=north west][inner sep=0.75pt]  [font=\footnotesize]  {${\textstyle r}$};
% Text Node
\draw (190,113.4) node [anchor=north west][inner sep=0.75pt]  [font=\footnotesize]  {$v_{1}$};
% Text Node
\draw (221,81.4) node [anchor=north west][inner sep=0.75pt]  [font=\footnotesize]  {$u$};
% Text Node
\draw (224,113.4) node [anchor=north west][inner sep=0.75pt]  [font=\footnotesize]  {$v_{2}$};
% Text Node
\draw (276,113.4) node [anchor=north west][inner sep=0.75pt]  [font=\footnotesize]  {$v_{3}$};
% Text Node
\draw (172,183.4) node [anchor=north west][inner sep=0.75pt]    {$T_{u, 2}$};
% Text Node
\draw (152,233.4) node [anchor=north west][inner sep=0.75pt]    {$T_{u, 3}$};

\end{tikzpicture}
    \caption{A tree $T$ rooted at $r$ and a vertex $u$ with three children.}
    \label{fig:tui}
\end{figure}
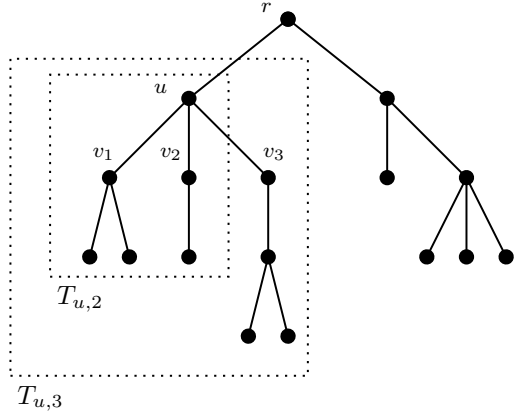

Bessy and Rautenbach's algorithm, based on dynamic programming, successively considers subtrees $T_{u, i}$ such that the vertices $u$ are ordered by nonincreasing distance from the root $r$.

Given a boundary independent broadcast $f$ on $T$, its restriction to $T_{u, i}$ satisfies the following conditions:
\begin{itemize}
    \item[(C1)]  $f(x)\leq e_T(x)$ for all $x\in V(T_{u, i})$ 
    \item[(C2)] $d_T(x, y)\geq f(x)+f(y)$ for every two distinct broadcasting vertices $x, y\in V(T_{u, i})$.
\end{itemize}

As in \cite{bessy}, we find that if $f(y)>0$ for some 
$y\in V(T\backslash T_{u, i})$, then $y$ imposes upper bounds
on $f(x)$ for all $x\in V(T_{u, i})$. Specifically, by C2, if $d_T(x, y)\leq f(y)$, then $x\in N_f(y)$ and hence $f(x)=0$. If $f(y)<d_T(x, y)$, then $f(x)\leq d_T(x, y)-f(y)$. 

We may express this upper bound as a function $g_{p, q}(d_T(u, x))$ as follows. First, consider a function $h_{t}:\mathbb{Z}\rightarrow \mathbb{N}_0$ such that $h_t(d)=\text{max}\{0, d-t\}$.

For a vertex $y\in V_f^+$, let $q=f(y)$, $p=f(y)-d_T(u, y)$, and define $g_{p, q}: \mathbb{Z} \rightarrow \mathbb{N}_0$ by $g_{p, q}(d)=\text{max}\{0, d-p\}$.
%\begin{align*}
%      g_{p, q}(d)&=\begin{cases}
% &0 \text{ if } d\leq p \\
% &d-p \text{ if } d\geq p+1.
%\end{cases}
%\end{align*}
Note that $g_{p, q}$ is equivalent to $h_t$ under the restriction $t=f(y)-d_T(u, y)$, and so we have the bound $f(x)\leq g_{p, q}(d_T(u, x))$. Figure \ref{fig:f_visualization} below motivates the choice of $p$ and $q.$

\begin{figure}[H]
    \centering

\tikzset{every picture/.style={line width=0.75pt}} %set default line width to 0.75pt        

\begin{tikzpicture}[x=0.75pt,y=0.75pt,yscale=-1,xscale=1]

%Straight Lines [id:da030960875956882683] 
\draw    (190,63) -- (190,210) ;
\draw [shift={(190,60)}, rotate = 90] [fill={rgb, 255:red, 0; green, 0; blue, 0 }  ][line width=0.08]  [draw opacity=0] (8.93,-4.29) -- (0,0) -- (8.93,4.29) -- cycle    ;
%Straight Lines [id:da5804199996199577] 
\draw    (170,190) -- (407,190) ;
\draw [shift={(410,190)}, rotate = 180] [fill={rgb, 255:red, 0; green, 0; blue, 0 }  ][line width=0.08]  [draw opacity=0] (8.93,-4.29) -- (0,0) -- (8.93,4.29) -- cycle    ;
%Straight Lines [id:da15884431337921234] 
\draw [line width=1.5]    (190,190) -- (271,190) ;
%Straight Lines [id:da3002460817081858] 
\draw [line width=1.5]    (271,190) -- (390,70) ;
%Straight Lines [id:da550689523337426] 
\draw    (183,150) -- (197,150) ;
%Straight Lines [id:da20565504214838404] 
\draw  [dash pattern={on 0.84pt off 2.51pt}]  (197,150) -- (310,150) -- (310,183) ;
%Straight Lines [id:da8243467408846956] 
\draw    (270,183) -- (270,197) ;
%Straight Lines [id:da0032228476325888433] 
\draw    (310,183) -- (310,197) ;

% Text Node
\draw (416,182.4) node [anchor=north west][inner sep=0.75pt]    {$d$};
% Text Node
\draw (172,143.4) node [anchor=north west][inner sep=0.75pt]  [font=\footnotesize]  {$1$};
% Text Node
\draw (169,39.4) node [anchor=north west][inner sep=0.75pt]    {$g_{p, q}( d)$};
% Text Node
\draw (253,202.4) node [anchor=north west][inner sep=0.75pt]  [font=\footnotesize]  {$f( y)$};
% Text Node
\draw (291,202.4) node [anchor=north west][inner sep=0.75pt]  [font=\footnotesize]  {$f( y) +1$};

\end{tikzpicture}

    \caption{When $d=d_T(x, y)\geq f(y)+1$, the upper bound on $f(x)$ imposed by $g_{p, q}(d)$ increases linearly as a function of $d$.}
    \label{fig:f_visualization}
\end{figure}
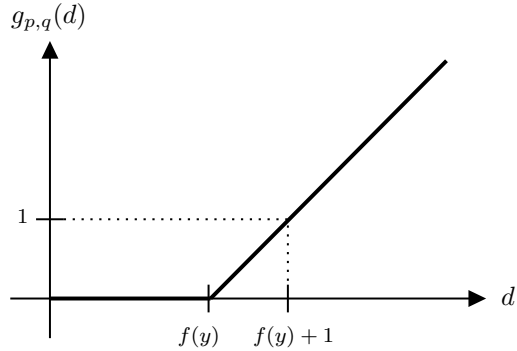
Define integers $p_{in},\ q_{in}$ and $y_{in}$ as follows. If $V_{f}^{+}\backslash V(T_{u, i})$ is empty, let $p_{in}=-n$ and $q_{in}=1$. Otherwise,
let $p_{in}=\max\{p=f(y)-d_T(u, y) :y\in V_{f}^{+}\backslash V(T_{u, i})\}$. Let
$y_{in}\in V_{f}^{+}\backslash V(T_{u, i})$ be a vertex for which this
maximum holds. If $y_{in}$ overdominates $u$, then, by bn-independence,
$y_{in}$ is the unique such vertex; otherwise, there may be more than one such
vertex and we choose $y_{in}$ arbitrarily. Let $q_{in}=f(y_{in})$. Note that

\begin{equation}
(p_{in},q_{in})=\left\{
\begin{tabular}
[c]{ll}%
%\left\{
%\begin{align*}
%[c]{rl}%

$(-n,1)$ & $\text{ if } V_{f}^{+}\cap V(T\setminus T_{u, i})=\varnothing$\\
$(f(y_{in})-d_{T}(u,y_{in}),f(y_{in}))$ & $\text{ otherwise.}$
\end{tabular}
\right.  \label{eq_pin}
\end{equation}

Since $d_{T}(u,y_{in})\geq0$, it follows that $p_{in}\leq q_{in}$ in either case.

Among all broadcasting vertices in $T\backslash T_{u, i}$, $y_{in}$ imposes the strictest upper bound on $f(x)$ for all $x\in T_{u, i}$. That is, for all $x\in V(T_{u, i})$ and all $q=f(y)$, where $y\in V_{f}^{+}\backslash V(T_{u, i})$, we have that $f(x)\leq g_{p_{in}, q_{in}}(d_T(u, x))\leq g_{p, q}(d_T(u, x))$.

The existence of $p_{in}$ and $q_{in}$ implies that bounds on $f(x)$ for all $x\in V(T_{u, i})$ can be encoded with just these two values. Symmetrically,
the upper bounds on the possible
values of $f$ on the vertices of $T\backslash T_{u, i}$ that are imposed by broadcasting vertices in $T_{u, i}$,
again expressed as a function of the distance from $u$ in $T$, can be encoded with two integers
$p_{out}$ and $q_{out}$. That is, if
$V_{f}^{+}\cap V(T_{u, i})=\varnothing$, let $p_{out}=-n$ and $q_{out}=1$;
otherwise, let $p_{out}=\max\{p:y\in V_{f}^{+}\cap V(T_{u, i})\}$, and choose
$y_{out}$ similar to $y_{in}$. Then%
\begin{equation} 
(p_{out},q_{out})=\left\{
\begin{tabular}
[c]{ll}%
$(-n,1)$ & if $V_{f}^{+}\cap V(T_{u, i})=\varnothing$\\
$(f(y_{out})-d_{T}(u,y_{out}),f(y_{out}))$ & otherwise,
\end{tabular}
\right.  \label{eq_pout}
\end{equation}
and $p_{out}\leq q_{out}$ in either case.

For all $O(n^4)$ possible choices for $p_{in}, q_{in}, p_{out}, q_{out}$ with  $-n\leq p_{in}, p_{out} \leq n$ and $1\leq q_{in}, q_{out}\leq n$, the algorithm determines the maximum contribution $\sum\limits_{x\in V(T_{u, i})}f(x) $ satisfying:

(C3) $f(x)\leq g_{p_{in}, q_{in}}(d_T(u, x))$ for every vertex $x$ of $T_{u, i}$.

(C4) If $f(y)>0$ for some vertex $y$ of $T_{u, i}$, then 
%\begin{center}
    $ g_{p_{out}, q_{out}}(d)\leq g_{f(y)-d_T(u, y), f(y)}(d)$
    %d+d_T(u, y)-f(y)$ 
    for every positive integer $d$.
%\end{center}

\subsection{The Algorithm}

The following lemma shows we can check C4 in linear time for a given vertex $y$.

\begin{lemma}
\label{lemma:501}
If $t, f$ and \textup{dist} are integers such that $-n\leq t\leq n$ and $f,$ \textup{dist} $\in [n]$, then
%\begin{center}
    $h_{t}(d)\leq h_{f-\textup{dist}}(d)$ for every positive integer $d$
%\end{center}
if and only if $\textup{dist}\geq \textup{max}\{f, f-t\}$.
\end{lemma}

\begin{proof}
Suppose $t\leq 0$. Since $d$ is positive, $d\geq t+1$, hence $h_t(d)=d-t$. Moreover,

\begin{align*}
    h_{f-\textup{dist}}(d)&=\begin{cases}
 &0 \text{ if } d\leq f-\textup{dist} \\
 &d-f+\textup{dist} \text{ if } d\geq f-\textup{dist}+1. 
\end{cases}
\end{align*}

\noindent Therefore, $h_t(d) \leq h_{f-\text{dist}}(d)$ if and only if 

(i) $d\geq f-\text{dist}+1$, and 

(ii) $d-t\leq d-f+\text{dist}$

\noindent for all positive integers $d$. With $d=1$, (i) becomes $\text{dist}\geq f$. Therefore, (i) and (ii) together are equivalent to $\text{dist}\geq \text{max}\{f, f-t\}$.

Suppose instead that $t\geq 1$. When $d\geq t+1$, $h_t(d)$ is positive, in which case we once again have that $d\geq f-\text{dist}+1$ and $d-t\leq d-f+\text{dist}$. Since we require that $h_t(d)\leq h_{f-\textup{dist}}(d)$ for every positive integer $d$, we again see that $\text{dist}\geq f$ and $\text{dist}\geq f-t$, and
the lemma follows. \end{proof}

We can thus check C4 using Lemma \ref{lemma:501} with $t=p_{out}$, $f=f(y)$ and dist$\,=d_T(u, y)$.

Given $u\in V(T)$ and $i\in [k]_0$, let $p_{in}$, $p_{out},$ $q_{in}$, $q_{out}$ be integers with  $-n\leq p_{in}, p_{out}\leq n$ and $1\leq q_{in}, q_{out}\leq n$. A function
\begin{center}
     $f:V(T_{u_i})\rightarrow \mathbb{N}_0$ is
(($p_{in}$,  $q_{in}$), ($p_{out}$, $q_{out}$))-\textit{compatible} 
\end{center}
if conditions C1, C2, C3, and C4 hold under $f$. Let $\alpha_{bn}(T_{u, i},(p_{in},q_{in}), (p_{out},  q_{out}))$ be the maximum weight of such a function. 

\begin{lemma}
\label{lemma:502}
For any non-leaf root $r$ of $T$, 
$\alpha_{bn}(T)=\alpha_{bn}(T_{r, c_T(r)}, (-n, 1), (n, 1))$. 
\end{lemma}

\begin{proof}

By definition, $T_{r, c_T(r)}=T$. Since $g_{-n, 1}(d)=d+n$ for all nonnegative integers $d$, $g_{-n, 1}(d)\geq n$. Therefore, since $f(x)\leq e_T(x)\leq n$ for all $x\in V(T)$, condition C3 holds by C1. 
Similarly, since $g_{n, 1}(d)=0$ for all nonnegative integers $d$, C4 holds by C1 and Lemma \ref{lemma:501} with $t=p_{out}$. \end{proof}

We now show how to determine $\alpha_{bn}(T_{u, i},(p_{in},q_{in}), (p_{out}, q_{out}))$ recursively for all $O(n^5)$ choices for $(u, i)$, $p_{in},q_{in}, p_{out},$ and $q_{out}$. Recall that $T_{u, 0}$ consists of only the single vertex $u$. 

The following lemma holds for all vertices $u$ of $T$, but is particularly
useful when $u$ is a leaf.

\begin{lemma}
\label{lemma:503}

For any vertex $u$ of $T$,

\begin{align*}
    \alpha_{bn}(T_{u, 0},(p_{in},q_{in}), (p_{out},  q_{out}))=\begin{cases}
 0 &\text{ if } q_{out}> p_{out}\\
 \min\{e_T(u), g_{p_{in}, q_{in}}(0), p_{out} \} &\text{ if } q_{out}\leq p_{out}.
\end{cases}
\end{align*}
\end{lemma}

\begin{proof}
First suppose $q_{out}>p_{out}$. By (\ref{eq_pout}), either $V_{f}^{+}\cap
V(T_{u,0})=\varnothing$ and so $f(u)=0$, or $y_{out}$ is a vertex of $T_{u,0}$
such that $d_{T}(u,y_{out})>0$, which is impossible because $V(T_{u,0}%
)=\{u\}$. Therefore $f(u)=0$ and $\alpha_{bn}(T_{u, 0},(p_{in},q_{in}), (p_{out},  q_{out}))=0$. 

Now suppose that $q_{out}\leq p_{out}$. Then (as mentioned above)
$q_{out}=p_{out}$ and, by (\ref{eq_pout}), $V_{f}^{+}\cap V(T_{u,i}%
)\neq\varnothing$ and $d_{T}(u,y_{out})=0$; that is, $u=y_{out}$ and
$f(u)=f(y_{out})=q_{out}>0$. By C1, $f(u)\leq e_{T}(u),$ and by
C3, $f(u)\leq g_{p_{in},q_{in}}(0)$. Hence $\alpha_{bn}(T_{u, 0},(p_{in},q_{in}), (p_{out},  q_{out}))= \min\{e_T(u), g_{p_{in}, q_{in}}(0), p_{out}\}$. 
\end{proof}

The following technical lemma describes the key recursive step of the algorithm. See Figure \ref{fig:rooted_tree}. 

\begin{lemma}
\label{lemma504recursive}
Let $u\in V(T)$ and $i\in [k]$ be given and let $v_1, ..., v_k$ denote the children of $u$. Suppose $v_i$ has $k_i$ children.  A function $f:V(T_{u, i})\rightarrow \mathbb{N}_0$ is $((p_{in}, q_{in}),(p_{out}, q_{out}))$-compatible if and only if there exist integers 
$p_{\text {in }}^{(0)}, p_{\mathrm{out}}^{(0)}, p_{\text {in }}^{(1)}, p_{\mathrm{out}}^{(1)}, q_{\text {in }}^{(0)}, q_{\mathrm{out}}^{(0)}, q_{\text {in }}^{(1)}, q_{\mathrm{out}}^{(1)}$ with $-n~\leq p_{\mathrm{in}}^{(0)}, p_{\mathrm{out}}^{(0)}, p_{\mathrm{in}}^{(1)}, p_{\mathrm{out}}^{(1)} \leq n$ and $1 \leq q_{\mathrm{in}}^{(0)}, q_{\mathrm{out}}^{(0)}, q_{\mathrm{in}}^{(1)}, q_{\mathrm{out}}^{(1)} \leq n$ satisfying the following conditions: 

(i) the restriction of $f$ to $V(T_{u, i-1})$ is $((p_{ {in }}^{(0)}, q_{ {in }}^{(0)}),(p_{\mathrm{out}}^{(0)}, q_{\mathrm{out}}^{(0)}))$-compatible;

(ii) the restriction of $f$ to $V(T_{v_{i}, k_{i}})$ is $((p_{ {in }}^{(1)}, q_{ {in }}^{(1)}),(p_{ {out }}^{(1)}, q_{ {out }}^{(1)}))$-compatible;

\noindent and for every nonnegative integer $d$: 

(iii) $g_{p_{ {in }}^{{(0) }}, q_{ {in }}^{ (0)}}(d)=\min \{g_{p_{ {in }}, q_{ {in }}}(d),\, g_{p_{out}^{(1)}, q_{out}^{(1)}}(d+1)\}$;

(iv) $g_{p_{ {in }}^{(1)}, q_{ {in }}^{(1)}}(d)=\min \{g_{p_{ {in }}, q_{ {in }}}(d+1),\, g_{p_{ {out }}^{(0)}, q_{ {out }}^{(0)}}(d+1)\}$;

(v) $g_{p_{ {out }}, q_{ {out }}}(d) \leq \min \{g_{p_{ {out }}^{(0)}, q_{ {out }}^{(0)}}(d),\, g_{p_{ {out }}^{(1)}, q_{ {out }}^{(1)}}(d+1)\}$.

\end{lemma}

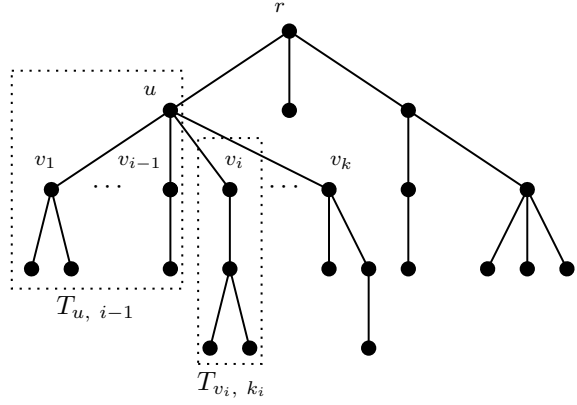
\begin{figure}[H]
\centering
\tikzset{every picture/.style={line width=0.75pt}} %set default line width to 0.75pt        
%\hspace{1 cm}
\begin{tikzpicture}[x=0.75pt,y=0.75pt,yscale=-1,xscale=1]

%Straight Lines [id:da5269158268874297] 
\draw    (380,150) -- (380,110) -- (320,70) -- (260,110) -- (260,150) ;
\draw [shift={(260,150)}, rotate = 90] [color={rgb, 255:red, 0; green, 0; blue, 0 }  ][fill={rgb, 255:red, 0; green, 0; blue, 0 }  ][line width=0.75]      (0, 0) circle [x radius= 3.35, y radius= 3.35]   ;
\draw [shift={(380,150)}, rotate = 270] [color={rgb, 255:red, 0; green, 0; blue, 0 }  ][fill={rgb, 255:red, 0; green, 0; blue, 0 }  ][line width=0.75]      (0, 0) circle [x radius= 3.35, y radius= 3.35]   ;
%Straight Lines [id:da46285484609350047] 
\draw    (380,110) -- (440,150) ;
\draw [shift={(440,150)}, rotate = 33.69] [color={rgb, 255:red, 0; green, 0; blue, 0 }  ][fill={rgb, 255:red, 0; green, 0; blue, 0 }  ][line width=0.75]      (0, 0) circle [x radius= 3.35, y radius= 3.35]   ;
\draw [shift={(380,110)}, rotate = 33.69] [color={rgb, 255:red, 0; green, 0; blue, 0 }  ][fill={rgb, 255:red, 0; green, 0; blue, 0 }  ][line width=0.75]      (0, 0) circle [x radius= 3.35, y radius= 3.35]   ;
%Straight Lines [id:da23578395719702616] 
\draw    (200,150) -- (260,110) ;
\draw [shift={(260,110)}, rotate = 0] [color={rgb, 255:red, 0; green, 0; blue, 0 }  ][fill={rgb, 255:red, 0; green, 0; blue, 0 }  ][line width=0.75]      (0, 0) circle [x radius= 3.35, y radius= 3.35]   ;
\draw [shift={(200,150)}, rotate = 326.31] [color={rgb, 255:red, 0; green, 0; blue, 0 }  ][fill={rgb, 255:red, 0; green, 0; blue, 0 }  ][line width=0.75]      (0, 0) circle [x radius= 3.35, y radius= 3.35]   ;
%Straight Lines [id:da4334559125953703] 
\draw    (260,150) -- (260,190) ;
\draw [shift={(260,190)}, rotate = 90] [color={rgb, 255:red, 0; green, 0; blue, 0 }  ][fill={rgb, 255:red, 0; green, 0; blue, 0 }  ][line width=0.75]      (0, 0) circle [x radius= 3.35, y radius= 3.35]   ;
\draw [shift={(260,150)}, rotate = 90] [color={rgb, 255:red, 0; green, 0; blue, 0 }  ][fill={rgb, 255:red, 0; green, 0; blue, 0 }  ][line width=0.75]      (0, 0) circle [x radius= 3.35, y radius= 3.35]   ;
%Straight Lines [id:da7286544894795308] 
\draw    (260,110) -- (290,150) ;
\draw [shift={(290,150)}, rotate = 53.13] [color={rgb, 255:red, 0; green, 0; blue, 0 }  ][fill={rgb, 255:red, 0; green, 0; blue, 0 }  ][line width=0.75]      (0, 0) circle [x radius= 3.35, y radius= 3.35]   ;
%Straight Lines [id:da43700645320297915] 
\draw    (190,190) -- (200,150) -- (204.42,167.69) -- (210,190) ;
\draw [shift={(210,190)}, rotate = 75.96] [color={rgb, 255:red, 0; green, 0; blue, 0 }  ][fill={rgb, 255:red, 0; green, 0; blue, 0 }  ][line width=0.75]      (0, 0) circle [x radius= 3.35, y radius= 3.35]   ;
\draw [shift={(190,190)}, rotate = 284.04] [color={rgb, 255:red, 0; green, 0; blue, 0 }  ][fill={rgb, 255:red, 0; green, 0; blue, 0 }  ][line width=0.75]      (0, 0) circle [x radius= 3.35, y radius= 3.35]   ;
%Straight Lines [id:da030658178619251908] 
\draw    (300,230) -- (290,190) -- (280,230) ;
\draw [shift={(280,230)}, rotate = 104.04] [color={rgb, 255:red, 0; green, 0; blue, 0 }  ][fill={rgb, 255:red, 0; green, 0; blue, 0 }  ][line width=0.75]      (0, 0) circle [x radius= 3.35, y radius= 3.35]   ;
\draw [shift={(300,230)}, rotate = 255.96] [color={rgb, 255:red, 0; green, 0; blue, 0 }  ][fill={rgb, 255:red, 0; green, 0; blue, 0 }  ][line width=0.75]      (0, 0) circle [x radius= 3.35, y radius= 3.35]   ;
%Straight Lines [id:da1885134551252412] 
\draw    (420,190) -- (440,150) -- (460,190) ;
\draw [shift={(460,190)}, rotate = 63.43] [color={rgb, 255:red, 0; green, 0; blue, 0 }  ][fill={rgb, 255:red, 0; green, 0; blue, 0 }  ][line width=0.75]      (0, 0) circle [x radius= 3.35, y radius= 3.35]   ;
\draw [shift={(420,190)}, rotate = 296.57] [color={rgb, 255:red, 0; green, 0; blue, 0 }  ][fill={rgb, 255:red, 0; green, 0; blue, 0 }  ][line width=0.75]      (0, 0) circle [x radius= 3.35, y radius= 3.35]   ;
%Straight Lines [id:da7607626976213073] 
\draw    (380,150) -- (380,170) -- (380,177.8) -- (380,190) ;
\draw [shift={(380,190)}, rotate = 90] [color={rgb, 255:red, 0; green, 0; blue, 0 }  ][fill={rgb, 255:red, 0; green, 0; blue, 0 }  ][line width=0.75]      (0, 0) circle [x radius= 3.35, y radius= 3.35]   ;
%Straight Lines [id:da9960120779294719] 
\draw    (290,150) -- (290,190) ;
\draw [shift={(290,190)}, rotate = 90] [color={rgb, 255:red, 0; green, 0; blue, 0 }  ][fill={rgb, 255:red, 0; green, 0; blue, 0 }  ][line width=0.75]      (0, 0) circle [x radius= 3.35, y radius= 3.35]   ;
%Shape: Polygon [id:ds5438833820363942] 
\draw  [dash pattern={on 0.84pt off 2.51pt}] (180,90) -- (266,90) -- (266,200) -- (180,200) -- (180,90) -- cycle ;
%Shape: Polygon [id:ds17791862174800555] 
\draw  [dash pattern={on 0.84pt off 2.51pt}] (306,124) -- (306,238) -- (274,238) -- (274,124) -- cycle ;
%Straight Lines [id:da6459391861225203] 
\draw    (260,110) -- (340,150) ;
\draw [shift={(340,150)}, rotate = 26.57] [color={rgb, 255:red, 0; green, 0; blue, 0 }  ][fill={rgb, 255:red, 0; green, 0; blue, 0 }  ][line width=0.75]      (0, 0) circle [x radius= 3.35, y radius= 3.35]   ;
%Straight Lines [id:da9783335669021407] 
\draw    (320,70) ;
\draw [shift={(320,70)}, rotate = 0] [color={rgb, 255:red, 0; green, 0; blue, 0 }  ][fill={rgb, 255:red, 0; green, 0; blue, 0 }  ][line width=0.75]      (0, 0) circle [x radius= 3.35, y radius= 3.35]   ;
%Straight Lines [id:da41132179323454543] 
\draw    (340,190) -- (340,150) -- (360,190) ;
\draw [shift={(360,190)}, rotate = 63.43] [color={rgb, 255:red, 0; green, 0; blue, 0 }  ][fill={rgb, 255:red, 0; green, 0; blue, 0 }  ][line width=0.75]      (0, 0) circle [x radius= 3.35, y radius= 3.35]   ;
\draw [shift={(340,190)}, rotate = 270] [color={rgb, 255:red, 0; green, 0; blue, 0 }  ][fill={rgb, 255:red, 0; green, 0; blue, 0 }  ][line width=0.75]      (0, 0) circle [x radius= 3.35, y radius= 3.35]   ;
%Straight Lines [id:da11291031099182747] 
\draw    (320,70) -- (320,110) ;
\draw [shift={(320,110)}, rotate = 90] [color={rgb, 255:red, 0; green, 0; blue, 0 }  ][fill={rgb, 255:red, 0; green, 0; blue, 0 }  ][line width=0.75]      (0, 0) circle [x radius= 3.35, y radius= 3.35]   ;
%Straight Lines [id:da2567225915134548] 
\draw    (440,150) -- (440,190) ;
\draw [shift={(440,190)}, rotate = 90] [color={rgb, 255:red, 0; green, 0; blue, 0 }  ][fill={rgb, 255:red, 0; green, 0; blue, 0 }  ][line width=0.75]      (0, 0) circle [x radius= 3.35, y radius= 3.35]   ;
%Straight Lines [id:da20253201269483334] 
\draw    (360,190) -- (360,230) ;
\draw [shift={(360,230)}, rotate = 90] [color={rgb, 255:red, 0; green, 0; blue, 0 }  ][fill={rgb, 255:red, 0; green, 0; blue, 0 }  ][line width=0.75]      (0, 0) circle [x radius= 3.35, y radius= 3.35]   ;

% Text Node
\draw (245,96.4) node [anchor=north west][inner sep=0.75pt]  [font=\footnotesize]  {${\textstyle u}$};
% Text Node
\draw (232,131.4) node [anchor=north west][inner sep=0.75pt]  [font=\footnotesize]  {$v_{i-1}$};
% Text Node
\draw (190,131.4) node [anchor=north west][inner sep=0.75pt]  [font=\footnotesize]  {$v_{1}$};
% Text Node
\draw (286,131.4) node [anchor=north west][inner sep=0.75pt]  [font=\footnotesize]  {$v_{i}$};
% Text Node
\draw (311,54.4) node [anchor=north west][inner sep=0.75pt]  [font=\footnotesize]  {$r$};
% Text Node
\draw (201,202.4) node [anchor=north west][inner sep=0.75pt]    {$T_{u,\ i-1}{}$};
% Text Node
\draw (272,240.4) node [anchor=north west][inner sep=0.75pt]    {$T_{v_{i} ,\ k_{i}}{}$};
% Text Node
\draw (339,131.4) node [anchor=north west][inner sep=0.75pt]  [font=\footnotesize]  {$v_{k}$};
% Text Node
\draw (219,144.4) node [anchor=north west][inner sep=0.75pt]    {$\cdots $};
% Text Node
\draw (308,144.4) node [anchor=north west][inner sep=0.75pt]    {$\cdots $};

\end{tikzpicture}

\caption{A rooted tree and a vertex $u$ with $k$ children, with $T_{u, i-1}$ and $T_{v_i, k_i}$ as shown.}
\label{fig:rooted_tree}
\end{figure}

\begin{proof}
Suppose first that $f$ is $((p_{in}, q_{in}),(p_{out}, q_{out}))$-compatible. Define $p_{in}^{(0)}, q_{in}^{(0)},p_{out}^{(0)}, q_{out}^{(0)}$ with respect to $T_{u, i-1}$ as in (2) and (3). Define
$p_{in}^{(1)}, q_{in}^{(1)},p_{out}^{(1)}, q_{out}^{(1)}$ with respect to $T_{v_i, k_i}$ similarly. 
Given a vertex  $w\in V(T\backslash T_{u, i})$, the distance from $w$ to $v_i$ is one more than the distance from $w$ to $u$. Thus, $(v)$ holds by C4. 

Let $y_{in}^{(0)}$ be an arbitrarily chosen vertex for which $f(y_{in}^{(0)})=q_{in}^{(0)}$ and $f(y_{in})-d_T(u, y_{in})=p_{in}^{(0)} $. If $y_{in}^{(0)}$ lies in $V(T\backslash T_{u, i})$, the bounds imposed by $p_{in}$ on $f(x)$ for all $x\in V(T_{u, i})$ match those imposed by $p_{in}^{(0)}$ on the vertices of $T_{u, i-1}$, hence $(p_{in}^{(0)}, q_{in}^{(0)})=(p_{in}, q_{in})$. Otherwise, $y_{in}^{(0)}$ lies in $V(T_{v_i, k_i})$. Since $g_{p, q}(d+1)=g_{p-1, q}(d)$ when $q\geq 1$, $g_{p_{in}^{(0)}, q_{in}^{(0)}}(d) = g_{p_{in}^{(1)}, q_{in}^{(1)}}(d+1)$.
Therefore $g_{p_{ {in }}^{{(0) }}, q_{ {in }}^{ (0)}}(d)=\min \{g_{p_{ {in }}, q_{ {in }}}(d), g_{p_{ {out }}^{(1)}, q_{ {out }}^{(1)}}(d+1)\}$ for every nonnegative integer $d$. Similarly, the minimum in $(iv)$ equals $g_{p_{{in}}^{(1)}, q_{{in}}^{(1)}}(d)$. Therefore $(iii)$ and $(iv)$ hold.

Since $f$ is $((p_{ {in }}, q_{ {in }}),(p_{ {out }}, q_{ {out }}))$-compatible, $f(x)\leq  g_{p_{ {in }}, q_{ {in }}}({d}_{T}(u, x))$ for all ${x\in V(T_{u, i-1})}$. Furthermore, by C2, ${f(x)\leq d_T(x, v_i)-f(v_i)\leq  g_{p_{ {out }}^{(1)}, q_{ {out }}^{(1)}}({d}_{T}(v_{i}, x))=g_{p_{ {out }}^{(1)}, q_{ {out }}^{(1)}}({d}_{T}(u, x)+1)}$, and so $(i)$ holds for $g_{p_{\mathrm{in}}^{(0)}, q_{\mathrm{in}}^{(0)}}(d)$ as in $(iii)$. Similarly, $(ii)$ holds for $g_{p_{\mathrm{in}}^{(1)}, q_{\mathrm{in}}^{(1)}}(d)$ as in $(iv).$ 

Conversely, suppose there exist integers 
$$
p_{\text {in }}^{(0)}, p_{\mathrm{out}}^{(0)}, p_{\text {in }}^{(1)}, p_{\mathrm{out}}^{(1)}, q_{\text {in }}^{(0)}, q_{\mathrm{out}}^{(0)}, q_{\text {in }}^{(1)}, q_{\mathrm{out}}^{(1)}
$$
with $-n \leq p_{\mathrm{in}}^{(0)}, p_{\mathrm{out}}^{(0)}, p_{\mathrm{in}}^{(1)}, p_{\mathrm{out}}^{(1)} \leq n$ and $1 \leq q_{\mathrm{in}}^{(0)}, q_{\mathrm{out}}^{(0)}, q_{\mathrm{in}}^{(1)}, q_{\mathrm{out}}^{(1)} \leq n$ satisfying $(i)-(v)$. 

By $(i)$ and $(ii)$, $f$ satisfies C1; that is, $f(x)\leq e_T(x)$ for all $x\in V(T_{u, i-1})$ and $x\in V(T_{v_i, k_i})$. 

Conditions $(i)$ and $(ii)$ further imply that C2 holds for any $x, y\in V(T_{u, i-1})$ and for any $x, y\in V(T_{v_i, k_i})$. Thus, to prove $f$ satisfies C2, it remains to show that $f(x)+f(y)\leq d_T(x, y)$ for all $x\in V(T_{u, i-1})$ and $y\in V(T_{v_i, k_i})$ such that $f(x), f(y)>0$. 

By $(i)$, the restriction  of $f$ to $V(T_{u, i-1})$ satisfies C3; that is, $f(x)\leq g_{p_{in}^{(0)}, q_{in}^{(0)}}(d_T(u, x))$. By $(iii)$,
\begin{align*}
    f(x)&\leq g_{p_{out}^{(1)}, q_{out}^{(1)}}(d_T(u, x)+1).
\end{align*}
Thus, by $(ii)$ and C4, 
\begin{align*}
    f(x)&\leq g_{f(y)-d_T(u, y), f(y)}(d_T(u, x)+1)
    \\
    &=g_{f(y)-d_T(v_i, y), f(y)}(d_T(v_i, x))
    \\
    &\leq \text{min}\{0, d_T(x, y)-f(y)\}.
\end{align*}
Since $x$ is broadcasting, $f(x)>0$ and so $f(x)\leq d_T(x, y)-f(y)$, satisfying C2.

By $(iii)$, $g_{p_{ {in }}^{(0)}, q_{ {in }}^{(0)}}(d) \leq g_{p_{\mathrm{in}}, q_{\mathrm{in}}}(d)$, and by $(iv)$,  $g_{p_{\mathrm{in}}^{(1)}, q_{\mathrm{in}}^{(1)}}(d) \leq g_{p_{\mathrm{in}}, q_{\mathrm{in}}}(d+1)$ for every nonnegative integer $d$. Thus $f$ satisfies (C3). 

Finally, consider $y\in V(T_{u, i})$ such that $f(y)>0$. If $y\in V(T_{u, i-1})$, then by $(v)$,
$$
g_{p_{out}, q_{out}}(d) \leq  g_{p_{out}^{(0)}, q_{out}^{(0)}}(d)
$$
for all positive integers $d$. Again by $(i)$, we may apply C4 to the restriction of $f$ to $T_{u, i-1}$, obtaining
$$
g_{p_{out}, q_{out}}(d) \leq  g_{f(y)-d_T(u, y), f(y)}(d).
$$
Similarly, if $f(y)>0$ for some $y\in V(T_{v_i, k_i})$, then
\begin{align*}
    g_{p_{out}, q_{out}}(d) &\leq  g_{p_{out}^{(1)}, q_{out}^{(1)}}(d+1)
    \\
    & \leq  g_{f(y)-d_T(v_i, y), f(y)}(d+1).
    \\
    & =  g_{f(y)-d_T(u, y), f(y)}(d)
\end{align*}
for every positive integer $d$. Therefore $f$ satisfies $C1-C4$. 
\end{proof}

Observe that for any $i\in [k]$, $V(T_{u, i-1})\cup V(T_{v_i, k_i})=V(T_{u, i})$. The following is a consequence of Lemma \ref{lemma504recursive}. 

\begin{corollary}
\label{corollary:505}
Suppose $u\in V(T)$ has $k$ children. Let $i\in [k]$ be given and let $v_i$ have $k_i$ children. Then
\begin{equation}
\begin{aligned}
        &\alpha_{bn}(T_{u, i}, (p_{in}, q_{in}),(p_{out}, q_{out})) \\
    &=\max \{\alpha_{bn}(T_{u, i-1}, (p_{in}^{(0)}, q_{in}^{(0)}),(p_{out}^{(0)}, q_{out}^{(0)})) 
    + \alpha_{bn}(T_{v_i, k_i}, (p_{in}^{(1)}, q_{in}^{(1)}),(p_{out}^{(1)}, q_{out}^{(1)}))\}
\end{aligned}
\label{eq_sum_max}
\end{equation}
over all choices of $p_{\text {in }}^{(0)}, p_{\mathrm{out}}^{(0)}, p_{\text {in }}^{(1)}, p_{\mathrm{out}}^{(1)}, q_{\text {in }}^{(0)}, q_{\mathrm{out}}^{(0)}, q_{\text {in }}^{(1)}, q_{\mathrm{out}}^{(1)}$
with $-n \leq p_{\mathrm{in}}^{(0)}, p_{\mathrm{out}}^{(0)}, p_{\mathrm{in}}^{(1)}, p_{\mathrm{out}}^{(1)} \leq n$ and 
\\
$1 \leq q_{\mathrm{in}}^{(0)}, q_{\mathrm{out}}^{(0)}, q_{\mathrm{in}}^{(1)}, q_{\mathrm{out}}^{(1)} \leq n$ satisfying conditions $(iii), (iv),$ and $(v)$ of Lemma \ref{lemma504recursive}. 
\end{corollary}

\begin{sloppypar}
\begin{proof}
Let $p_{\text {in }}^{(0)}, p_{\mathrm{out}}^{(0)}, p_{\text {in }}^{(1)}, p_{\mathrm{out}}^{(1)}, q_{\text {in }}^{(0)}, q_{\mathrm{out}}^{(0)}, q_{\text {in }}^{(1)}, q_{\mathrm{out}}^{(1)}$ be integers satisfying the conditions of the lemma such that the maximum in (4) is attained. 
Let $f$ be a broadcast such that the restrictions of $f$ to $T_{u, i-1}$ and $T_{v_i, k_i}$ correspond to $\alpha_{bn}(T_{u, i-1}, (p_{in}^{(0)}, q_{in}^{(0)}),(p_{out}^{(0)}, q_{out}^{(0)})) $ and $\alpha_{bn}(T_{v_i, k_i}, (p_{in}^{(1)}, q_{in}^{(1)}),(p_{out}^{(1)}, q_{out}^{(1)}))$, respectively. By definition, the restrictions are 
$((p_{ {in }}^{(0)}, q_{ {in }}^{(0)}),(p_{\mathrm{out}}^{(0)}, q_{\mathrm{out}}^{(0)}))$-compatible and {$((p_{ {in }}^{(1)}, q_{ {in }}^{(1)}),(p_{ {out }}^{(1)}, q_{ {out }}^{(1)}))$-compatible}, hence conditions $(i)$ and $(ii)$ in the statement of Lemma \ref{lemma504recursive} are satisfied. Thus, by Lemma \ref{lemma504recursive}, the restriction of $f$ to $T_{u, i}$ is boundary independent.\end{proof} \end{sloppypar} 

We may now prove the main theorem of this section. 

\begin{theorem}
\label{theorem:on9}
Given a tree $T$ with $n \geq 3$ vertices, its boundary independent broadcast number $\alpha_{bn}(T)$ can be determined in $O(n^9)$ time.
\end{theorem}

%% < link (1)
\begin{proof} %does it matter if r, or the u's are leaves?
Select a non-leaf vertex $r$ as the root and process the vertices of $T$ in order of nonincreasing distance from $r$. Recall that by (1), there are $O(n)$ choices for $(u, i)$.
Suppose $u$ and $i$ are given. For each of the $O(n^4)$ choices for $p_{in}, q_{in}, p_{out},$ and $q_{out}$, 
the value $\alpha_{bn}(T_{u, i}, (p_{in}, q_{in}), (p_{out},q_{out}))$ can be determined in $O\left(n^{4}\right)$ time as follows.
If $i=0$, apply Lemma \ref{lemma:503} to find $\alpha_{bn}(T_{u, i}, (p_{in}, q_{in}), (p_{out},q_{out}))$ in linear time. %? yes because ecc of a vertex? 
Otherwise:
\begin{itemize}
    \item For each of the $O(n^4)$ possible choices for the four integers $p_{\text {out }}^{(0)}, p_{\text {out }}^{(1)}, q_{\text {out }}^{(0)}$, and $q_{\text {out }}^{(1)}$, check condition $(v)$ from Lemma \ref{lemma504recursive} in constant time.
    
    \item Given $p_{in}, q_{in}, p_{\text {out}}^{(0)}, p_{ {out }}^{(1)}, q_{\text {out}}^{(0)}$, and $q_{\text {out }}^{(1)},$ apply conditions $(iii)$ and $(iv)$ of Lemma \ref{lemma504recursive} to determine $p_{\text {in }}^{(0)}, p_{\text {in }}^{(1)}, q_{\text {in }}^{(0)}$, and $q_{\text {in }}^{(1)}$ in constant time. 
    
    \item Finally, add $\alpha_{bn}(T_{u, i-1},(p_{ {in }}^{(0)}, q_{ {in }}^{(0)}),(p_{ {out }}^{(0)}, q_{ {out }}^{(0)}))$ and $\alpha_{bn}(T_{v_{i}, k_{i}},(p_{ {in }}^{(1)}, q_{ {in }}^{(1)}),(p_{ {out }}^{(1)}, q_{ {out }}^{(1)}))$ and apply Corollary \ref{corollary:505}. 
\end{itemize}
Since there are $O(n)$ choices for $(u, i)$, it follows that $\alpha_{bn}(T_{r, c_T(r)}, (-n, 1),(n, 1))$ can be determined in $O(n^9)$ time. By Lemma \ref{lemma:502}, this value equals the boundary independent broadcast number $\alpha_{bn}(T)$.
\end{proof}

We illustrate the algorithm with a simple example. Consider the tree $T$ below, in which one of the two non-leaf vertices is arbitrarily defined to be the root. The remaining vertices are labelled $u_1$ through $u_4$ such that $d_T(u_i, r)\leq d_T(u_j, r)$ whenever $i>j$.

\begin{figure}[H]
    \centering

\tikzset{every picture/.style={line width=0.75pt}} %set default line width to 0.75pt        

\begin{tikzpicture}[x=0.75pt,y=0.75pt,yscale=-1,xscale=1]

%Straight Lines [id:da5229565884285214] 
\draw    (398,118) -- (418,78) ;
\draw [shift={(418,78)}, rotate = 296.57] [color={rgb, 255:red, 0; green, 0; blue, 0 }  ][fill={rgb, 255:red, 0; green, 0; blue, 0 }  ][line width=0.75]      (0, 0) circle [x radius= 3.35, y radius= 3.35]   ;
\draw [shift={(398,118)}, rotate = 296.57] [color={rgb, 255:red, 0; green, 0; blue, 0 }  ][fill={rgb, 255:red, 0; green, 0; blue, 0 }  ][line width=0.75]      (0, 0) circle [x radius= 3.35, y radius= 3.35]   ;
%Straight Lines [id:da4948973055861092] 
\draw    (438,118) -- (418,78) -- (378,118) ;
\draw [shift={(378,118)}, rotate = 135] [color={rgb, 255:red, 0; green, 0; blue, 0 }  ][fill={rgb, 255:red, 0; green, 0; blue, 0 }  ][line width=0.75]      (0, 0) circle [x radius= 3.35, y radius= 3.35]   ;
\draw [shift={(438,118)}, rotate = 243.43] [color={rgb, 255:red, 0; green, 0; blue, 0 }  ][fill={rgb, 255:red, 0; green, 0; blue, 0 }  ][line width=0.75]      (0, 0) circle [x radius= 3.35, y radius= 3.35]   ;
%Straight Lines [id:da6933369476044762] 
\draw    (438,118) -- (438,158) ;
\draw [shift={(438,158)}, rotate = 90] [color={rgb, 255:red, 0; green, 0; blue, 0 }  ][fill={rgb, 255:red, 0; green, 0; blue, 0 }  ][line width=0.75]      (0, 0) circle [x radius= 3.35, y radius= 3.35]   ;
%Straight Lines [id:da41318174154349463] 
\draw    (195,118) -- (235,118) ;
\draw [shift={(235,118)}, rotate = 0] [color={rgb, 255:red, 0; green, 0; blue, 0 }  ][fill={rgb, 255:red, 0; green, 0; blue, 0 }  ][line width=0.75]      (0, 0) circle [x radius= 3.35, y radius= 3.35]   ;
\draw [shift={(195,118)}, rotate = 0] [color={rgb, 255:red, 0; green, 0; blue, 0 }  ][fill={rgb, 255:red, 0; green, 0; blue, 0 }  ][line width=0.75]      (0, 0) circle [x radius= 3.35, y radius= 3.35]   ;
%Straight Lines [id:da4446282684704983] 
\draw    (275,118) -- (315,118) ;
\draw [shift={(315,118)}, rotate = 0] [color={rgb, 255:red, 0; green, 0; blue, 0 }  ][fill={rgb, 255:red, 0; green, 0; blue, 0 }  ][line width=0.75]      (0, 0) circle [x radius= 3.35, y radius= 3.35]   ;
\draw [shift={(275,118)}, rotate = 0] [color={rgb, 255:red, 0; green, 0; blue, 0 }  ][fill={rgb, 255:red, 0; green, 0; blue, 0 }  ][line width=0.75]      (0, 0) circle [x radius= 3.35, y radius= 3.35]   ;
%Straight Lines [id:da4165069730478994] 
\draw    (235,118) -- (275,118) ;
%Straight Lines [id:da08414171360117795] 
\draw    (235,78) -- (235,118) ;
\draw [shift={(235,118)}, rotate = 90] [color={rgb, 255:red, 0; green, 0; blue, 0 }  ][fill={rgb, 255:red, 0; green, 0; blue, 0 }  ][line width=0.75]      (0, 0) circle [x radius= 3.35, y radius= 3.35]   ;
\draw [shift={(235,78)}, rotate = 90] [color={rgb, 255:red, 0; green, 0; blue, 0 }  ][fill={rgb, 255:red, 0; green, 0; blue, 0 }  ][line width=0.75]      (0, 0) circle [x radius= 3.35, y radius= 3.35]   ;

% Text Node
\draw (404,64.4) node [anchor=north west][inner sep=0.75pt]  [font=\footnotesize]  {$r$};
% Text Node
\draw (431,164.4) node [anchor=north west][inner sep=0.75pt]  [font=\footnotesize]  {$u_{1}$};
% Text Node
\draw (371,123.4) node [anchor=north west][inner sep=0.75pt]  [font=\footnotesize]  {$u_{2}$};
% Text Node
\draw (393,123.4) node [anchor=north west][inner sep=0.75pt]  [font=\footnotesize]  {$u_{3}$};
% Text Node
\draw (421,123.4) node [anchor=north west][inner sep=0.75pt]  [font=\footnotesize]  {$u_{4}$};
% Text Node
\draw (250,130.4) node [anchor=north west][inner sep=0.75pt]    {$T$};

\end{tikzpicture}

    \caption{A tree $T$ and its rooted structure.}
    \label{fig:exampletree}
\end{figure}
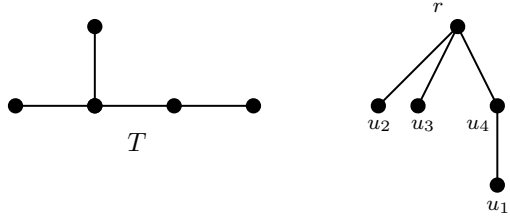

Note that $e_T(u_i)=3$ for $i\in \{1, 2, 3\}$.  
By Lemma \ref{lemma:503}, for each $1\leq i\leq 3$, we have that 

\begin{align*}%%
    \alpha_{bn}(T_{u_i, 0},(p_{in},q_{in}), (p_{out},  q_{out}))=\begin{cases}
 &0 \text{ if } q_{out}> p_{out}\\
 &\min\{3, g_{p_{in}, q_{in}}(0), p_{out} \} \text{ if } q_{out}\leq p_{out}
\end{cases}
\end{align*}

\begin{sloppypar}
\noindent over all choices for $p_{in}, q_{in}, p_{out},$ and $q_{out}$ such that $-5\leq p_{in}, p_{out}\leq 5$ and ${1\leq q_{in}, q_{out}\leq 5}$. 
Observe that since $T_{u_i,0}$ consists of only a single vertex, $u_i$ is broadcasting if and only if $q_{out}=p_{out}$. Then $\alpha_{bn}(T_{u_i, 0},(p_{in},q_{in}), (p_{out},  q_{out}))=0$ if $p_{in}\geq 0$, and  $\alpha_{bn}(T_{u_i, 0},(p_{in},q_{in}), (p_{out},  q_{out}))\in \{1, 2, 3\}$ if $p_{in}<0$. 

Otherwise, $q_{out}>p_{out}$, in which case $\alpha_{bn}(T_{u_i, 0},(p_{in},q_{in}), (p_{out},  q_{out}))=0$ as expected. 
\end{sloppypar}

Recall that by Lemma \ref{lemma:leaf_hears_f2}, given a tree $T$ and an $\alpha_{bn}$-broadcast $f$ on $T$, no leaf of $T$ hears $f$ from a non-leaf.

In particular, no vertex adjacent to a leaf belongs to $V_f^+$, hence the restriction of an $\alpha_{bn}$-broadcast $f$ to $T_{u_4, 0}=\{u_4\}$ has weight 0. By Lemma \ref{corollary:505}, we have that 
 \begin{align*}
     & \alpha_{bn}(T_{u_4, 1},(p_{in},q_{in}), (p_{out},  q_{out}))
     \\
    &=\max \{\alpha_{bn}(T_{u_4, 0}, (p_{in}^{(0)}, q_{in}^{(0)}),(p_{out}^{(0)}, q_{out}^{(0)})) 
    + \alpha_{bn}(T_{u_1, 0}, (p_{in}^{(1)}, q_{in}^{(1)}),(p_{out}^{(1)}, q_{out}^{(1)}))\}
    %\label{}
\\
    &=\max \{\alpha_{bn}(T_{u_1, 0}, (p_{in}^{(1)}, q_{in}^{(1)}),(p_{out}^{(1)}, q_{out}^{(1)}))\}
    %\label{}
\end{align*}

\noindent \begin{sloppypar}
 over all choices of $p_{\text {in }}^{(0)}, p_{\mathrm{out}}^{(0)}, p_{\text {in }}^{(1)}, p_{\mathrm{out}}^{(1)}, q_{\text {in }}^{(0)}, q_{\mathrm{out}}^{(0)}, q_{\text {in }}^{(1)}, q_{\mathrm{out}}^{(1)}$
with $-5 \leq p_{\mathrm{in}}^{(0)}, p_{\mathrm{out}}^{(0)}, p_{\mathrm{in}}^{(1)}, p_{\mathrm{out}}^{(1)} \leq 5$ and $1 \leq q_{\mathrm{in}}^{(0)}, q_{\mathrm{out}}^{(0)}, q_{\mathrm{in}}^{(1)}, q_{\mathrm{out}}^{(1)} \leq 5$ satisfying conditions $(iii), (iv),$ and $(v)$ of Lemma \ref{lemma504recursive}.  \end{sloppypar}

We may calculate $\alpha_{bn}(T_{r, 1},(p_{in},q_{in}), (p_{out},  q_{out}))$ as $\max \{\alpha_{bn}(T_{u_2, 0}, (p_{in}^{(1)}, q_{in}^{(1)}),(p_{out}^{(1)}, q_{out}^{(1)}))\}$ similarly. Thus, to determine $\alpha_{bn}(T_{r, 2},(p_{in},q_{in}), (p_{out}, q_{out})$, we need only consider the subtrees $T_{u_2, 0}$ and $T_{u_3, 0}$.  

Finally, by Lemmas \ref{lemma:502} and \ref{corollary:505}, we may calculate $\alpha_{bn}(T)$ as 
\begin{equation}
\begin{aligned}
    &\alpha_{bn}(T_{r, 3}, (-n, 1),(n, 1))\\
    &=\max \{\alpha_{bn}(T_{r, 2}, (p_{in}^{(0)}, q_{in}^{(0)}),(p_{out}^{(0)}, q_{out}^{(0)})) 
    + \alpha_{bn}(T_{u_4, 1}, (p_{in}^{(1)}, q_{in}^{(1)}),(p_{out}^{(1)}, q_{out}^{(1)}))\}
\end{aligned}
    \label{eq_sum_max}
\end{equation}
\noindent again over all choices of $p_{\text {in }}^{(0)}, p_{\mathrm{out}}^{(0)}, p_{\text {in }}^{(1)}, p_{\mathrm{out}}^{(1)}, q_{\text {in }}^{(0)}, q_{\mathrm{out}}^{(0)}, q_{\text {in }}^{(1)}, q_{\mathrm{out}}^{(1)}$
with $-5 \leq p_{\mathrm{in}}^{(0)}, p_{\mathrm{out}}^{(0)}, p_{\mathrm{in}}^{(1)}, p_{\mathrm{out}}^{(1)} \leq 5$ and $1 \leq q_{\mathrm{in}}^{(0)}, q_{\mathrm{out}}^{(0)}, q_{\mathrm{in}}^{(1)}, q_{\mathrm{out}}^{(1)} \leq 5$ satisfying 
%conditions $(iii), (iv),$ and $(v)$ of 
Lemma \ref{lemma504recursive}.

\begin{figure}[H]
    \centering

\tikzset{every picture/.style={line width=0.75pt}} %set default line width to 0.75pt        

\begin{tikzpicture}[x=0.75pt,y=0.75pt,yscale=-1,xscale=1]

%Straight Lines [id:da5397065810499175] 
\draw    (210,100) -- (230,60) ;
\draw [shift={(230,60)}, rotate = 296.57] [color={rgb, 255:red, 0; green, 0; blue, 0 }  ][fill={rgb, 255:red, 0; green, 0; blue, 0 }  ][line width=0.75]      (0, 0) circle [x radius= 3.35, y radius= 3.35]   ;
\draw [shift={(210,100)}, rotate = 296.57] [color={rgb, 255:red, 0; green, 0; blue, 0 }  ][fill={rgb, 255:red, 0; green, 0; blue, 0 }  ][line width=0.75]      (0, 0) circle [x radius= 3.35, y radius= 3.35]   ;
%Straight Lines [id:da9064255089096172] 
\draw    (250,100) -- (230,60) -- (190,100) ;
\draw [shift={(190,100)}, rotate = 135] [color={rgb, 255:red, 0; green, 0; blue, 0 }  ][fill={rgb, 255:red, 0; green, 0; blue, 0 }  ][line width=0.75]      (0, 0) circle [x radius= 3.35, y radius= 3.35]   ;
\draw [shift={(250,100)}, rotate = 243.43] [color={rgb, 255:red, 0; green, 0; blue, 0 }  ][fill={rgb, 255:red, 0; green, 0; blue, 0 }  ][line width=0.75]      (0, 0) circle [x radius= 3.35, y radius= 3.35]   ;
%Straight Lines [id:da27527763151940143] 
\draw    (250,100) -- (250,140) ;
\draw [shift={(250,140)}, rotate = 90] [color={rgb, 255:red, 0; green, 0; blue, 0 }  ][fill={rgb, 255:red, 0; green, 0; blue, 0 }  ][line width=0.75]      (0, 0) circle [x radius= 3.35, y radius= 3.35]   ;
%Shape: Polygon [id:ds6702313450642228] 
\draw  [dash pattern={on 0.84pt off 2.51pt}] (180,40) -- (250,40) -- (221.29,104.82) -- (220,110) -- (180,110) -- cycle ;
%Shape: Polygon [id:ds183201999889129] 
\draw  [dash pattern={on 0.84pt off 2.51pt}] (238,90) -- (262,90) -- (262,150) -- (238,150) -- cycle ;

% Text Node
\draw (216,46.4) node [anchor=north west][inner sep=0.75pt]  [font=\footnotesize]  {$r$};
% Text Node
\draw (182,113.4) node [anchor=north west][inner sep=0.75pt]    {$T_{r,\ 2} \ $};
% Text Node
\draw (231,152.4) node [anchor=north west][inner sep=0.75pt]    {$T_{u_{4} ,\ 1} \ $};

\end{tikzpicture}

    \caption{The partition of $T$ considered in \ref{eq_sum_max}.}
    \label{fig:my_label}
\end{figure}
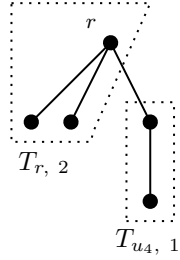

\begin{sloppypar}
Repetitive calculations show that this maximum is achieved when $\alpha_{bn}(T_{r, 2}, (p_{in}^{(0)}, q_{in}^{(0)}),(p_{out}^{(0)}, q_{out}^{(0)})) =2
$ and $ \alpha_{bn}(T_{u_4, 1}, (p_{in}^{(1)}, q_{in}^{(1)}),(p_{out}^{(1)}, q_{out}^{(1)}))=2$. Since $f(r)=f(u_4)=0$ under an $\alpha_{bn}$-broadcast $f$, we find that maximum boundary independence is achieved when $f(u_1)=2$ and $f(u_2)=f(u_3)=1$. Hence $\alpha_{bn}(T)=4.$
\end{sloppypar}

\section{Open problems}

It is unknown whether there exists a graph $G$ for which $\frac{i_h(G)}{i_{bn}(G)}=\frac{5}{4}$. As there exist graphs with $i_{bn}(G)=5$ and $i_h(G)=6$ (Figure \ref{fig:65counterexample}), a sharp upper bound  for $\frac{i_h(G)}{i_{bn}(G)}$ must lie between $\frac{6}{5}$ and $\frac{5}{4}$. 

\begin{problem}
Improve the bound $\frac{i_h(G)}{i_{bn}(G)}\leq \frac{5}{4}$, or show it is best possible for an infinite family of graphs.
\end{problem}

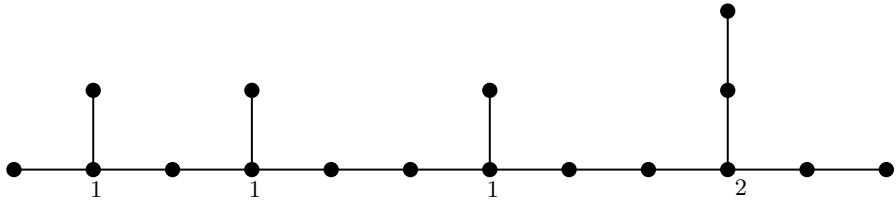
\begin{figure}[H]
    \centering

\tikzset{every picture/.style={line width=0.75pt}} %set default line width to 0.75pt        

\begin{tikzpicture}[x=0.75pt,y=0.75pt,yscale=-1,xscale=1]

%Straight Lines [id:da11729502651800416] 
\draw    (60,150) -- (100,150) ;
\draw [shift={(60,150)}, rotate = 0] [color={rgb, 255:red, 0; green, 0; blue, 0 }  ][fill={rgb, 255:red, 0; green, 0; blue, 0 }  ][line width=0.75]      (0, 0) circle [x radius= 3.35, y radius= 3.35]   ;
%Straight Lines [id:da8241341459782758] 
\draw    (100,150) -- (140,150) ;
\draw [shift={(100,150)}, rotate = 0] [color={rgb, 255:red, 0; green, 0; blue, 0 }  ][fill={rgb, 255:red, 0; green, 0; blue, 0 }  ][line width=0.75]      (0, 0) circle [x radius= 3.35, y radius= 3.35]   ;
%Straight Lines [id:da6663875127924141] 
\draw    (260,150) -- (300,150) ;
\draw [shift={(260,150)}, rotate = 0] [color={rgb, 255:red, 0; green, 0; blue, 0 }  ][fill={rgb, 255:red, 0; green, 0; blue, 0 }  ][line width=0.75]      (0, 0) circle [x radius= 3.35, y radius= 3.35]   ;
%Straight Lines [id:da7373877101877837] 
\draw    (220,150) -- (260,150) ;
\draw [shift={(220,150)}, rotate = 0] [color={rgb, 255:red, 0; green, 0; blue, 0 }  ][fill={rgb, 255:red, 0; green, 0; blue, 0 }  ][line width=0.75]      (0, 0) circle [x radius= 3.35, y radius= 3.35]   ;
%Straight Lines [id:da24921236590764284] 
\draw    (180,150) -- (220,150) ;
\draw [shift={(180,150)}, rotate = 0] [color={rgb, 255:red, 0; green, 0; blue, 0 }  ][fill={rgb, 255:red, 0; green, 0; blue, 0 }  ][line width=0.75]      (0, 0) circle [x radius= 3.35, y radius= 3.35]   ;
%Straight Lines [id:da45278214335505873] 
\draw    (140,150) -- (180,150) ;
\draw [shift={(140,150)}, rotate = 0] [color={rgb, 255:red, 0; green, 0; blue, 0 }  ][fill={rgb, 255:red, 0; green, 0; blue, 0 }  ][line width=0.75]      (0, 0) circle [x radius= 3.35, y radius= 3.35]   ;
%Straight Lines [id:da7502214152733591] 
\draw    (300,150) -- (340,150) ;
\draw [shift={(300,150)}, rotate = 0] [color={rgb, 255:red, 0; green, 0; blue, 0 }  ][fill={rgb, 255:red, 0; green, 0; blue, 0 }  ][line width=0.75]      (0, 0) circle [x radius= 3.35, y radius= 3.35]   ;
%Straight Lines [id:da4210616389070554] 
\draw    (340,150) -- (380,150) ;
\draw [shift={(380,150)}, rotate = 0] [color={rgb, 255:red, 0; green, 0; blue, 0 }  ][fill={rgb, 255:red, 0; green, 0; blue, 0 }  ][line width=0.75]      (0, 0) circle [x radius= 3.35, y radius= 3.35]   ;
\draw [shift={(340,150)}, rotate = 0] [color={rgb, 255:red, 0; green, 0; blue, 0 }  ][fill={rgb, 255:red, 0; green, 0; blue, 0 }  ][line width=0.75]      (0, 0) circle [x radius= 3.35, y radius= 3.35]   ;
%Straight Lines [id:da7782142185943635] 
\draw    (300,110) -- (300,150) ;
\draw [shift={(300,110)}, rotate = 90] [color={rgb, 255:red, 0; green, 0; blue, 0 }  ][fill={rgb, 255:red, 0; green, 0; blue, 0 }  ][line width=0.75]      (0, 0) circle [x radius= 3.35, y radius= 3.35]   ;
%Straight Lines [id:da4310157655871778] 
\draw    (100,110) -- (100,150) ;
\draw [shift={(100,110)}, rotate = 90] [color={rgb, 255:red, 0; green, 0; blue, 0 }  ][fill={rgb, 255:red, 0; green, 0; blue, 0 }  ][line width=0.75]      (0, 0) circle [x radius= 3.35, y radius= 3.35]   ;
%Straight Lines [id:da3257888395102415] 
\draw    (180,110) -- (180,150) ;
\draw [shift={(180,110)}, rotate = 90] [color={rgb, 255:red, 0; green, 0; blue, 0 }  ][fill={rgb, 255:red, 0; green, 0; blue, 0 }  ][line width=0.75]      (0, 0) circle [x radius= 3.35, y radius= 3.35]   ;
%Straight Lines [id:da9094871056277447] 
\draw    (420,150) -- (380,150) ;
\draw [shift={(420,150)}, rotate = 180] [color={rgb, 255:red, 0; green, 0; blue, 0 }  ][fill={rgb, 255:red, 0; green, 0; blue, 0 }  ][line width=0.75]      (0, 0) circle [x radius= 3.35, y radius= 3.35]   ;
%Straight Lines [id:da7854851749856411] 
\draw    (460,150) -- (420,150) ;
\draw [shift={(460,150)}, rotate = 180] [color={rgb, 255:red, 0; green, 0; blue, 0 }  ][fill={rgb, 255:red, 0; green, 0; blue, 0 }  ][line width=0.75]      (0, 0) circle [x radius= 3.35, y radius= 3.35]   ;
%Straight Lines [id:da9454686616246282] 
\draw    (420,110) -- (420,150) ;
\draw [shift={(420,110)}, rotate = 90] [color={rgb, 255:red, 0; green, 0; blue, 0 }  ][fill={rgb, 255:red, 0; green, 0; blue, 0 }  ][line width=0.75]      (0, 0) circle [x radius= 3.35, y radius= 3.35]   ;
%Straight Lines [id:da6541052517364212] 
\draw    (420,70) -- (420,110) ;
\draw [shift={(420,70)}, rotate = 90] [color={rgb, 255:red, 0; green, 0; blue, 0 }  ][fill={rgb, 255:red, 0; green, 0; blue, 0 }  ][line width=0.75]      (0, 0) circle [x radius= 3.35, y radius= 3.35]   ;
%Straight Lines [id:da7798554118923793] 
\draw    (500,150) -- (460,150) ;
\draw [shift={(500,150)}, rotate = 180] [color={rgb, 255:red, 0; green, 0; blue, 0 }  ][fill={rgb, 255:red, 0; green, 0; blue, 0 }  ][line width=0.75]      (0, 0) circle [x radius= 3.35, y radius= 3.35]   ;

% Text Node
\draw (422,153.4) node [anchor=north west][inner sep=0.75pt]  [font=\small]  {$2$};
% Text Node
\draw (97,154.4) node [anchor=north west][inner sep=0.75pt]  [font=\small]  {$1$};
% Text Node
\draw (177,154.4) node [anchor=north west][inner sep=0.75pt]  [font=\small]  {$1$};
% Text Node
\draw (297,154.4) node [anchor=north west][inner sep=0.75pt]  [font=\small]  {$1$};

\end{tikzpicture}

    \caption{A tree with $i_{bn}(T)=5$ and $i_h(T)=\text{rad}(T)=6$.}
    \label{fig:65counterexample}
\end{figure}

Recall that $\gamma_b(G)$ denotes the minimum cost of a dominating broadcast on $G$. 
Neilson showed in \cite{neilsonphd} that $\gamma_b(G)\leq i_{bn}(G)$ for all graphs $G$.
Since $i_{bn}(G)\leq i_h(G)\leq \text{rad}(G)$, we have that $i_{bn}(G)=i_h(G)$ for all $G$ such that $\gamma_b(G)=\text{rad}(G)$, known as \textit{radial} graphs. Radial trees were characterized by Herke and Mynhardt  in \cite{herke}.

Let $f$ be an $i_{bn}$-broadcast on a graph $G$.
If $|V_f^+|=1$, or if $V_f^+=V_f^1$, then $i_h(G)=i_{bn}(G)$ and $f$ is an $i_h$-broadcast on $G$. In particular, $i_{bn}(G)=i_h(G)$ for 
paths and cycles. Equality also holds for graphs $G$ such that $\text{rad}(G)\leq 5$, for if $f$ is an $i_{bn}$-broadcast with $\sigma(f)\leq 4$, then $G-U_f^E$ consists of either a single component (in which case $\sigma(f)=\text{rad}(G)$), or two components, each with two broadcasts of strength 1. It would be of interest to further classify graphs for which these parameters are equal. 

\begin{question}
    For which graphs $G$ is $i_{bn}(G)=i_{h}(G)$?
\end{question}

Trees with exactly one branch vertex are known as \textit{generalized spiders}.
The generalized spider $S=S(n_1, n_2, ..., n_k)$ consists of a branch vertex $b$ of degree $k$, and $k\geq 3$ paths or `legs' $L_1, ..., L_k$, each with one endpoint at $b$, such that $\ell(L_i)=n_i$ for all $1\leq i\leq k$.
It was shown in \cite{thesis} that generalized spiders satisfy the equality in Question 10.2.

\begin{theorem} \textup{\cite{thesis}}
Let $S=S(n_1, n_2, ..., n_k)$ be a generalized spider.
Then $i_{bn}(S)=i_h(S)$.
\end{theorem}

However, a closed formula to determine the exact values of these parameters remains unknown.

\begin{problem}
Determine $i_{bn}(S)$ for all generalized spiders $S$.
\end{problem}

The problem of determining $\gamma(G)$ for a given graph $G$ is known to be NP-complete \cite{garey}. In \cite{heggernes}, however, Heggernes and Lokshtanov showed that the minimum broadcast domination problem is solvable in polynomial time for all graphs. 

\begin{problem}
Study the complexity of determining $i_{bn}(G)$ and $i_h(G)$ for trees or other graph classes. 
\end{problem}

In \cite{neilsonphd}, Neilson determined that $\alpha_{bn}(T)\leq n-|B_T|+|R_T|$ for all trees with at least one branch vertex, which we improved to $\alpha_{bn}(C)\leq n-|B_C|+\alpha(C[R_C])$ for caterpillars.  It is unknown whether this result can be extended to all trees. 

\begin{question}
For any tree $T$ of order $n$ and at least one branch vertex, is it true that $\alpha_{bn}(T)\leq n-|B_T|+\alpha(C[R_T])$? 
\end{question}

In particular, we found that equality holds for caterpillars $C$ with no trunks, as well as caterpillars with no branches of degree 3 and no two adjacent trunks, when $|V(C)|\geq 3$.

\begin{problem}
    Characterize trees $T$ such that $\alpha_{bn}(T)= n-|B_T|+\alpha(C[R_T])$.
\end{problem}

\begin{problem}
Fine a closed formula to determine $\alpha_{bn}(C)$ exactly for caterpillars or other classes of trees. 
\end{problem}

In Section 9, we showed that the maximum bn-independence problem is solvable in polynomial time on all trees. For further research, more efficient algorithms may be possible when considering the additional constraints of boundary independence compared to hearing independence. 
 
\begin{problem}
Improve the running time in Theorem \ref{theorem:on9}, or show it is best possible. 
\end{problem}

\end{document}